\documentclass{article}
\usepackage{graphicx} 
\usepackage{amsmath}
\usepackage{amssymb}
\usepackage{amsthm}
\usepackage{mathrsfs}
\usepackage{dsfont}
\usepackage{bm}
\usepackage[left=2.00cm, right=2.00cm, top=2.00cm, bottom=2.00cm]{geometry} 
\usepackage{todonotes}
\usepackage{parskip}
\usepackage{enumitem}
\usepackage{authblk}
\usepackage{aliascnt}
\usepackage[
  colorlinks=true,
  linkcolor=red,
  citecolor=blue,
  urlcolor=blue
]{hyperref}
\usepackage{cleveref}

\usepackage{subcaption}
\usepackage{tikz}
\usetikzlibrary{decorations.markings,arrows.meta}

\numberwithin{equation}{section}

\newcommand{\newaliastheorem}[3]{%
  \newaliascnt{#1}{thm}%
  \newtheorem{#1}[#1]{#2}%
  \aliascntresetthe{#1}%
  \crefname{#1}{#2}{#3}%
}
\newtheorem{thm}{Theorem}[section]
\crefname{thm}{Theorem}{Theorems}
\newaliastheorem{theorem}{Theorem}{Theorems}
\newaliastheorem{prop}{Proposition}{Propositions}
\newaliastheorem{proposition}{Proposition}{Propositions}
\newaliastheorem{lem}{Lemma}{Lemmas}
\newaliastheorem{lemma}{Lemma}{Lemmas}
\newaliastheorem{cor}{Corollary}{Corollaries}
\newaliastheorem{coro}{Corollary}{Corollaries}
\theoremstyle{definition}
\newaliastheorem{defn}{Definition}{Definitions}
\newaliastheorem{assumption}{Assumption}{Assumptions}
\newaliastheorem{remark}{Remark}{Remarks}

\crefname{section}{Section}{Sections}
\crefname{figure}{Fig.}{Figs.}

\def\Ab{{\mathbb{A}}}

\def\Cb{{\mathbb{C}}}
\def\Db{{\mathbb{D}}}
\def\Eb{{\mathbb{E}}}

\def\Hb{{\mathbb{H}}}

\def\Nb{{\mathbb{N}}}

\def\Qb{{\mathbb{Q}}}
\def\Rb{{\mathbb{R}}}

\def\Tb{{\mathbb{T}}}

\def\Zb{{\mathbb{Z}}}
\def\Ac{{\mathcal{A}}}

\def\Dc{{\mathcal{D}}}
\def\Ec{{\mathcal{E}}}

\def\Hc{{\mathcal{H}}}

\def\Lc{{\mathcal{L}}}

\def\Qc{{\mathcal{Q}}}

\def\Sc{{\mathcal{S}}}

\def\Abf{{\mathbf{A}}}

\def\Dbf{{\mathbf{D}}}

\def\Hbf{{\mathbf{H}}}

\def\Lbf{{\mathbf{L}}}

\def\Tbf{{\mathbf{T}}}

\def\bfr{{\mathfrak{b}}}

\def\sfr{{\mathfrak{s}}}
\def\tfr{{\mathfrak{t}}}

\def\cbf{{\mathbf{c}}}

\def\kbf{{\mathbf{k}}}

\def\mbf{{\mathbf{m}}}
\def\nbf{{\mathbf{n}}}

\def\vbf{{\mathbf{v}}}

\def\xbf{{\mathbf{x}}}

\def\zbf{{\mathbf{z}}}
\def\alphabf{{\boldsymbol{\alpha}}}

\def\zetabf{{\boldsymbol{\zeta}}}
\def\etabf{{\boldsymbol{\eta}}}
\def\thetabf{{\boldsymbol{\theta}}}

\def\mubf{{\boldsymbol{\mu}}}

\def\xibf{{\boldsymbol{\xi}}}

\def\varphibf{{\boldsymbol{\varphi}}}

\def\Pibf{{\boldsymbol{\Pi}}}

\def\d{{\mathrm{d}}} 

\def\i{{\mathrm{i}}}
\def\id{{\mathrm{id}}}
\def\Re{\operatorname{Re}}
\def\Im{\operatorname{Im}}

\def\dist{\operatorname{dist}}

\def\XXint#1#2#3{{\setbox0=\hbox{$#1{#2#3}{\int}$ }
	\vcenter{\hbox{$#2#3$ }}\kern-.6\wd0}}

\def\div{\operatorname{div}}

\def\la{\langle} 
\def\ra{\rangle}
\def\fr{\forall\,}

\def\zb{{\bar{z}}}
\def\pa{\partial}

\newcommand{\wt}[1]{\widetilde{#1}}

\newcommand{\ol}[1]{\overline{#1}}

\def\N{{\mathrm{N}}}
\def\D{{\mathrm{D}}}

\def\MD{{\mathrm{M_D}}}
\def\paN{{\wt{\pa}_\N}}
\def\paD{{\wt{\pa}_\D}}

\def\b{\bfr}

\def\s{\sfr}
\def\t{\tfr}
\def\h{{\mathrm{h}}}
\def\L{{\mathrm{L}}}

\def\reg{{\mathrm{reg}}}

\def\bt{{\wt{\varphibf}}}

\def\TR{\Tb_R}

\def\DT{\Dc'(\Tb,\Tb_R)}
\def\DTh{\Dc'(\Tb^+,\Tb_R)}
\def\Ssg{\Sc}
\def\v{{\rm v}}

\def\emptyset{\varnothing}
\def\epsilon{\varepsilon}

\title{Segal's semigroup, Virasoro structure, and Ward identities in Compactified Imaginary Liouville Theory}
\author[1]{Yang Xiao}
\author[2]{Yuxiao Xie}
\affil[1]{Aix Marseille Université, CNRS, Institut de Mathématiques de Marseille, Marseille, France}
\affil[2]{Université Paris-Saclay, CNRS, Laboratoire de mathématiques d’Orsay, Orsay, France}
\date{}

\begin{document}

\maketitle

\begin{abstract}
Compactified Imaginary Liouville Theory (CILT for short) is a probabilistically constructed Conformal Field Theory (CFT), expected to have connections to minimal models, Coulomb gas integrals, and loop models. In this paper, we study the quantization of Segal's semigroup given by CILT. Each element $\mathbf T_f$ in the quantized semigroup corresponds to a conformal annulus: the unit disk $\mathbb D$ with $f(\mathbb D^\circ)$ removed, where $f:\mathbb D\to f(\mathbb D)\subset \mathbb D^\circ$ is a contracting biholomorphism. This semigroup encodes the local conformal structure of the CFT, in particular the Virasoro symmetry with central charge $<1$ for CILT. We first relate the quantized Segal's semigroup to Segal's amplitudes of annuli and recover the Virasoro generators $\mathbf L_n,\widetilde{\mathbf L}_n$ ($n\geq0$) of CILT as generators of this semigroup. Compared to the real Liouville CFT, the main difficulty here is the absence of positivity, which we address through estimates for the imaginary Gaussian Multiplicative Chaos and non-self-adjoint operators. Using Segal's gluing axiom as in the real case, we differentiate the amplitude and use conformal covariance to derive the Ward identities in terms of all Virasoro operators $\mathbf L_n,\widetilde{\mathbf L}_n$ ($n\in\mathbb Z$). We also outline the parallel theory for Boundary CILT (BCILT), concerning Segal's semigroup of semiannuli. These results provide essential tools for the conformal bootstrap program of (B)CILT, boundary BPZ equations, boundary structure constants, and the global construction of conformal blocks.
\end{abstract}

\tableofcontents

\section{Introduction}
The Compactified Imaginary Liouville Theory (CILT for short) is a 2-dimensional Conformal Field Theory (CFT) defined on a Riemann surface $\Sigma$ with a conformal metric $g$, whose correlation functions can be formally written as a path integral of the form
\[\la F(\phi)\ra_{\Sigma,g}:=\int F(\phi)\,e^{-S_\L(\phi)}\,\D\phi,\]
where $F(\phi)$ is a suitable test functional and $\D\phi$ is the formal ``Lebesgue'' measure on the space of maps $\phi$ from $\Sigma$ to the circle $\mathbb R/2\pi R\mathbb Z$ of radius $R>0$. If $\Sigma$ has boundary, we require that $\phi$ satisfies the Neumann boundary condition at $\pa\Sigma$. Here the \textit{imaginary Liouville action} $S_{\rm L}$ is formally given by (without the last two terms if $\partial\Sigma=\emptyset$)
\begin{equation*}
    S_\L(\phi)=
\frac{1}{4\pi}\int_{\Sigma}|d\phi|_g^2\,\d v_g
+\frac{\i Q}{4\pi}\int_\Sigma K_g\phi\,\d v_g
+\mu\int_\Sigma e^{\i \beta\phi}\,\d v_g+\frac{\i Q}{2\pi}\int_{\pa\Sigma}k_g\phi\,\d\ell_g+\int_{\pa\Sigma}\mu_\pa\,e^{\i\frac{\beta}{2}\phi}\,\d\ell_g,
\end{equation*}
where $Q,\beta\in\mathbb R$ with $\beta^2\in(0,2)$, $\mu\in\mathbb C$, $\mu_\partial:\pa\Sigma\to\Cb$ is a piecewise constant function, and $K_g$, $v_g$, $k_g$, $\ell_g$ denote respectively the scalar curvature, volume, geodesic curvature, and length with respect to $g$.
When $\mu=0$, $\mu_\partial=0$, this becomes the Coulomb gas formalism (see \cite{dotsenko1984conformal,dotsenko1985four,Francesco1987,Nienhuis1987}). Throughout this paper, we impose the relation
\[Q=\frac\beta2-\frac2\beta.\]
The well-definedness of the Liouville action requires that $\beta R, QR\in\Zb$ if $\partial\Sigma=\emptyset$ and $\beta R, QR\in2\Zb$ if $\partial\Sigma\not=\emptyset$. These conditions force $\beta^2\in\Qb$, and given such $\beta$, one can always choose $R$ accordingly. Then this path integral defines a CFT with central charge $c_{\mathrm L}=1-6Q^2$.
The mathematical construction of this CFT was performed in \cite{CILT} for the case without boundary, and later the boundary case, called Boundary CILT (BCILT for short), was studied in \cite{BCILT}. In physics, there have been many related research directions as in \cite{Schomerus2003,Zamolodchikov2005,Kostov2006,Harlow2011,Ribault2015,Usciati2025}, aiming either to capture the ``imaginary'' version of the real Liouville theory (proposed in \cite{Polyakov1981}), or to formulate a CFT candidate for the 2-dimensional loop models, see for instance \cite{Francesco1987,Nienhuis1987,Kondev1996,Ikhlef2016,Nivesvivat2024}. On the mathematical side, there are also related works such as \cite{Ang21,Baverez2024,Chatterjee2026}.  

In BCILT, we consider correlation functions of vertex operators of the form
\begin{equation*}
   \prod_{j=1}^\s V_{\alpha_j}(z_j)\prod_{j=1}^\t V_{\eta_j}(x_j),\ \ \ z_j\in\Sigma^\circ,\ x_j\in\partial\Sigma,\ \ \ \s,\t\in\mathbb N
\end{equation*}
with $V_{\alpha_j}(z_j):=e^{\i\alpha_j\phi(z_j)}$ and $V_{\eta_j}(x_j):=e^{\i\frac{\eta_j}2\phi(x_j)}$, called \textit{electric operators}, in the interior (or \textit{bulk}) $\Sigma^\circ$ and on the boundary $\partial \Sigma$, respectively. 
In addition to these, there is another type of operators known as \textit{magnetic operators}, defined only in the bulk. Specifically, for ${\bf z}=(z_1,\ldots,z_\s)\in(\Sigma^\circ)^\s$ and \textit{magnetic charges} ${\bf m}=(m_1,\ldots,m_\s)\in\Zb^\s$, we define formally
\begin{equation*}
    \left\langle F(\phi)\prod_{j=1}^\s V_{(0,m_j)}(z_j)\right\rangle_{\Sigma,g}:=\int F(\phi+2\pi R\,I_{\zbf,\mbf})\,e^{-S_\L(\phi+2\pi R\,I_{\zbf,\mbf})}\,\D\phi,
\end{equation*}
where $I_{\zbf,\mbf}(z):=\int_{x_0}^z\nu_{\zbf,\mbf}$ (with fixed basepoint $x_0\in\Sigma$) and $\nu_{\zbf,\mbf}$ is a chosen $1$-form that equals $\frac{m_j}{2\pi} \, d\theta$ in a local coordinate $z=z_j+r e^{\i\theta}$ near $z_j$ for $j=1,\cdots,\s$. If $\pa\Sigma=\emptyset$, such a $1$-form exists if and only if $\sum_{j=1}^{\s} m_j=0$.
We remark that here both the field $\phi$ and $2\pi R\,I_{\zbf,\mbf}$ take values in the circle $\Tb_R:=\Rb/2\pi R\Zb$, so they are multi-valued as functions to $\Rb$. In order to make sense of the curvature terms $\frac{\i Q}{4\pi}\int_\Sigma K_g\phi\,\d v_g+\frac{\i Q}{2\pi}\int_{\pa\Sigma}k_g\phi\,\d\ell_g$ in the Liouville functional, a special regularization procedure is needed. This is done by fixing a ``branch cut'' of the surface, called \textit{separating family} in \cite{BCILT} (the notion of \textit{defect graph} in \cite{CILT} is a special case thereof), and adding appropriate regularization terms at the cuts. It can be shown that the result does not depend on the choice of the separating family, up to an integral multiple of $2\pi R$.
Finally, one can combine the electric and magnetic operators to form the \textit{electromagnetic operators}, by taking a limit $z_j'\to z_j$ in the direction specified by a nonzero tangent vector $v_j$ at $z_j$ and defining $V_{(\alpha_j,m_j)}(v_j):=\lim_{z_j'\to z_j}V_{\alpha_j}(z_j')V_{(0,m_j)}(z_j)$.
The operators $V_{(\alpha_j,m_j)}(v_j)$ and $V_{\eta_j}(x_j)$ give the primary states of BCILT with conformal weights $\Delta_{\alpha_j,m_j}$ and $\Delta_{\eta_j}$, where
\[\Delta_{\alpha,m}:=\frac\alpha2\left(\frac\alpha2-Q\right)+\frac{m^2R^2}4,\qquad\Delta_\eta:=\frac{\eta}{2}\left(\frac\eta2-Q\right).\]
For more background on BCILT, including its rigorous construction, we refer to \cite{CILT,BCILT}.

Given a CFT, a central mathematical problem is to establish the conformal bootstrap rigorously. The BPZ proposal \cite{Belavin1984}, together with Segal's gluing axioms \cite{Segal1988}, provides a guideline for recovering correlation functions from amplitudes and conformal blocks. In the case of the real Liouville CFT, ever since its probabilistic construction on general surfaces in \cite{Guillarmou2019}, there has been significant progress in this direction. In particular, the bootstrap factorization formula was proved for the $4$-point function on the sphere in \cite{Guillarmou2024} and for general surfaces in \cite{Guillarmou2021}. For the boundary Liouville CFT, the conformal bootstrap was also solved in \cite{boundary_bootstrap}. A natural next step is to carry out the corresponding bootstrap program for BCILT. Let us illustrate this with the simplest example of the $4$-point function on the sphere. First recall that the conformal covariance reduces the $3$-point function on the Riemann sphere $(\hat{\mathbb C},g_0)$ with $g_0=\frac{|\d z|^2}{\max(|z|,1)^4}$ to a constant, known as the \textit{structure constant}. In CILT, it is given by (assuming the neutrality condition $2Q-\sum_j\alpha_j\in\beta\Nb$)
\begin{equation*}
\begin{split}
    &\left\langle V_{(\alpha_1,m_1)}(0,e_1)V_{(\alpha_2,m_2)}(1,e_1)V_{(\alpha_3,m_3)}(\infty,e_1) \right\rangle_{\hat{\mathbb C},g_0}^2\\
    &= C_{g_0}^2 4\pi^2 R^2 \mu^{\frac2\beta(2Q-\sum_{j=1}^3 \alpha_j)} C_{\beta}^{\mathrm{DOZZ}}(\alpha_1+m_1R,\alpha_2+m_2R,\alpha_3+m_3R)\,C_{\beta}^{\mathrm{DOZZ}}(\alpha_1-m_1R,\alpha_2-m_2R,\alpha_3-m_3R),
\end{split}
\end{equation*}
where $C_{g_0}$ is a universal constant, $C_{\beta}^{\mathrm{DOZZ}}(\cdot,\cdot,\cdot)$ is the \textit{imaginary DOZZ constant}, and $e_1=\partial_x$ in the coordinate $z=x+\i y$ in $\mathbb C$. This formula appeared in \cite{Zamolodchikov2005,Kostov2006}, and it was proved for CILT based on \cite{Neretin2024}, see \cite[Section 7]{CILT}. Now for the $4$-point function on the sphere, we expect a factorization formula of the form (again assuming the neutrality condition $2Q-\sum_j\alpha_j\in\beta\Nb$)
\begin{equation}\label{eq_bootstrap}
\left\langle V_{(\alpha_1,m_1)}(0,e_1)V_{(\alpha_2,m_2)}(z,e_1) V_{(\alpha_3,m_3)}(1,e_1)V_{(\alpha_4,m_4)}(\infty,e_1) \right\rangle_{\hat{\mathbb C},g_0}
    =\sum_{n=0}^{\frac{2Q-\sum_j\alpha_j}\beta}\rho(\alphabf,\mbf,n)\mathcal F_{\alphabf,\mbf,n}(z)\widetilde{\mathcal F}_{\alphabf,\mbf,n}(\bar z),
\end{equation}
where $z\in\hat\Cb\setminus\{0,1,\infty\}$,
\begin{equation*}
\begin{split}
     \rho(\alphabf,\mbf,n):={}&\left\langle V_{(\alpha_1,m_1)}(0,e_1)V_{(\alpha_2,m_2)}(1,e_1)V_{(2Q-\alpha_1-\alpha_2-n\beta,-m_1-m_2)}(\infty,e_1) \right\rangle_{\hat{\mathbb C},g_0}\\
     &\times\left\langle V_{(\alpha_1+\alpha_2+n\beta,m_1+m_2)}(0,e_1) V_{(\alpha_3,m_3)}(1,e_1)V_{(\alpha_4,m_4)}(\infty,e_1)\right\rangle_{\hat{\mathbb C},g_0},
\end{split}
\end{equation*}
and $\mathcal F_{\alphabf,\mbf,n}(z),\widetilde{\mathcal F}_{\alphabf,\mbf,n}(\bar z)$ are the holomorphic and anti-holomorphic Virasoro conformal blocks. The existence of spin means that CILT is \textit{nondiagonal}, and thus the two conformal blocks $\mathcal F_{\alphabf,\mbf,n},\widetilde{\mathcal F}_{\alphabf,\mbf,n}$ are distinct; this is also seen in the Virasoro actions in Section \ref{sec_Virasoro}. The formula \eqref{eq_bootstrap} is only formally written since both the structure constants and the conformal blocks can have poles, and the existence of these poles is a consequence of the complicated algebraic structure, in particular the logarithmic Jordan partners and the degenerate Shapovalov forms. Understanding these poles would be the last step towards the solution of the $4$-point function on the sphere. In a companion work, the first author, together with C. Guillarmou and R. Rhodes, will handle the $4$-point function assuming $\beta^2\notin\mathbb Q$ or $(\alpha_1+\alpha_2-Q)\beta\not\in \mathbb Q$, which would save us from the algebraic problems encountered in the proof of the bootstrap formula.

In this paper, we provide the prerequisite theory for the conformal bootstrap, namely the quantization of Segal's semigroup of annuli given by (B)CILT. We prove differentiability of the quantized semigroup, identify its infinitesimal generators and use this to derive the Ward identities for (B)CILT amplitudes. Analogously to the real counterpart \cite{Segal_semigroup}, these results provide the main tools for the construction of the conformal block and the proof of its Ward identities in the case of (B)CILT. Besides, this paper is also a part of the proof of the BPZ equations in (B)CILT, which will be done in a future study of higher equations of motion as in the real case \cite{Baverez2023} (see also \cite{Cercle2025}). The boundary BPZ equations will play an important role in the study of the boundary structure constants proposed in BCILT \cite[Section 8]{BCILT}. There has been a very recent paper \cite{Cercle2026} on the BPZ equations and Ward identities (for correlation functions only) of the non-compactified imaginary Liouville theory, which also covers those of CILT, at least if there are only electric operators.

We now give an overview of our main results. First we review the definition of Segal amplitudes in (B)CILT. For simplicity, consider the case of CILT, i.e., assume that there is no Neumann boundary. According to Segal's axiomatization of CFT \cite{Segal1988}, we associate to each Riemannian surface with parametrized boundary an amplitude operator between tensor products of a Hilbert space $\mathcal H$, called the \textit{state space} of the CFT. In CILT, the Hilbert space $\mathcal H$ is the space of $L^2$ functionals on the space of $\Rb/2\pi R\Zb$-valued boundary fields $\wt{\varphi}^k$ on a boundary circle, having the form
\begin{equation*}
    \widetilde{\varphi}^k(\theta) = c + \varphi(\theta) + kR\theta,
\end{equation*}
where $c\in\mathbb R/2\pi R\mathbb Z,k\in\mathbb Z$, and $\varphi$ is the mean-zero Gaussian Free Field (GFF)
\begin{equation*}
    \varphi(\theta)=\sum_{n\not=0}\varphi_n e^{\i n\theta},
\end{equation*}
where $\varphi_n=\frac{x_n+\i y_n}{2\sqrt{n}}$ and $\varphi_{-n}=\overline{\varphi}_n$ for $n>0$, and $x_n,y_n$ are i.i.d.\ standard Gaussians. We equip the space of $(c,k,\varphi)$ with the product of the Lebesgue measure on $\mathbb R/2\pi R\mathbb Z$, the counting measure on $\mathbb Z$, and the Gaussian probability measure on the component $\varphi$, i.e., 
\begin{equation}\label{eq_defn_phik}
\d\widetilde{\varphi}^k:=\d c\otimes\left(\sum_{k\in\mathbb Z}\delta_k\right)\otimes\bigotimes_{n\geq1}\frac{1}{2\pi}e^{-\frac{x^2_n+y_n^2}{2}}\,\d x_n\d y_n.
\end{equation}
See \cref{Sec_background} for details. Then an $L^2$ functional of $\bt^\kbf=(\wt{\varphi}_1^{k_1},\ldots,\wt{\varphi}_\b^{k_\b})$ lives in the tensor product $\mathcal H^{\otimes \b}$. 
Now let $\Sigma$ be a Riemann surface with conformal metric $g$ and boundary parametrizations $\zetabf=(\zeta_1,\cdots,\zeta_\b)$, where each $\zeta_j:\Tb=\{e^{\i\theta}:\theta\in\Rb\}\to\pa_j\Sigma$ is a real-analytic parametrization of a boundary circle $\pa_j\Sigma$. 
Here the orientation of each parametrization in $\zetabf$ is either incoming or outgoing, and we denote by $\b^+$ (resp.\ $\b^-$) the number of outgoing (resp.\ incoming) boundary circles, with $\b=\b^++\b^-$. 
In this paper, the metric $g$ will be assumed to be  \textit{admissible} in the sense that $\zeta_j^*g=\frac{|\mathrm{d}z|^2}{|z|^2}$ near each boundary circle. 
Let $\zbf=(z_1,\cdots,z_\s)$ be distinct marked points in $\Sigma^\circ$, $\vbf=(v_1,\cdots,v_\s)$ where each $v_j$ is a nonzero tangent vector at $z_j$, $\alphabf=(\alpha_1,\ldots,\alpha_\s)$ electric charges, and $\mbf=(m_1,\ldots,m_\s)$ magnetic charges. Then the CILT \textit{amplitude} of $(\Sigma,g,\vbf,\alphabf,\mbf,\zetabf)$ is given by the formal path integral
\[\Ac_{\Sigma,g,\vbf,\alphabf,\mbf,\zetabf}(F,\bt^\kbf):=
\int F(\phi)\prod_{j=1}^\s V_{(\alpha_j,m_j)}(v_j)\,e^{-S_\L(\phi)}\,\D\phi,\]
where the integral is over maps $\phi:\Sigma\setminus\{z_1,\ldots,z_\s\}\to\Rb/2\pi R\Zb$ satisfying $\zetabf^*\phi|_{\pa\Sigma}=\bt^\kbf=(\wt{\varphi}_1^{k_1},\ldots,\wt{\varphi}_\b^{k_\b})$ and having winding number $m_j\in\Zb$ around each $z_j$. Here $\kbf=(k_1,\cdots,k_{\b})$ are thus the winding numbers of $\phi$ on the Dirichlet boundary $\pa\Sigma$. For the rigorous definition of this path integral, we refer to \cite[Section 8]{CILT}. One can show that, under suitable conditions, the functional $\Ac_{\Sigma,g,\vbf,\alphabf,\mbf,\zetabf}(F,\cdot)$ is $L^2$ and can thus be viewed as the kernel of a Hilbert--Schmidt operator
\begin{equation*}
    \Ac_{\Sigma,g,\vbf,\alphabf,\mbf,\zetabf}(F):\mathcal H^{\otimes\b^-}\to\mathcal H^{\otimes\b^+}.
\end{equation*}
In the case $F\equiv1$, we shall simply write $\Ac_{\Sigma,g,\vbf,\alphabf,\mbf,\zetabf}$, or even $\mathcal A_{\Sigma,g,\zetabf}$ if there are no marked points. 
These amplitudes are diffeomorphism-invariant and satisfy conformal covariance properties under a conformal change of the metric $g$. Moreover, they satisfy the following gluing property. If $(\Sigma^\ell,g^\ell,\vbf^\ell,\alphabf^\ell,\mbf^\ell,\zetabf^\ell)$ ($\ell=1,2$) are two such surfaces with $\zeta_{j_1}^1$ incoming and $\zeta_{j_2}^2$ outgoing, then we can glue them into a new surface $(\Sigma,g,\vbf,\alphabf,\mbf,\zetabf)$ by setting $\zeta^1_{j_1}(e^{\i\theta})=\zeta^2_{j_2}(e^{\i\theta})$ for $\theta\in\Rb$. Then we have
\begin{equation*}
    \Ac_{\Sigma,g,\vbf,\alphabf,\mbf,\zetabf}=C\Ac_{\Sigma^1,g^1,\vbf^1,\alphabf^1,\mbf^1,\zetabf^1}\circ_{j_1,j_2}\Ac_{\Sigma^2,g^2,\vbf^2,\alphabf^2,\mbf^2,\zetabf^2},
\end{equation*}
where $C$ is a universal constant and $\circ_{j_1,j_2}$ denotes the composition (or trace) of the $j_1$-th tensor component of $\Ac_{\Sigma^1,g^1,\vbf^1,\alphabf^1,\mbf^1,\zetabf^1}$ with the $j_2$-th tensor component of $\Ac_{\Sigma^2,g^2,\vbf^2,\alphabf^2,\mbf^2,\zetabf^2}$. For the proofs of these claims, we refer again to \cite{CILT,BCILT}.

Now we introduce Segal's semigroup of annuli, appearing first in \cite[Section 2]{Segal1988} and independently in \cite{Neretin1987,Neretin1990}. The real analog of the following quantization for the real Liouville CFT was done in \cite{Segal_semigroup} and later \cite[Section 7.3]{boundary_bootstrap} for the boundary case. 
Let $\Db=\{|z|\leq1\}$ be the closed unit disk, and consider the semigroup $\Sc$ of holomorphic embeddings $f:\Db\to\Db^\circ$ with $f(0) = 0$ such that it extends holomorphically to a neighborhood of the boundary $\pa\Db=\mathbb T$. For $f\in\Sc$, consider the annulus $\mathbb A_{f}:=\mathbb D\setminus f(\mathbb D^\circ)$ with boundary parametrizations $\zetabf_f:=(\id_\Tb,f|_\Tb)$. Then for a functional $F(\wt{\varphi}^k)$ in an appropriate space (see \eqref{eq_def_E_space}), we define
\begin{equation*}
(\Tbf_fF)(\wt{\varphi}^k) := |f'(0)|^{\frac{R^2k^2-Q^2}2} \,\Eb_{\wt{\varphi}^k}\left[F\left((P\wt{\varphi}^k+X_\Db)\circ f\big|_\Tb + \i Q\log\tfrac{|f'|}{|f|}\big|_\Tb\right)\,e^{-\mu\int_{\Ab_f}|x|^{\beta Q}M_\beta(P\wt{\varphi}^k+X_\Db,\d x)}\right],
\end{equation*}
where $P\wt{\varphi}^k$ is the harmonic extension of $\widetilde{\varphi}^k$ to $\mathbb D$ (with prescribed singularity at $0$) and $X_{\mathbb D}$ is the Dirichlet GFF on $\mathbb D$. The conditional expectation $\Eb_{\wt{\varphi}^k}$ averages over the Dirichlet GFF $X_{\mathbb D}$. Here $M_{\beta}$ is the imaginary Gaussian Multiplicative Chaos (GMC, see \cite{Lacoin2015}) defined as follows. For a general field $\phi_g$ on a Riemannian surface $(\Sigma,g)$, we define (see \cref{Sec_background} for details)
\begin{equation*}
    M_{\beta}(\phi_g,\d v_g):=\lim_{\epsilon\to 0} \epsilon^{-\frac{\beta^2}{2}}e^{\i\beta \phi_{g,\epsilon}(x)}\d v_g,
\end{equation*}
where $\phi_{g,\epsilon}$ is an $\epsilon$-regularization with respect to $g$. 
We prove that $\Tbf_f$ extends to a bounded operator $\Hc\to\Hc$, and that this coincides (up to a constant) with the amplitude $\Ac_{\Ab_f,g_f,\zetabf_f}$ of the annulus $(\Ab_f,\zetabf_f)$ equipped with an admissible metric $g_f$ (see below). The homomorphism $f\in\Sc\mapsto\Tbf_f$ provides a quantization of Segal's semigroup $\Sc$ into operators acting on $\Hc$.

The operator $\Tbf_f$ can be used to describe the local conformal symmetry of CILT. In particular, differentiating $\Tbf_f$ gives the Virasoro actions of CILT. An important special case is that the generator of the semigroup $t\in\Rb_{>0}\mapsto\mathbf T_{z\mapsto e^{-t}z}$ is the \textit{Hamiltonian} of CILT, given by (using the decomposition in \eqref{eq_defn_phik})
\begin{equation}\label{eq_def_H}
    \mathbf H:=-\frac12\partial_c^2+\frac{R^2k^2-Q^2}2+2\sum_{n\geq1}(\mathbf A^*_n\mathbf A_n+\widetilde{\mathbf A}_n^*\widetilde{\mathbf A}_n)+\mu\, e^{\i\beta c}\int_0^{2\pi} e^{\i\beta Rk\theta}\,e^{\i\beta\varphi(\theta)}\,\d\theta,
\end{equation}
where $\mathbf A_n:=\tfrac\i2\sqrt n(\partial_{x_n}-\i\partial_{y_n})$, $\widetilde{\mathbf A}_n:=\tfrac\i2\sqrt n(\partial_{x_n}+\i\partial_{y_n})$, and $\int_0^{2\pi} e^{\i\beta Rk\theta}\,e^{\i\beta\varphi(\theta)}\,\d\theta$ is a boundary GMC (see \cref{Sec_Virasoro_Liouville}). Note that $\mathbf H$ is \textit{not} self-adjoint, unlike the real case. 
In general, we can consider the flow $t\in\Rb_{\geq0}\mapsto f_t\in\mathcal S$ generated by a Markovian holomorphic vector field $\v=v(z)\partial_z=-\sum_{n\geq0}v_n z^{n+1}\partial_z$ on $\mathbb D$, i.e., we assume that $v(0)=0$ and $\Re(\zb v(z))<0$ for $z\in\Tb$, and we solve the ordinary differential equation $\partial_t f_t(z)=v(f_t(z))$ for $z\in\Db$ with $f_0=\id_\Db$. Then the generator of the semigroup $t\in\Rb_{>0}\mapsto\mathbf T_{f_t}$, denoted $\mathbf H_\v$, is given by
\begin{equation}\label{intro_H_v}
    \mathbf H_\v:=\sum_{n\geq0}v_n\,\Lbf_n+\sum_{n\geq0}\ol{v_n}\,\wt{\Lbf}_n,
\end{equation}
where $\mathbf L_n,\widetilde{\mathbf L}_n$ are generators of the holomorphic and anti-holomorphic Virasoro algebras defined by CILT. They enjoy explicit formulas \eqref{eq_L} in terms of $\mathbf A_n,\widetilde{\mathbf A}_n$ and the boundary GMC. 

We are ready to state our first result on the connection between the annulus amplitude and Virasoro generators:
\begin{thm}\label{prop_intro}
\mbox{}
   \begin{enumerate}[label=\emph{(\arabic*)}]
       \item For $f\in\Sc$, the operator $\Tbf_f$ extends to a bounded operator $\Hc\to\Hc$, and the map $f\in\Ssg\mapsto\Tbf_f$ is a semigroup homomorphism, i.e., for $f_1,f_2\in\Ssg$,
        \begin{equation*}
            \Tbf_{f_1}\circ\Tbf_{f_2}=\Tbf_{f_1\circ f_2}.
        \end{equation*}
       \item  Let $f\in\Ssg$, and let $g_f$ be an admissible metric on the annulus $(\Ab_f=\Db\setminus f(\Db^\circ),\zetabf_f=(\id_\Tb,f|_\Tb))$. Viewing the CILT amplitude $\Ac_{\Ab_f,g_f,\zetabf_f}$ of $(\Ab_f,g_f,\zetabf_f)$ as an operator $\mathcal H\to\mathcal H$, we have
\[\Tbf_f = \frac{e^{-\frac{c_{\mathrm L}}{12} W(f, g_f)}}{\sqrt{2}\pi} \Ac_{\Ab_f,g_f,\zetabf_f},\]
where $W(f,g_f)$ is defined in \eqref{eq_W_f} and depends on $f$ and $g_f$ only.
\item \emph{(\textbf{Feynman--Kac formula})} Let $\v = -\sum_{n=0}^\infty v_n z^{n+1} \pa_z$ be a Markovian holomorphic vector field on $\Db$ such that the coercivity condition \eqref{omega_condition} holds. Then the operator $\Hbf_{\v}$ admits a closed extension that generates a strongly continuous semigroup $t\in\Rb_{\geq0}\mapsto e^{-t\Hbf_\v}$, and we have $\Tbf_{f_t} = e^{-t\Hbf_\v}$ for $t>0$, where $t\mapsto f_t$ is the flow generated by $\v$.
   \end{enumerate}
\end{thm}
When proving the Feynman--Kac formula, we have fewer tools compared to the real case where the GMC potential is positive. We comment on the technical difficulty in \cref{rmk_tech_GMC}. In short, we need uniform exponential-moment estimates for a sequence of GMC renormalizations compatible with the Hilbert space $\mathcal H$.

\cref{prop_intro} is proved in \cref{prop_semigroup_as_amplitude,T_f_semigroup,prop_FK_formula}. In the course of the proof, we will consider the positive quadratic form $\widetilde{\mathcal Q}^0(F,G):=\la(\mathbf H^0+Q^2)F,G\ra_{L^2}$, where $\mathbf H^0$ is the free Hamiltonian obtained by setting $\mu=0$ in the formula of $\mathbf H$ in \eqref{eq_def_H}. We refer to \cite[Section 4.2]{Guillarmou2024} which uses an analogous quadratic form. Let $\mathcal D(\widetilde{\mathcal Q}^0 )\subset\mathcal H$ be the domain of $\widetilde{\mathcal Q}^0$, and $\Dc'(\wt{\Qc}^0)\supset\Hc$ its dual. We now state our second result on the differentiability of $\Tbf_f$ in $f$:

\begin{thm}\label{thm_rep_segalsemigroup}
The map $f\in\Ssg\mapsto \Tbf_f$ is differentiable with differential
\[D_\v\Tbf_f = -\Tbf_f\circ\Hbf_{\rm w},\qquad {\rm w} = \tfrac{v(z)}{f'(z)}\,\pa_z,\]
where $\v=v(z)\pa_z$ is any holomorphic vector field on $\Db$ with $v(0)=0$. More precisely, given a differentiable family $t\mapsto f_t\in\Ssg$ with $f_0=f$ and $\pa_t|_{t=0}f_t=\v$ (in the sense that $\pa_t|_{t=0}f_t(z)=v(z)$ for $z\in\Db$), we have
\[\lim_{t\to0}\left\|\frac{\Tbf_{f_t}F-\Tbf_fF}{t} +\Tbf_f\Hbf_{\rm w}F\right\|_{\Dc(\wt{\Qc}^0)}=0,\qquad\fr F\in\Dc(\wt{\Qc}^0).\]
\end{thm}
The proof is done by adapting the differentiability framework of the real theory \cite[Theorem 4.8]{Segal_semigroup} to the non-self-adjoint generators. We propose \cref{assmp_diff} as a modification of \cite[Assumption 4.7]{Segal_semigroup} and prove \cref{prop_general_T_f_diff}.  The main difficulty of the proof is to show that $f\mapsto\mathbf T_f$ is continuous in $f$ (\cref{prop_Tf_continuity}), which again involves estimates of the imaginary GMC.

Note that the operators so far involve only $\mathbf L_n,\widetilde{\mathbf L}_n$ for $n\geq0$, and this is because we are only considering annuli $\Ab_f$ where $f$ extends holomorphically to $\Db$. In the case of a general annulus, we can use the same conformal welding argument as in the real case (\cite[Section 4.3]{Segal_semigroup}) to differentiate its amplitude, which will involve $\mathbf L_n,\widetilde{\mathbf L}_n$ for all $n\in\Zb$; see \cref{annulus_differential} for the precise statement. By the gluing axiom, we then know how to differentiate the amplitude of a general surface under a deformation of its boundary parametrizations; see \cref{fig_cutannuli}. As an application, we derive the Ward identities for the CILT amplitudes:


\begin{figure}[htbp]
\centering
\begin{tikzpicture}[scale=0.55,line join=round, line cap=round, line width=1.1pt]
\def\aI{0.82}   \def\bI{0.28}   \def\dI{0.80}   
\def\aII{0.75}  \def\bII{0.26}  \def\dII{0.80}  
\colorlet{collar}{black!12}

\begin{scope}[shift={(-4.3,1.75)}, rotate=-20]
  \coordinate (T1)  at (0,\aI);        \coordinate (B1)  at (0,-\aI);
  \coordinate (T1m) at (0.45,\aI);     \coordinate (B1m) at (0.45,-\aI);
  \coordinate (T1g) at (1.05,\aI);     \coordinate (B1g) at (1.05,-\aI);
  \fill[collar] (0,\aI) arc[start angle=90, end angle=-90, x radius=\bI, y radius=\aI]
             -- (\dI,-\aI) arc[start angle=-90, end angle=90, x radius=\bI, y radius=\aI] -- cycle;
  \draw (\dI,\aI) arc[start angle=90, end angle=-90, x radius=\bI, y radius=\aI];
  \draw (0,0) ellipse [x radius=\bI, y radius=\aI];
\end{scope}

\begin{scope}[shift={(-4.6,-1.3)}, rotate=3]
  \coordinate (T2)  at (0,\aII);       \coordinate (B2)  at (0,-\aII);
  \coordinate (T2m) at (0.45,\aII);    \coordinate (B2m) at (0.45,-\aII);
  \coordinate (T2g) at (1.05,\aII);    \coordinate (B2g) at (1.05,-\aII);
  \fill[collar] (0,\aII) arc[start angle=90, end angle=-90, x radius=\bII, y radius=\aII]
             -- (\dII,-\aII) arc[start angle=-90, end angle=90, x radius=\bII, y radius=\aII] -- cycle;
  \draw (\dII,\aII) arc[start angle=90, end angle=-90, x radius=\bII, y radius=\aII];
  \draw (0,0) ellipse [x radius=\bII, y radius=\aII];
\end{scope}

\draw plot[smooth, tension=0.7] coordinates {
  (T1) (T1m) (T1g) (-2.0,1.80) (-0.6,1.60) (0.8,1.75) (2.2,2.20) (3.6,2.45)
  (4.7,2.10) (5.45,1.20) (5.65,0.10) (5.35,-1.00) (4.5,-1.95) (3.2,-2.45)
  (1.6,-2.55) (0.0,-2.35) (-1.6,-2.15) (-2.8,-2.05) (B2g) (B2m) (B2) };

\draw plot[smooth, tension=0.6] coordinates {
  (B1) (B1m) (B1g) (-3.1,0.35) (-2.85,0.02) (-3.05,-0.28) (T2g) (T2m) (T2) };

\draw (1.30, 0.10) .. controls (2.10,-0.55) and (3.30,-0.55) .. (4.10, 0.10);
\draw (1.70,-0.08) .. controls (2.35, 0.35) and (3.05, 0.35) .. (3.70,-0.08);
\end{tikzpicture}
\caption{Annular collars cut from the boundary; gluing reduces amplitude variations to variations of these collars.
}
\label{fig_cutannuli}
\end{figure}
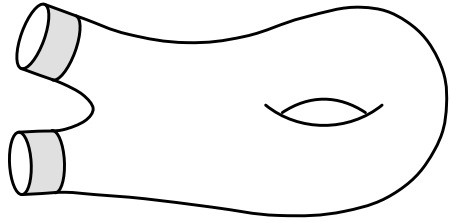

\begin{thm}[\textbf{Ward identity}]\label{thm_intro_WARD_CILT}
Let $\Sigma$ be a compact Riemann surface with real-analytic boundary parametrizations $\zetabf=(\zeta_1,\ldots,\zeta_\b)$. Suppose $\b=\b^++\b^-$ where $\zeta_j$ is outgoing for $j=1,\ldots,\b^+$ and incoming for $j=\b^++1,\ldots,\b$. Let $\zbf=(z_1,\ldots,z_\s)$ be distinct marked points on $\Sigma^\circ$, $\vbf=(v_1,\ldots,v_\s)$ with $v_j\in T_{z_j}\Sigma\setminus\{0\}$ for each $j$, $\alphabf=(\alpha_1,\ldots,\alpha_\s)\in\Rb^\s$ electric charges such that $\alpha_j>Q$, $\alpha_jR\in\Zb$ for all $j$, and $\mbf=(m_1,\ldots,m_\s)\in\Zb^\s$ magnetic charges. Let $g$ be an admissible metric on $\Sigma$. We have the amplitude operator $\Ac_{\Sigma,g,\vbf,\alphabf,\mbf,\zetabf}:\Hc^{\otimes\b^-}\to\Hc^{\otimes\b^+}$. Finally, let ${\rm w}$ be a holomorphic vector field on $\Sigma$ (analytic up to the boundary). For each $j=1,\ldots,\b$, let ${\rm w}_j:=\zeta_j^*{\rm w}=-\sum_{n\in\Zb}w_{jn}z^{n+1}\pa_z$, where we extend $\zeta_j$ analytically so that ${\rm w}_j$ is holomorphic on an open neighborhood of $\Tb$. Then the following identity holds:
\begin{align*}
0 &= \sum_{j=1}^{\b^+}\Hbf_{{\rm w}_j}^{(j)}\Ac_{\Sigma,g,\vbf,\alphabf,\mbf,\zetabf} - \sum_{j=1}^{\b^-}\Ac_{\Sigma,g,\vbf,\alphabf,\mbf,\zetabf}\Hbf_{{\rm w}_{\b^++j}}^{(j)}
+ \sum_{j=1}^{\s}\left(\Delta_{\alpha_j,m_j}(\div_gW)(z_j)+W_\vbf^{(j)}\right)\Ac_{\Sigma,g,\vbf,\alphabf,\mbf,\zetabf} \\
&\qquad+ c_\L\left(-\sum_{j=1}^{\b^+}\frac{\Re w_{j0}}{12}+\sum_{j=\b^++1}^{\b}\frac{\Re w_{j0}}{12}-\frac{1}{96\pi}\int_\Sigma 2K_g\div_gW\,\d v_g\right)\Ac_{\Sigma,g,\vbf,\alphabf,\mbf,\zetabf},
\end{align*}
where $\Hbf_{{\rm w}}^{(j)}:=\id_\Hc \otimes \cdots \otimes \underset{j}{\underbrace{\Hbf_{\rm w}}} \otimes \cdots \otimes \id_\Hc$ acts on the $j$-th $\Hc$, $W={\rm w}+\bar{{\rm w}}$, and $W_\vbf^{(j)}$ denotes differentiation with respect to $v_j$ in the direction of $W$ at $z_j$. Here $\Hbf_{\rm w}$ is defined as in \eqref{intro_H_v}, summing over $n\in\Zb$.
\end{thm}

This is proved as \cref{ward_cilt}.

\begin{remark}
The dependence of $\Ac_{\Sigma,g,\vbf,\alphabf,\mbf,\zetabf}$ on $\vbf$ records the spin $R(\alpha_j-Q)m_j$ at each $z_j$. More precisely, if for $\thetabf=(\theta_1,\ldots,\theta_\s)\in\Rb^\s$ we write $e^{\i\thetabf}\vbf=(e^{\i\theta_1}v_1,\ldots,e^{\i\theta_\s}v_\s)$, then we have (\cite[Theorem 1.2(3)]{BCILT})
\begin{equation*}
    \Ac_{\Sigma,g,e^{\i\thetabf}\boldsymbol{\v},\alphabf,\mbf,\zetabf}=e^{\i R\sum_{j=1}^\s(\alpha_j-Q)m_j\theta_j}\Ac_{\Sigma,g,\vbf,\alphabf,\mbf,\zetabf}.
\end{equation*}
In the Ward identity above, the tangent vectors $\vbf$ are also deformed by the flow $f_t$ generated by $\mathrm{w}$. If each $(f_t)_*v_j=e^{\lambda_j(t)}P_{g,j,t}v_j$ where $P_{g,j,t}:T_{z_j}\Sigma\to T_{f_t(z_j)}\Sigma$ is the $g$-parallel transport along $t\mapsto f_t(z_j)$ and $\lambda_j(t)\in\Cb$ depends continuously on $t$, then writing $\theta_j:=\Im\lambda_j$, we have $\partial_t|_{t=0} \theta_j(t)=-\frac12\div_g(JW)(z_j)$, where $JW:=\i \mathrm{w}-\i\bar{\mathrm{w}}$. Thus
\begin{equation*}
    W_\vbf^{(j)}\Ac_{\Sigma,g,\vbf,\alphabf,\mbf,\zetabf}=-\frac\i2 R(\alpha_j-Q)m_j (\div_g(J W))(z_j)\Ac_{\Sigma,g,\vbf,\alphabf,\mbf,\zetabf}+W_\zbf^{(j)}\Ac_{\Sigma,g,\vbf,\alphabf,\mbf,\zetabf},
\end{equation*}
where $W_\zbf^{(j)}$ denotes differentiation with respect to the marked point $z_j$ while parallel transporting $v_j$.
\end{remark}

Finally, we also develop the analogous theory for BCILT, i.e., the case with (Neumann) boundary. In particular, we study the quantization of Segal's semigroup of semiannuli and use this to prove the Ward identities for general BCILT amplitudes. The results are summarized in \cref{sec_BCILT}; see \cref{thm_BCILT_semigroup,ward_bcilt} for the statements.

\subsection*{Organization of the paper}
In \cref{Sec_background}, we introduce the background notations and recall some definitions and facts that are used in the paper. \cref{Sec_segal_quant} defines the quantization of Segal's semigroup and establishes its relation with the amplitudes of annuli. \cref{sec_Virasoro} then studies the generators of this quantization and proves the Feynman--Kac formula; this is done first for the free field, then for the Liouville case. \cref{sec_con_diff} is devoted to proving the differentiability of amplitudes, which is applied in \cref{sec:ward} to prove the Ward identities. Finally, we give a rapid treatment of the analogous theory for BCILT in \cref{sec_BCILT}.

\subsection*{Acknowledgments}
We would like to thank C. Guillarmou and R. Rhodes for introducing us to the topic, and C. Guillarmou, Yulai Huang and R. Rhodes for many helpful discussions. Y. Xiao acknowledges that this project has received funding from the European Union’s Horizon 2020 research and innovation programme under the Marie Skłodowska-Curie grant agreement No 101126554, as well as the support of ANR-21-CE40-0003. Y. Xie acknowledges support from the Jean-Pierre Aguilar scholarship from the Fondation CFM pour la Recherche. 
This work was funded by a Simons Foundation International grant administered by the Simons Foundation [SFI-MPS-PP-00012621-04, R.R.] and  [SFI-MPS-PP-00012621-04, C.G.].

The authors acknowledge the use of AI tools, mainly ChatGPT and Claude, in the proof of elementary Green's function estimates in \cref{lemma_lowerB_DGreen,lemma_upper_Green_DOZZ,lem_AI3}, a small functional analysis deduction in \cref{lemma_resolvent_technical}, and the continuity of the compactified boson partition function (item 1 in \cref{lem_continuity_tech}). AI has also provided many useful suggestions when verifying and polishing the paper.
\section{Background and notations}\label{Sec_background}

For a Hilbert space $\Hc$, we denote by $\Lc(\Hc)$ the space of bounded linear operators on $\Hc$.

We denote by $\Db:=\{z\in\Cb:|z|\leq1\}$ the closed unit disk and by $\Tb:=\pa\Db=\{z\in\Cb:|z|=1\}$ the unit circle. With $\Hb:=\{z\in\Cb:\Im z \geq0\}$ denoting the closed upper half plane, we denote by $\Db^+:=\Db\cap\Hb$ the closed unit half disk and by $\Tb^+:=\Tb\cap\Hb$ the unit half circle. A map defined on a closed set in a Riemann surface is said to be holomorphic if it extends to a holomorphic map on an open neighborhood. In particular, we shall speak of holomorphic maps on $\Db$ and $\Tb$. These will extend to holomorphic maps on $\{|z|<1+\delta\}$ and $\{(1+\delta)^{-1}<|z|<1+\delta\}$, respectively, for some $\delta>0$.

Throughout the paper, we fix a compactification radius $R>0$ and write $\TR:=\Rb/2\pi R\Zb$. Let us recall the measure spaces of $\TR$-valued boundary fields defined in \cite{BCILT}. For the unit circle $\Tb$, we consider the space
\[\DT = \TR \times \Dc_0'(\Tb,\Rb) \times \Zb,\]
with the corresponding decomposition of an element therein
\[\widetilde{\varphi}^k = c + \varphi + kR\arg,\]
understood intuitively as a $\Tb_R$-valued random distribution on $\Tb$ with winding number $k$. (In the notation of \cite{CILT}, the space is $H^s_{\mathbb Z}(\mathbb R)/\langle\tau_{R}\rangle\cong \TR \times H^s_0(\Tb,\Rb) \times \Zb$, defined in terms of the Sobolev space $H^s$ of fixed negative order $s<0$. This differs from our space only by a set of measure zero.) Here the mean-zero component $\varphi$ is sampled from the circle GFF (viewing functions on $\Tb$ as $2\pi$-periodic functions on $\Rb$)
\[\varphi(\theta)=\sum_{n>0}\frac{1}{\sqrt{n}}\big(x_n\cos(n\theta)-y_n\sin(n\theta)\big)=\sum_{n\neq0}\varphi_ne^{\i n\theta},\]
where the $x_n,y_n$ are i.i.d.\ standard Gaussians, $\varphi_n=\frac{1}{2\sqrt{n}}(x_n+\i y_n)$, $\varphi_{-n}=\ol{\varphi_n}$ for $n>0$. We put the induced Gaussian probability measure
\[\d\varphi=\bigotimes_{n>0}\frac{1}{2\pi}\,e^{-\frac{1}{2}(x_n^2+y_n^2)}\,\d x_n\,\d y_n\]
on $\Dc_0'(\Tb,\Rb)$, the Lebesgue measure on $\TR$, and the counting measure on $\Zb$, so as to obtain a product measure $\d\wt{\varphi}^k=\d c\,\d\varphi\,\d k$ on $\DT$. Similarly, for the unit half circle $\Tb^+$, we consider
\[\DTh=\TR\times\Dc_0'(\Tb^+,\Rb),\]
with the corresponding decomposition
\[\wt{\varphi}^\h=c+\varphi^\h,\]
where $\h$ stands for ``half''. The mean-zero component $\varphi^\h$ is sampled from the GFF
\[\varphi^\h(\theta)=\sum_{n>0}\sqrt{\frac{2}{n}}\,x_n^\h\cos(n\theta)=\sum_{n\neq0}\varphi_n^\h e^{\i n\theta},\]
where the $x_n^\h$ are i.i.d.\ standard Gaussians, $\varphi_n^\h=\frac{1}{\sqrt{2n}}x_n^\h$, $\varphi_{-n}^\h=\varphi_n^\h$ for $n>0$. This defines a measure $\d\wt{\varphi}^\h=\d c\,\d\varphi^\h$ on $\DTh$ with
\[\d\varphi^\h=\bigotimes_{n>0}\frac{1}{\sqrt{2\pi}}\,e^{-\frac{1}{2}(x_n^\h)^2}\,\d x_n^\h.\]
We shall be concerned with the Hilbert spaces $\Hc:=L^2(\DT)$ and $\Hc^+:=L^2(\DTh)$, which are the state spaces of BCILT. Note that one can view $\DTh$ naturally as the subspace of \textit{even} elements $u\in\DT$, i.e., they satisfy $u\circ\tau=u$ for $\tau:\theta\mapsto-\theta$.

For $\widetilde{\varphi}^k\in\DT$, we denote $P\widetilde{\varphi}^k = c + P\varphi + kR\arg$ as a $\TR$-valued field on $\Db$, where the latter $P=P_\Db$ is the Poisson operator on $\Db$. Similarly, for $\widetilde{\varphi}^\h\in\DTh$, $P\widetilde{\varphi}^\h = c + P\varphi^\h$ is an even field on $\Db$, so it can be viewed as a field on $\Db^+$ with Neumann boundary condition at $\Db\cap\Rb\subset\pa\Db^+$. We shall also need the Dirichlet-to-Neumann operator $\Dbf$ on $\Tb$ defined by $\Dbf e^{\i n\theta}=|n|e^{\i n\theta}$ for $n\in\Zb$. See \cite[Section 2]{BCILT} for more background material on these operators.

We recall briefly the definition of the Gaussian Free Field (GFF) and the imaginary Gaussian Multiplicative Chaos (GMC) on a general surface $\Sigma$ with possibly both Dirichlet and Neumann boundaries. More precisely, each of its Neumann boundary $\paN\Sigma$ and Dirichlet boundary $\paD\Sigma$ consists of circles and semicircles, and each boundary component of $\Sigma$ is either a Neumann circle, a Dirichlet circle, or a piecewise real-analytic circle consisting of Neumann and Dirichlet semicircles, where the intersection between a Neumann semicircle and a Dirichlet semicircle is a corner of $\Sigma$. The Dirichlet circles and semicircles all come with real-analytic parametrizations using $\Tb$ and $\Tb^+$, respectively. Such a surface (together with marked points) is called an \textit{extended surface}, see \cite[Definition 1.1]{BCILT} for the precise definition. Let $g$ be an admissible conformal metric on $\Sigma$. We denote by $\Delta_g$ the $g$-Laplacian on $\Sigma$, with the appropriate mixed boundary conditions at $\pa\Sigma$. If it has orthonormal eigenbasis $(e_j)_{j\ge1}$ with eigenvalues $(\lambda_j)_{j\geq1}$ (with mean zero if $\Sigma$ has no Dirichlet boundary), then the GFF on $(\Sigma,g)$ is defined by
\begin{equation*}
    X_g:=\sqrt{2\pi}\sum_{j\ge1}a_j\frac{e_j}{\sqrt{\lambda_j}},
\end{equation*}
where $a_j$ ($j\geq1$) are i.i.d.\ standard Gaussian random variables. This sum converges almost surely as a random distribution. Denoting by $G_g(\cdot,\cdot)$ the Green's function on $(\Sigma,g)$ (with the same boundary conditions, and mean zero if $\paD\Sigma=\varnothing$), we have 
\begin{equation*}
    \mathbb E[X_g(x)X_g(x')]=2\pi G_g(x,x').
\end{equation*}
Next we turn to the bulk and boundary imaginary GMCs, used in the construction of (B)CILT and studied in \cite{Lacoin2015}. The bulk GMC on $(\Sigma,g)$ is defined by
\begin{equation*}
    M_{\beta}(X_g,\d v_g):=\lim_{\epsilon\to 0} \epsilon^{-\frac{\beta^2}{2}}e^{\i\beta X_{g,\epsilon}(x)}\d v_g,
\end{equation*}
where $X_{g,\epsilon}$ is an appropriate regularization of the GFF $X_g$ at scale $\epsilon$ with respect to the metric $g$; for example, we can take it to be the average over the $g$-geodesic circle of radius $\epsilon$. We also have the boundary GMC on the Neumann boundary $\paN\Sigma$, defined by
\begin{equation*}
    L_{\beta}(X_g,\d \ell_g):=\lim_{\epsilon\to 0} \epsilon^{-\frac{\beta^2}{4}}e^{\i\frac\beta2 X_{g,\epsilon}(x)}\d \ell_g.
\end{equation*}
For $\beta^2<2$, $M_\beta$ and $L_\beta$ converge in $L^2$ and weakly in the space of distributions to a nontrivial distribution of order $2$ \cite[Theorem 3.1]{Lacoin2015}. Specifically, there exists a random variable $D_{\Sigma}$ with $\mathbb E|D_{\Sigma}|^2<\infty$ such that for all $f\in C^{\infty}(\Sigma)$,
\begin{equation*}
    \left|\int_{\Sigma}f(x)M_\beta(X_g,\d v_g(x))+\int_{\paN\Sigma}f(x)L_\beta(X_g,\d\ell_g(x))\right|\leq D_{\Sigma}(\sup_{\Sigma}|f|+\sup_{\Sigma}|\Delta_gf|+\sup_{\paN\Sigma}|\partial_\nu f|).
\end{equation*}

There are many alternative ways of defining the $\epsilon$-regularization $X_{g,\epsilon}$, and in the proof of \cref{lemma_gmc_delta_exp_bound} below, we will use the white-noise regularization. Assume that $\paD\Sigma\not=\emptyset$. Then we take
\begin{equation*}
    X_{g,\delta}(x):=\sqrt{2\pi}\int_{\delta^2}^{\infty}\int_{\Sigma}p_{t/2}(x,y) W(\d t,\d y),
\end{equation*}
where $p_{t/2}(x,y)$ is the heat kernel of $\Delta_g$, and $W$ is the white noise on $\mathbb R_+\times\Sigma$ with intensity $\d t\otimes \d v_g(x)$. We can also compute the covariance of $X_{g,\delta}$:
\begin{equation*}
    \mathbb E[X_{g,\delta}(x)X_{g,\delta'}(x')]=2\pi\int_{\delta^2\vee\delta'^2}^\infty p_t(x,x')\,\d t.
\end{equation*}


For the definition of the BCILT amplitude of a general extended surface, we refer to \cite[Section 7.2]{BCILT}. Here we recall only the definition for annuli, since this will be used in \cref{sec_con_diff}. Consider a holomorphic embedding $f:\Tb\to\Db^\circ$ such that $f(\Tb)$ winds around $0$. Let $\Ab_f$ be the annulus bounded between $\Tb$ and $f(\Tb)$, with boundary parametrizations $\zetabf_f=(\id_\Tb,f|_\Tb)$. Let $g_f$ be an admissible metric on $\Ab_f$. The amplitude $\Ac_{\Ab_f,g_f,\zetabf_f}$ is a function of $\widetilde{\boldsymbol{\varphi}}^k=(\widetilde{{\varphi}}^k_1,\widetilde{{\varphi}}^k_2)\in\DT^2$ with $k\in\Zb$, since a field on $\Ab_f$ necessarily has the same winding number on $\Tb$ and $f(\Tb)$. Recall the decompositions $\wt\varphi_1^k=c_1+\varphi_1+kR\arg$, $\wt\varphi_2^k=c_2+\varphi_2+kR\arg$ with $c_1,c_2\in\TR$, $\varphi_1,\varphi_2\in\Dc_0'(\Tb,\Rb)$. Let us take representatives of $c_1,c_2$ in $\Rb$, denoted in the same way. For $j\in\Zb$, define $\cbf_j:=(c_1+2\pi Rj,c_2)$ and the $1$-form $\nu_{f,k}$ on $\Ab_f$
\begin{equation*}
    \nu_{f,k}:= kR\, d\arg+\,d P(0,k R(\arg f^{-1}-\arg)|_{f(\mathbb T)}),
\end{equation*}
where $P=P_{\Ab_f}$ is the Poisson operator on $\Ab_f$, i.e., $P(h_1,h_2)$ is the unique harmonic function on $\Ab_f$ that equals $h_1$ on $\mathbb T$ and $h_2$ on $f(\mathbb T)$. To define the curvature terms, we also need to fix a separating family (\cite[Section 2.5]{BCILT}, also called a defect graph in \cite{CILT}) on $\Ab_f$. 
Let $d_f\subset\mathbb A_f$ be a smooth path from $f(1)$ to $1$ such that it is orthogonal to $f(\Tb)$ at $f(1)$ and orthogonal to $\Tb$ at $1$ (i.e., tangent to $\Rb$). Then we set $I(\nu_{f,k})(z):=\int_1^z\nu_{f,k}$ by integrating along any path $\gamma:[0,1]\to\Ab_f$ with $\gamma(0)=1$, $\gamma(1)=z$, intersecting $d_f$ only at $\gamma(0)=1$ and such that $\Im\gamma(t)\geq0$ for $t$ close to $0$. This is well-defined as a single-valued function $I(\nu_{f,k}):\Ab_f\setminus d_f\to\Rb$. Denote by $\Delta_{\mathbb A_f,g_f,\mathrm{D}}$ the Dirichlet Laplacian on $(\Ab_f,g_f)$ and $\mathbf{D}_{\Ab_f}$ the Dirichlet-to-Neumann operator on $\mathbb A_f$ (see for instance \cite[Section 2.4]{BCILT}). The annulus amplitude $\Ac_{\Ab_f,g_f,\zetabf_f}$ is defined by
\begin{equation*}
\begin{split}
    \Ac_{\Ab_f,g_f,\zetabf_f}(\widetilde{\boldsymbol{\varphi}}^k)={}&\det(\Delta_{\mathbb A_f,g_f,\mathrm{D}})^{-\frac12}
    e^{-\frac1{4\pi} \int_{\Ab_f}|\nu_{f,k}|^2_{g_f}\d v_{g_f}}
    \sum_{j\in\mathbb Z} e^{-\frac1{2}\la\cbf_j+\boldsymbol{\varphi},(\mathbf{D}_{\Ab_f}-\mathbf D)(\cbf_j+\boldsymbol{\varphi})\ra-\frac{1}{2\pi}\int_{\partial\mathbb A_f}P(\cbf_j+\boldsymbol{\varphi})\,\nu_{f,k}(\nbf)\,\d\ell_{g_f}}\\
    &\times\mathbb E_{\widetilde{\boldsymbol{\varphi}}^k}\left[e^{-\frac{\i Q}{4\pi}\int_{\Ab_f}^{\mathrm{reg}}K_{g_f}(I(\nu_{f,k})+P(\cbf_j+{\boldsymbol{\varphi}})+X_{g_f})\,\d v_{g_f}-\mu\int_{\Ab_f}M_\beta(I(\nu_{f,k})+P(\cbf_j+{\boldsymbol{\varphi}})+X_{g_f},\d v_{g_f})}\right],
\end{split}
\end{equation*}
where $\nbf$ is the outward unit normal vector along $\pa\mathbb A_f$ and $X_{g_f}$ is the Dirichlet GFF on $(\Ab_f,g_f)$. Here the curvature term $\int_{\Ab_f}^{\mathrm{reg}}K_{g_f}I(\nu_{f,k})\,\d v_{g_f}$ is regularized as in \cite[Section 5]{BCILT}. Explicitly, it is given by
\begin{equation*}
    \int_{\Ab_f}^{\mathrm{reg}}K_{g_f}I(\nu_{f,k})\,\d v_{g_f}:= \int_{\Ab_f\setminus d_f}K_{g_f}I(\nu_{f,k})\,\d v_{g_f}-4\pi kR\int_{d_f}k_{g_f}\d \ell_{g_f},
\end{equation*}
where $k_{g_f}$ is the geodesic curvature and $\d\ell_{g_f}$ is the length measure under metric $g_f$, both along $d_f$.

\section{Quantization of Segal's semigroup}
\label{Sec_segal_quant}
\begin{figure}[ht]
\centering

\begin{subfigure}[b]{0.47\textwidth}
\centering
\begin{tikzpicture}[scale=2.0,
  orient/.style={
    postaction={
      decorate,
      decoration={markings, mark=at position #1 with {\arrow{>}}}
    }
  },
  innerorient/.style={
    postaction={
      decorate,
      decoration={
        markings,
        mark=at position #1 with {
          \begin{scope}[rotate=-15]
            \arrow{>}
          \end{scope}
        }
      }
    }
  }
]
  \fill[gray!25] (0,0) circle (1);

  \fill[white]
    plot[smooth cycle, tension=0.95] coordinates {
      (0.42,0.02) (0.34,0.16) (0.39,0.30) (0.18,0.32)
      (0.03,0.40) (-0.12,0.32) (-0.30,0.35) (-0.34,0.17)
      (-0.42,0.01) (-0.33,-0.15) (-0.35,-0.30) (-0.16,-0.31)
      (0.00,-0.40) (0.15,-0.31) (0.34,-0.28) (0.31,-0.12)
    };

  \draw[thick, orient=.25] (0,0) circle (1);

  \draw[thick, innerorient=.2]
    plot[smooth cycle, tension=0.95] coordinates {
      (0.42,0.02) (0.34,0.16) (0.39,0.30) (0.18,0.32)
      (0.03,0.40) (-0.12,0.32) (-0.30,0.35) (-0.34,0.17)
      (-0.42,0.01) (-0.33,-0.15) (-0.35,-0.30) (-0.16,-0.31)
      (0.00,-0.40) (0.15,-0.31) (0.34,-0.28) (0.31,-0.12)
    };

  \node at (0,0.70) {$\mathbb A_f$};
  \node at (0,0) {$f(\mathbb D)$};
  \node[above right] at (0.68,0.68) {$\mathbb T$};
  \node[below] at (0.20,-0.35) {$f(\mathbb T)$};
\end{tikzpicture}
\vspace{0.2cm}
\caption{The annulus \(\mathbb A_f=\Db\setminus f(\Db^\circ)\) associated to $f\in\Ssg$}
\label{fig:segal-annulus}
\end{subfigure}
\hspace{0.02\textwidth}
\begin{subfigure}[b]{0.47\textwidth}
\centering
\raisebox{0.5cm}{%
\begin{tikzpicture}[scale=2.0,
  orient/.style={
    postaction={
      decorate,
      decoration={markings, mark=at position #1 with {\arrow{>}}}
    }
  },
  innerorient/.style={
    postaction={
      decorate,
      decoration={
        markings,
        mark=at position #1 with {
          \begin{scope}[rotate=-10]
            \arrow{>}
          \end{scope}
        }
      }
    }
  }
]
  \fill[gray!25] (-1,0) arc[start angle=180,end angle=0,radius=1] -- cycle;

  \fill[white]
    (0.42,0)
    .. controls (0.42,0.08) and (0.39,0.15) .. (0.35,0.21)
    .. controls (0.29,0.30) and (0.20,0.30) .. (0.13,0.37)
    .. controls (0.04,0.45) and (-0.05,0.37) .. (-0.13,0.34)
    .. controls (-0.23,0.29) and (-0.30,0.37) .. (-0.36,0.25)
    .. controls (-0.41,0.16) and (-0.42,0.08) .. (-0.42,0)
    -- cycle;

  \draw[thick, orient=.5] (1,0) arc[start angle=0,end angle=180,radius=1];
  \draw[thick] (-1,0) -- (1,0);

  \draw[thick, innerorient=.48]
    (0.42,0)
    .. controls (0.42,0.08) and (0.39,0.15) .. (0.35,0.21)
    .. controls (0.29,0.30) and (0.20,0.30) .. (0.13,0.37)
    .. controls (0.04,0.45) and (-0.05,0.37) .. (-0.13,0.34)
    .. controls (-0.23,0.29) and (-0.30,0.37) .. (-0.36,0.25)
    .. controls (-0.41,0.16) and (-0.42,0.08) .. (-0.42,0);

  \draw[thick] (-0.42,0) -- (0.42,0);

  \node at (0,0.70) {$\mathbb A_f^+$};
  \node at (0,0.15) {$f(\mathbb D^+)$};

  \node[above right] at (0.68,0.68) {$\mathbb T^+$};
  \node[right] at (0.29,0.3) {$f(\mathbb T^+)$};

  \node[below] at (-0.70,0) {$[-1,f(-1)]$};
  \node[below] at (0.70,0) {$[f(1),1]$};
\end{tikzpicture}
}

\caption{The semiannulus \(\mathbb A_f^+=\Ab_f\cap\Hb\) associated to $f\in\Ssg^+$}
\label{fig:segal-semiannulus}
\end{subfigure}

\caption{Annuli and semiannuli associated to Segal's semigroups}
\label{fig:segal-annuli}
\end{figure}
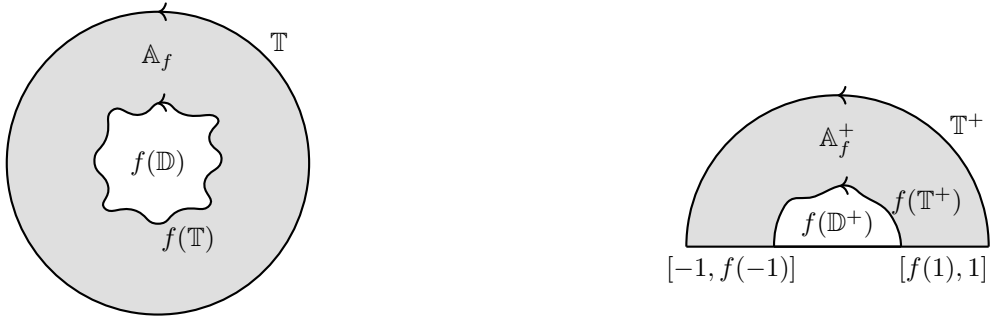

The set of holomorphic embeddings $f:\Db\to\Db^\circ$ with $f(0) = 0$ forms a semigroup under composition. This is known as \textit{Segal's semigroup} and we will denote it by $\Ssg$. To an element $f\in\Ssg$, we associate the annulus $\Ab_f:=\Db\setminus f(\Db^\circ)$ with boundary parametrizations $\zetabf_f:=(\id_\Tb,f|_\Tb)$ (\Cref{fig:segal-annulus}). Then composition on $\Ssg$ corresponds to gluing the annuli. Indeed, for $f_1,f_2\in\Ssg$, the annulus $\Ab_{f_1\circ f_2}=\Db\setminus(f_1\circ f_2)(\Db^\circ)=(\Db\setminus f_1(\Db^\circ))\cup f_1(\Db\setminus f_2(\Db^\circ))=\Ab_{f_1}\cup f_1(\Ab_{f_2})$ is obtained by gluing the incoming boundary $f_1(\Tb)$ of $\Ab_{f_1}$ with the outgoing boundary $\Tb$ of $\Ab_{f_2}$ via $f_1|_\Tb$. Thus Segal's semigroup $\Ssg$ can be viewed as a semigroup of annuli.
Using the construction of amplitude operators, the Segal functor of CILT provides a quantization of this semigroup into bounded operators acting on the Hilbert space $\Hc=L^2(\DT)$, which we shall describe in the following paragraphs.

For $f \in \Ssg$, we define an operator $\Tbf_f$ on $L^2(\DT)$ as follows. Analogously to the definition of the elementary function spaces in \cite{BCILT}, let $\Ec(\DT)$ be the linear space spanned by functions $F:\DT\to\Cb$ of the form
\begin{equation}\label{eq_def_E_space}
F(\wt{\varphi}^k) = e^{\i nc/R}\, P(\la\varphi,h_1\ra,\ldots,\la\varphi,h_m\ra)\, G(k),
\end{equation}
where $n\in\Zb$, $m\in\Nb$, $P$ is a complex polynomial, $h_1,\ldots,h_m\in C^\infty(\Tb)$, $G:\Zb\to\Cb$ has finite support. Here the bracket is $\la u,v\ra:=\frac{1}{2\pi}\int_\Tb u\bar{v}\,\d\theta$. This function space is clearly dense in $L^2(\DT)$. For $F\in\Ec(\DT)$, we define $\Tbf_fF:\DT\to\Cb$ by
\begin{equation}\label{eq_def_T_f}
(\Tbf_fF)(\wt{\varphi}^k) := |f'(0)|^{\frac{R^2k^2-Q^2}2} \,\Eb\left[F\left((P\wt{\varphi}^k+X_\Db)\circ f\big|_\Tb + \i Q\log\tfrac{|f'|}{|f|}\big|_\Tb\right)\,e^{-\mu\int_{\Ab_f}|x|^{\beta Q}M_\beta(P\wt{\varphi}^k+X_\Db,\d x)}\right],
\end{equation}
where $X_\Db$ is the Dirichlet GFF on $\Db$ and the expectation is over $X_\Db$, and $M_\beta$ denotes the imaginary GMC. Here the imaginary translation by $\i Q\log\tfrac{|f'|}{|f|}\big|_\Tb$ makes sense since $F$ is analytic in $c$ and polynomial in $\varphi$.

\begin{prop}\label{prop_semigroup_as_amplitude}
Let $f\in\Ssg$. Let $g_\Ab:=|dz|^2/|z|^2$ and $g_f$ an admissible metric on the annulus $(\Ab_f,\zetabf_f)$. Viewing the CILT amplitude $\Ac_{\Ab_f,g_f,\zetabf_f}$ of $(\Ab_f,g_f,\zetabf_f)$ as an operator $L^2(\DT)\to L^2(\DT)$, we have
\[\Tbf_f = \frac{e^{-\frac{c_{\mathrm L}}{12} W(f, g_f)}}{\sqrt{2}\pi} \Ac_{\Ab_f,g_f,\zetabf_f},\]
where
\begin{equation}\label{eq_W_f}
    W(f,g_f) := -\log|f'(0)| - \left\langle\Dbf\log\tfrac{|f'|}{|f|}\big|_\Tb,\log\tfrac{|f'|}{|f|}\big|_\Tb\right\rangle - 12S_{\mathrm L}^0(\Ab_f,g_f,g_\Ab).
\end{equation}
Here the bracket $\la\cdot,\cdot\ra$ on $L^2(\Tb)$ is defined by $\la u,v\ra:=\frac{1}{2\pi}\int_\Tb u\bar{v}\,\d\theta$.
\end{prop}
\begin{proof}
Let $g_\Db:=|dz|^2$ and $g$ a metric on $\Db$ extending $g_f$. We apply the gluing axiom for the free field (\cite[Theorem 4.1]{BCILT}) to $\Ab_f\# f(\Db)=\Db$ with a magnetic insertion at $z=0$. Informally and in terms of integrals, it states that for $\wt{\varphi}^k\in\DT$ and functionals $F_1,F_2$, we have
\[\int_{\phi:\Db\to\TR}F_1(\phi|_{f(\Db)})F_2(\phi|_{\Ab_f})\,\D\phi = C\int_{\wt{\varphi}_2^k:\Tb\to\TR}\left(\int_{\phi_1:f(\Db)\to\TR}F_1(\phi_1)\,\D\phi_1\right)\left(\int_{\phi_2:\Ab_f\to\TR}F_2(\phi_2)\,\D\phi_2\right)\D\wt{\varphi}_2^k,\]
where the constraints on the integrals are $\phi|_\Tb=\wt{\varphi}^k$ and $\phi_1\circ f|_\Tb=\phi_2\circ f|_\Tb=\wt{\varphi}_2^k$, and $C=\frac{1}{\sqrt{2}\pi}$ in this case. Taking $F_1(\phi_1)=F(\phi_1\circ f|_\Tb)$ and $F_2$ the integrand of the CILT amplitude on $(\Ab_f,g_f,\zetabf_f)$, we obtain
\begin{equation}\label{annulus_amplitude_gluing}
\begin{aligned}
&\frac{1}{\sqrt{2}\pi} \frac{e^{-\frac{1}{4\pi} R^2k^2\int_{f(\Db)}^{\reg}|d\arg f^{-1}|_{g}^2\,\d v_{g}}}{\sqrt{\det'\Delta_{g|_{f(\Db)}}}} \int_{\DT} \Ac_{\Ab_f,g_f,\zetabf_f}(\wt{\varphi}^k,\wt{\varphi}_2^k) \,F(\wt{\varphi}_2^k)\,\d\wt{\varphi}_2^k \\
&\quad = \frac{e^{-\frac{1}{4\pi} R^2k^2\int_{\Db}^{\reg}|d\arg\!|_{g}^2\,\d v_{g}}}{\sqrt{\det'\Delta_g}}\, \Eb\left[F\big((P\wt{\varphi}^k+X_\Db)\circ f\big|_{\Tb}\big)\,e^{-\frac{\i Q}{4\pi}\int_{\Ab_f}K_{g_f}(P\wt{\varphi}^k+X_\Db)\,\d v_{g_f}-\mu\int_{\Ab_f}M_\beta(P\wt{\varphi}^k+X_\Db,\d v_{g_f})}\right],
\end{aligned}
\end{equation}
where the expectation is over the GFF $X_\Db$, the integrals $\int_{f(\Db)}^{\reg}|d\arg f^{-1}|_{g}^2\,\d v_{g},\int_{\Db}^{\reg}|d\arg\!|_{g}^2\,\d v_{g}$ are regularized as in \cite[Section 2.6]{BCILT} by taking the principal part, and the curvature term $\int_{\Ab_f}K_{g_f}P\wt{\varphi}^k\,\d v_{g_f}$ is regularized as in \cite[Section 5]{BCILT}.

We first compute $\int_{f(\Db)}^{\reg}|d\arg f^{-1}|_{g}^2\,\d v_{g}-\int_{\Db}^{\reg}|d\arg\!|_{g}^2\,\d v_{g}$. Since this is conformally invariant, using $g_\Db$ in place of $g$, it becomes
\[\lim_{\varepsilon\to0}\left(\int_{\Db\setminus f^{-1}(\varepsilon\Db)}|d\arg\!|_{g_\Db}^2\,\d v_{g_\Db} - \int_{\Db\setminus\varepsilon\Db}|d\arg\!|_{g_\Db}^2\,\d v_{g_\Db}\right) = -\lim_{\varepsilon\to0}\int_{f^{-1}(\varepsilon\Db)\setminus\varepsilon\Db}\frac{\d r}{r}\,\d\theta.\]
Now $f(z)=f'(0)z+O(z^2)$ as $z\to0$, so it is not hard to see that this converges to $2\pi\log|f'(0)|$. Thus $e^{-\frac{1}{4\pi} R^2k^2\left(\int_{f(\Db)}^{\reg}|d\arg f^{-1}|_{g}^2\,\d v_{g}-\int_{\Db}^{\reg}|d\arg\!|_{g}^2\,\d v_{g}\right)} = |f'(0)|^{-\frac{R^2k^2}{2}}$.

Next we claim that $\int_{\Ab_f}K_{g_f}P\wt{\varphi}^k\,\d v_{g_f}=0$ (modulo $2\pi R\Zb$). 
Indeed,
\[\int_{\Ab_f}K_{g_f}P\wt{\varphi}^k\,\d v_{g_f} = \int_{\Db}K_{g}P\wt{\varphi}^k\,\d v_{g} - \int_{f(\Db)}K_{g|_{f(\Db)}}P\wt{\varphi}^k\,\d v_{g|_{f(\Db)}},\]
where we view both $\Db$ and $f(\Db)$ as surfaces with a magnetic insertion at $0$ and tangent vector $\pa_z$. By the formula for change of conformal metric (\cite[Proposition 5.2.(4)]{BCILT}), writing $g=e^\rho\,|dz|^2$, we have
\[\int_{\Db}K_{g}P\wt{\varphi}^k\,\d v_{g} = 2\int_{\pa\Db}\wt{\varphi}^k\,|dz| + \int_{\Db}\langle d\rho,dP\wt{\varphi}^k\rangle_{g_\Db}\,|dz|^2.\]
We compute the first term by regularizing with respect to the separating family consisting only of the straight line segment from $0$ to $1$. Then
\[\int_{\pa\Db}\wt{\varphi}^k\,|dz| = 2\pi c + \int_{\Tb}\varphi + Rk\int_0^{2\pi}\theta\,\d\theta = 2\pi(P\wt{\varphi})(0) + 2\pi^2Rk,\]
where we used the mean value property of harmonic functions. As for the second term, since $g_f$ is admissible, we have $\rho(z) = -2\log|z|$ near $\pa\Db$, so by integration by parts,
\[\int_{\Db}\langle d\rho,dP\wt{\varphi}^k\rangle_{g_\Db}\,|dz|^2 = \int_{\pa\Db} \rho\,\frac{\pa P\wt{\varphi}^k}{\pa\nu}\,|dz| = 0.\]
By the same token,
\[\int_{f(\Db)}K_{g|_{f(\Db)}}P\wt{\varphi}^k\,\d v_{g|_{f(\Db)}} = \int_\Db K_{f^*g}(P\wt{\varphi}^k\circ f)\,\d v_{f^*g} = 4\pi(P\wt{\varphi}\circ f)(0) + 2Rk\int_{-\arg f'(0)}^{2\pi-\arg f'(0)}\arg f(e^{\i\theta})\,\d\theta,\]
where we regularize on $(\Db,f^*g)$ with respect to the separating family consisting only of the straight line segment from $0$ to $e^{-\i\arg f'(0)}$. 
Now $\arg(f(e^{-\i\arg f'(0)}z)/z)$ is well-defined and harmonic on $\Db$ and its value at $0$ is $0$, so by the mean value property,
\[\int_{-\arg f'(0)}^{2\pi-\arg f'(0)}\arg f(e^{\i\theta})\,\d\theta=\int_0^{2\pi}\arg f(e^{-\i\arg f'(0)}e^{\i\theta})\,\d\theta=\int_0^{2\pi}\theta\,\d\theta=2\pi^2.\]
Thus the claim is proved.

Returning to \eqref{annulus_amplitude_gluing}, we perform the imaginary Girsanov transform (see \cite[Theorem 3.4]{BCILT}) with respect to
\[Y = -\frac{\i Q}{4\pi}\int_{\Ab_f}K_{g_f}X_\Db\,\d v_{g_f}.\]
Consider the harmonic function $h_f$ on $f(\Db)$ defined by $h_f(f(z))=\log\frac{|zf'(z)|}{|f(z)|}$, so that $f^*g_\Ab=e^{2h_f\circ f}g_\Ab$. Note that $h_f\circ f\big|_{\Tb}=\log\frac{|f'|}{|f|}\big|_\Tb$. Since $g_f$ is admissible, we have $g_f=e^{-2\chi_f}g_\Ab$ where $\chi_f\in C^\infty(\Ab_f,\Rb)$ is $0$ near $\Tb$ and equal to $h_f$ near $f(\Tb)$. We extend $\chi_f$ to $\Db$ by defining it to be $h_f$ on $f(\Db)$. Now $K_{g_f} = -2e^{2\chi_f}\Delta_{g_\Ab}\chi_f$, $\d v_{g_f} = e^{-2\chi_f}\,\d v_{g_\Ab}$, so
\[\Eb[YX_\Db] = \i Q\int_{\Ab_f}\Delta_{g_\Ab}\chi_f(x)\,G(x,\cdot)\,\d v_{g_\Ab}(x) = \i Q\int_{\Db}\Delta_{g_\Db}\chi_f(x)\,G(x,\cdot)\,\d v_{g_\Db}(x) = \i Q\chi_f,\]
where $G=G_\Db$ is the Green's function on $\Db$ with Dirichlet boundary condition. Similarly,
\begin{align*}
\Eb[Y^2] &= -\frac{Q^2}{2\pi}\int_{\Ab_f}\int_{\Ab_f}\Delta_{g_\Ab}\chi_f(x)\,G(x,y)\,\Delta_{g_\Ab}\chi_f(y)\,\d v_{g_\Ab}(x)\,\d v_{g_\Ab}(y) = -\frac{Q^2}{2\pi}\int_\Db\chi_f\,\Delta_{g_\Ab}\chi_f\,\d v_{g_\Ab} \\
&= -12Q^2S_{\mathrm L}^0(\Ab_f,g_f,g_\Ab) - \frac{Q^2}{2\pi}\int_{f(\Tb)}\chi_f\,\frac{\pa\chi_f}{\pa\nu}\,\d\ell_{g_\Ab},
\end{align*}
where $\nu$ is the unit normal vector pointing towards $\Ab_f$. Since $\chi_f$ is harmonic in $f(\Db)$, $\frac{\pa\chi_f}{\pa\nu}\circ f=\frac{|f|}{|f'|}\Dbf(\chi_f\circ f|_\Tb)$ by definition, so the last term is precisely $Q^2\left\langle\Dbf\log\tfrac{|f'|}{|f|}\big|_\Tb,\log\tfrac{|f'|}{|f|}\big|_\Tb\right\rangle$.

It remains to show that
\begin{equation}\label{det_formula}
\frac{\det'\Delta_g}{\det'\Delta_{g|_{f(\Db)}}}=e^{-\frac{1}{6}W(f,g_f)}.
\end{equation}
By \cite[Lemma 4.3]{BCILT},
\[\textstyle\det'\Delta_g = 2\pi \det'\Delta_{g|_{f(\Db)}}\det'\Delta_{g_f}\det_{\text{Fr}}(\Dbf_{\Db,f(\Tb)}(2\Dbf_0)^{-1}),\]
where the last term is a certain Fredholm determinant. The calculation at the end of the proof of \cite[Proposition 3.4]{Segal_semigroup} shows that
\[2\pi\,\frac{\det'\Delta_{g_\Ab}}{e^{\frac{1}{4\pi}\int_{f(\Tb)}k_{g_\Ab}\d\ell_{g_\Ab}}} \det\nolimits_{\text{Fr}}(\Dbf_{\Db,f(\Tb)}(2\Dbf_0)^{-1}) = e^{\frac{1}{6}\log|f'(0)|+\frac{1}{6}\left\langle\Dbf\log\frac{|f'|}{|f|}\big|_\Tb,\log\frac{|f'|}{|f|}\big|_\Tb\right\rangle}.\]
By the Polyakov formula (see \cite[Proposition 5.10]{Boundary_Segal} for example),
\[\det\nolimits'\Delta_{g_f} = \frac{\det'\Delta_{g_\Ab}}{e^{\frac{1}{4\pi}\int_{f(\Tb)}k_{g_\Ab}\d\ell_{g_\Ab}}}\,e^{-2S_{\mathrm L}^0(\Ab_f,g_f,g_\Ab)}.\]
Thus \eqref{det_formula} is proved.
\end{proof}


Since the amplitude operator $\Ac_{\Ab_f,g_f,\zetabf_f}$ is bounded (\cite[Section 8.7]{CILT}, \cite[Theorem 7.5]{BCILT}), we obtain immediately:

\begin{coro}
The operator $\Tbf_f$ extends to a bounded operator on $L^2(\DT)$.
\end{coro}

As an application of the gluing axiom of amplitudes, we can prove:

\begin{prop}\label{T_f_semigroup}
The map $f\in\Ssg\mapsto\Tbf_f\in\Lc(L^2(\DT))$ is a semigroup homomorphism, i.e., $\Tbf_{f_1}\circ\Tbf_{f_2}=\Tbf_{f_1\circ f_2}$ for $f_1,f_2\in\Ssg$.
\end{prop}
\begin{proof}
By the gluing axiom of CILT amplitudes (\cite[Proposition 8.13]{CILT}, \cite[Theorem 7.8]{BCILT}), it suffices to show that the functional $W$ is additive in the sense that
\begin{equation}\label{W_additivity}
W(f_1\circ f_2,g_{f_1\circ f_2}) = W(f_1,g_{f_1}) + W(f_2,g_{f_2}),
\end{equation}
where $g_{f_1\circ f_2}$ is an admissible metric on $\Ab_{f_1\circ f_2}=\Ab_{f_1}\# f_1(\Ab_{f_2})$ obtained by gluing an admissible metric $g_{f_1}$ on $\Ab_{f_1}$ and an admissible metric $(f_1)_*g_{f_2}$ on $f_1(\Ab_{f_2})$. Additivity of the first term $\log|f'(0)|$ in $W$ is clear. Recall from the previous proof that
\[E(f,g_f) := \left\langle\Dbf\log\tfrac{|f'|}{|f|}\big|_\Tb,\log\tfrac{|f'|}{|f|}\big|_\Tb\right\rangle + 12S_{\mathrm L}^0(\Ab_f,g_f,g_\Ab) = \frac{1}{2\pi}\int_{\Ab_f}\chi_f\,\Delta_g\chi_f\,\d v_g,\]
where $g$ is any conformal metric on $\Db$, and we define $h_f,\chi_f$ as in the previous proof for $f\in\{f_1,f_2,f_1\circ f_2\}$. We have
\[e^{-2\chi_{f_2}}g_\Ab=g_{f_2}=f_1^*g_{f_1\circ f_2}=f_1^*(e^{-2\chi_{f_1\circ f_2}}g_\Ab)=e^{-2\chi_{f_1\circ f_2}\circ f_1}f_1^*g_\Ab=e^{-2\chi_{f_1\circ f_2}\circ f_1}e^{2h_{f_1}\circ f_1}g_\Ab,\]
so
\[\chi_{f_1\circ f_2}\circ f_1 = h_{f_1}\circ f_1 + \chi_{f_2}\quad\text{ on }\Ab_{f_2}.\]
Similarly,
\[e^{2h_{f_1\circ f_2}\circ f_1\circ f_2}g_\Ab=(f_1\circ f_2)^*g_\Ab=f_2^*(e^{2h_{f_1}\circ f_1}g_\Ab)=e^{2h_{f_1}\circ f_1\circ f_2}e^{2h_{f_2}\circ f_2}g_\Ab,\]
so 
\[h_{f_1\circ f_2}\circ f_1 = h_{f_1}\circ f_1 + h_{f_2}\quad\text{ on }f_2(\Db).\]
Thus the identity $\chi_{f_1\circ f_2}\circ f_1 = h_{f_1}\circ f_1 + \chi_{f_2}$ holds on $\Db$. Clearly $\chi_{f_1\circ f_2} = \chi_{f_1}$ on $\Ab_{f_1}$, so
\begin{align*}
E(f_1\circ f_2,g_{f_1\circ f_2}) - E(f_1,g_{f_1}) &= \frac{1}{2\pi}\int_{f_1(\Ab_{f_2})}\chi_{f_1\circ f_2}\,\Delta_g\chi_{f_1\circ f_2}\,\d v_g = \frac{1}{2\pi}\int_{\Db}(\chi_{f_1\circ f_2}\circ f_1)\,\Delta_g(\chi_{f_1\circ f_2}\circ f_1)\,\d v_g \\
&= \frac{1}{2\pi}\int_{\Db}\chi_{f_2}\,\Delta_g\chi_{f_2}\,\d v_g=E(f_2,g_{f_2}),
\end{align*}
where we used that $h_{f_1}\circ f_1$ is harmonic.
\end{proof}

\section{Virasoro generators and the Feynman--Kac formula}\label{sec_Virasoro}

In this section, we prove the Feynman--Kac formula for the operator $\Tbf_f$, showing that the infinitesimal generators of the quantized semigroup $f\mapsto\Tbf_f$ form two commuting Virasoro algebras of central charge $c_\L=1-6Q^2$.

We shall write a holomorphic vector field as $\v=v(z)\pa_z$ where $v$ is the corresponding holomorphic function in the local trivialization. Following \cite[Section 2.4]{Virasoro_structure}, a holomorphic vector field $\v=v(z)\pa_z$ on $\Db$ is called \textit{Markovian} if $v(0)=0$ and $\Re(\zb v(z))<0$ for $z\in\Tb$. These conditions ensure that the flow $t\mapsto f_t:\Db\to\Db$ generated by $\v$, obtained by solving the differential equation
\[\pa_tf_t(z)=v(f_t(z)),\qquad f_0(z)=z,\qquad z\in\Db,\]
exists for all $t\geq0$ and defines a semigroup homomorphism $\Rb_{>0}\to\Ssg$ with $f_0=\id_\Db\not\in\Ssg$. With the convention that $\Tbf_{\id_\Db}:=\id_{L^2(\DT)}$, we thus obtain a one-parameter semigroup of bounded operators $t\mapsto \Tbf_{f_t}$. Its generator is given by the Segal--Sugawara construction as follows.

We define the following differential operators on $\Ec(\DT)$:
\begin{equation}
\partial_0:=\partial_c,
\qquad \partial_n:=\partial_{\varphi_n}=\sqrt n(\partial_{x_n}-\i \partial_{y_n}),
\qquad\partial_{-n}:=\partial_{\varphi_{-n}}=\sqrt n(\partial_{x_n}+\i \partial_{y_n}),
\qquad n>0,
\end{equation}
\begin{equation}\label{eq_def_A}
\begin{aligned}
\Abf_0:=\tfrac{\i}{2}(\pa_c+\i Q+\i Rk),
&\quad&\mathbf A_n:=\tfrac\i2(\partial_n+2n\varphi_{-n}\textbf{1}_{n<0}),\\
\wt{\Abf}_0:=\tfrac{\i}{2}(\pa_c+\i Q-\i Rk),
&\quad&\widetilde{\mathbf A}_n:=\tfrac\i2(\partial_{-n}+2n\varphi_{n}\textbf{1}_{n<0}),
\end{aligned}
\qquad
\begin{aligned}
n\neq0,
\end{aligned}
\end{equation}
where $k$ is the $\Zb$-component (i.e., winding number) of an element $\wt{\varphi}^k\in\DT$.
We have the formal adjoint relations
\begin{equation}
\Abf_n^*=\Abf_{-n},\qquad\wt\Abf_n^*=\wt\Abf_{-n},\qquad n\in\Zb,
\end{equation}
which hold in the sense that $\la\Abf_n F,G\ra_{L^2}=\la F,\Abf_{-n}G\ra_{L^2}$ for $F,G\in\Ec(\DT)$, and similarly for $\wt\Abf_n$. Furthermore, one can verify the Heisenberg commutation relations
\begin{equation*}
    [\mathbf A_n,\mathbf A_m]=\frac{n}{2}\mathbf 1_{n+m=0},\ \ \ [\widetilde{\mathbf A}_n,\widetilde{\mathbf A}_m]=\frac{n}{2}\mathbf 1_{n+m=0},\ \ \ [{\mathbf A}_n,\widetilde{\mathbf A}_m]=0.
\end{equation*}

Fix $\mu\in\Cb$ as a parameter of CILT. For $n\in\Zb$, we define
\begin{equation}\label{eq_L}
\begin{aligned}
\Lbf_n &:= (n+1)Q\Abf_n + \sum_{m\in{\mathbb Z}}:\!\Abf_{n-m}\Abf_m\!: +\; \frac{\mu}{2}\, e^{\i\beta c}\int_0^{2\pi}e^{\i\beta\varphi(\theta)}\,e^{\i\beta Rk\theta}\,e^{\i n\theta}\,\d\theta,\\
\wt{\Lbf}_n &:= (n+1)Q
\wt{\Abf}_n + \sum_{m\in{\mathbb Z}}:\!\wt{\Abf}_{n-m}\wt{\Abf}_m\!: +\; \frac{\mu}{2}\, e^{\i\beta c}\int_0^{2\pi}e^{\i\beta\varphi(\theta)}\,e^{\i\beta Rk\theta}\,e^{-\i n\theta}\,\d\theta,
\end{aligned}
\end{equation}
where $e^{\i\beta\varphi(\theta)}\,\d\theta$ is understood as the imaginary GMC on $\Tb$. Here we have used the normal ordered product $:\!\Abf_m\Abf_n\!:$ defined to be $\Abf_m\Abf_n$ if $n>0$ or $\Abf_n\Abf_m$ otherwise, and similarly for $:\!\wt{\Abf}_{n-m}\wt{\Abf}_m\!:$. In the case $\mu=0$, we denote these operators by $\Lbf_n^0$ and $\wt{\Lbf}_n^0$, which are well-defined differential operators on $\Ec(\DT)$ and satisfy the commutation relations (pointwise on $\Ec(\DT)$)
\begin{equation}\label{Virasoro_commutation_relation}
[\Lbf_m^0,\wt{\Lbf}_n^0]=0,\qquad [\Lbf_m^0,\Lbf_n^0]=(m-n)\Lbf_{m+n}^0+\frac{1-6Q^2}{12}(m^3-m)\,{\bf 1}_{m+n=0},
\end{equation}
and similarly for $[\wt{\Lbf}_m^0,\wt{\Lbf}_n^0]$. We also have the formal adjoint relations
\begin{equation}
(\Lbf_n^{0,\beta})^*=\Lbf_{-n}^{0,-\beta},\qquad(\wt\Lbf_n^{0,\beta})^*=\wt\Lbf_{-n}^{0,-\beta},\qquad n\in\Zb,
\end{equation}
where we have made explicit the dependence on $\beta$ (and on $Q$ through $Q=\beta/2-2/\beta$). 
The general case $\mu\in\Cb$ will be defined rigorously below. 

For a holomorphic vector field $\v = -\sum_{n=0}^\infty v_n z^{n+1} \pa_z$, we define
\begin{equation}\label{H_v}
\Hbf_\v:=\sum_{n=0}^\infty v_n\,\Lbf_n+\sum_{n=0}^\infty\ol{v_n}\,\wt{\Lbf}_n.
\end{equation}
In the case $\v=-z\pa_z$, the operator
\begin{equation}\label{Hamiltonian}
\Hbf:=\Hbf_{-z\pa_z}=\Lbf_0+\wt{\Lbf}_0
\end{equation}
is the \textit{Hamiltonian} of CILT. As above, we have the corresponding operators $\Hbf^0,\Hbf^0_\v$ for the case $\mu=0$. Note that
\begin{equation}\label{Pi}
\Pibf:=\tfrac{1}{\i}\Hbf_{-\i z\pa_z}=\Lbf_0-\wt{\Lbf}_0=-\i Rk\pa_c+\sum_{m\geq1}m(\varphi_m\partial_{m}-\varphi_{-m}\partial_{-m})
\end{equation}
contains a nontrivial term in $k$, exhibiting the existence of spin, i.e., CILT is a \textit{nondiagonal} CFT. The goal of this section is to show that for $\Re\omega>0$ (writing $\omega:=v_0=-v'(0)\in\Cb$) large enough, the operator $\Hbf_\v$ defined on a suitable domain is a closed unbounded operator generating a strongly continuous semigroup $t\mapsto e^{-t\Hbf_\v}$, and that this coincides with the semigroup $t\mapsto\Tbf_{f_t}$ where $f_t$ is the flow generated by $\v$.

The following lemma summarizes the functional-analytic situation encountered in this section. Following the terminology of \cite[Section VIII.6]{ReedSimon}, a \textit{quadratic form} on a Hilbert space is a densely defined sesquilinear form.

\begin{lemma}\label{HilleYosida_lemma}
Let $(\Hc,\la\cdot,\cdot\ra_{\Hc})$ be a complex Hilbert space and $\Qc_0:\Dc(\Qc_0)\times \Dc(\Qc_0)\to\Cb$ a closed positive symmetric quadratic form on $\Hc$ with dense domain $\Dc(\Qc_0)\subset\Hc$. Let $\Qc:\Dc(\Qc_0)\times\Dc(\Qc_0)\to\Cb$ be a (nonsymmetric) quadratic form defined on the domain $\Dc(\Qc_0)$ of $\Qc_0$. Suppose that there exists $C>0$ such that
\begin{equation}\label{sectorial_property}
\Re\Qc(u,u)\geq C^{-1}\Qc_0(u,u),\qquad |\Qc(u,u)|\leq C\Qc_0(u,u),\qquad \fr u\in\Dc(\Qc_0).
\end{equation}
Then the following hold:
\begin{enumerate}[label=\emph{(\arabic*)}]
    \item There exists a unique closed operator $\Hbf:\Dc(\Hbf)\subset\Hc\to\Hc$ with domain $$\Dc(\Hbf)=\{u\in\Dc(\Qc_0):\ol{\Qc(u,\cdot)}\text{ is bounded }\Hc\to\Cb\},$$ such that $\Qc(u,v)=\la \Hbf u,v\ra_{\Hc}$ for $u\in\Dc(\Hbf)$, $v\in\Dc(\Qc_0)$. Moreover, $\Hbf$ extends to a bounded operator $\Dc(\Qc_0)\to\Dc'(\Qc_0)$, where $\Dc'(\Qc_0)\supset\Hc$ is the dual space of $\Dc(\Qc_0)$ equipped with the norm $\|u\|_{\Dc(\Qc_0)}:=\sqrt{\Qc_0(u,u)}$.
    \item The operator $\Hbf$ generates a strongly continuous contraction semigroup $t\mapsto e^{-t\Hbf}$ ($t\geq0$) on $\Hc$. Moreover, for $t>0$, the bounded operator $e^{-t\Hbf}$ on $\Hc$ extends to a bounded operator $\Dc'(\Qc_0)\to\Dc(\Qc_0)$.
\end{enumerate}
\end{lemma}
\begin{proof}
The quadratic form $\Qc$ is strictly $m$-accretive in the sense defined in \cite[Section VIII.6]{ReedSimon}. Thus we obtain the operator $\Hbf$ which coincides with the adjoint of the operator associated with $\Qc$ as in Theorem VIII.16 (\textit{ibid.}). Using the polarization identity and the second bound in \eqref{sectorial_property}, there exists $C'>0$ such that $|\Qc(u,v)|\leq C'\sqrt{\Qc_0(u,u)\Qc_0(v,v)}$ for $u,v\in\Dc(\Qc_0)$. This shows that $\Hbf$ is bounded as an operator $\Dc(\Qc_0)\to\Dc'(\Qc_0)$.

By the lemma following Theorem VIII.16 (\textit{ibid.}), we have that any $\lambda\in\Cb$ with $\Re\lambda<0$ is in the resolvent set of $\Hbf$ and $\|(\Hbf-\lambda)^{-1}\|_{\Hc\to\Hc}\leq(-\Re\lambda)^{-1}$. By the Hille--Yosida theorem and \cite[Theorem 2.5.2.(d)]{Pazy}, $\Hbf$ generates an analytic contraction semigroup $t\mapsto e^{-t\Hbf}$ and we have $\|\Hbf e^{-t\Hbf}\|_{\Hc\to\Hc}\leq C_1/t$ ($t>0$) for some uniform $C_1>0$. In particular, $\Hbf e^{-t\Hbf}\in\Lc(\Hc)$. Thus for $u\in\Hc$, we have $\|e^{-t\Hbf}u\|_{\Dc(\Qc_0)}^2=\Qc_0(e^{-t\Hbf}u,e^{-t\Hbf}u)\leq C|\Qc(e^{-t\Hbf}u,e^{-t\Hbf}u)|=C|\la \Hbf e^{-t\Hbf}u,e^{-t\Hbf}u\ra_{\Hc}|\leq C\|\Hbf e^{-t\Hbf}u\|_\Hc\|e^{-t\Hbf}u\|_\Hc\leq C_2\|u\|_\Hc^2$ with $C_2:=C\|\Hbf e^{-t\Hbf}\|_{\Hc\to\Hc}\|e^{-t\Hbf}\|_{\Hc\to\Hc}<\infty$. In other words, $e^{-t\Hbf}$ is bounded as an operator $\Hc\to\Dc(\Qc_0)$. By duality (for $\Hbf^*$), it extends to a bounded operator $\Dc'(\Qc_0)\to\Hc$. Thus for $t>0$, $e^{-t\Hbf}=e^{-\frac{t}{2}\Hbf}e^{-\frac{t}{2}\Hbf}:\Dc'(\Qc_0)\to\Hc\to\Dc(\Qc_0)$ is bounded.
\end{proof}

\subsection{The free case}
We first treat the case of the free field with $\mu=0$. 
Consider the positive symmetric quadratic form $\wt{\Qc}^0$ given by
\begin{equation}\label{Qtilde0}
\wt{\Qc}^0(F,F) :=  \sum_{k\in\Zb} \left(\frac{1}{2}\|\pa_cF\|_k^2+\frac{R^2k^2+Q^2}2\|F\|_k^2+2\sum_{n\geq1}\|\Abf_nF\|_k^2+2\sum_{n\geq1}\|\wt{\Abf}_nF\|_k^2\right),
\end{equation}
where $\|\!\cdot\!\|_k$ is the $L^2$ norm on $L^2(\TR\times\Dc_0'(\Tb,\Rb)\times\{k\})$, i.e., $\|F\|_k^2 := \int_{\TR} \Eb_\varphi\big[|F(\wt{\varphi}^k)|^2\big]\,\d c$ where the expectation is over the mean-zero GFF $\varphi$ on $\Tb$. This will play the role of the reference quadratic form $\Qc_0$ in \cref{HilleYosida_lemma}. First let us check that it is closable.

\begin{lemma}
The quadratic form $\wt{\Qc}^0$ is closable. We identify it with its closure and denote by $\Dc(\wt{\Qc}^0)$ its domain.
\end{lemma}
\begin{proof}
Clearly $\wt{\Qc}^0$ is well-defined on $\Ec(\DT)$. Let $(F_n)_{n=1}^\infty\subset\Ec(\DT)$ be a sequence such that $\wt{\Qc}^0(F_m-F_n,F_m-F_n)\to0$ as $m,n\to\infty$. Since $\wt{\Qc}^0\geq\frac{\min(1,Q^2)}2\|\!\cdot\!\|_{L^2}^2$, $(F_n)_n$ is Cauchy in $L^2(\DT)$. Suppose $\|F_n-F_0\|_{L^2}\to0$ for some $F_0\in L^2(\DT)$. To show that $\wt{\Qc}^0$ is closable, it suffices to show that $F_0=0$ implies $\wt{\Qc}^0(F_n,F_n)\to0$. Consider the Hilbert space $\Hc_0:=L^2(\TR\times\Dc_0'(\Tb,\Rb))$ with the norm $\|F\|_{\Hc_0}^2:=\int_{\TR}\Eb[|F(c+\varphi)|^2]\,\d c$ (the expectation being over the mean-zero GFF $\varphi$ on $\Tb$) and the product Hilbert space $\Hc_0^\infty:=\bigoplus_{k\in\Zb}(\Hc_0\oplus\Hc_0\oplus\Hc_0^{\Nb_{\geq1}}\oplus\Hc_0^{\Nb_{\geq1}})$ with the norm
\[\|(F_k^1,F_k^2,(F_{kn}^3)_{n\geq1},(F_{kn}^4)_{n\geq1})_{k\in\Zb}\|_{\Hc_0^\infty}^2:=\sum_{k\in\Zb}\left(\frac{1}{2}\|F_k^1\|_{\Hc_0}^2+\frac{R^2k^2+Q^2}{2}\|F_k^2\|_{\Hc_0}^2+2\sum_{n\geq1}\|F_{kn}^3\|_{\Hc_0}^2+2\sum_{n\geq1}\|F_{kn}^4\|_{\Hc_0}^2\right).\]
Then we have the linear map
\[\iota:F\mapsto(\pa_cF_k,F_k,(\Abf_n F_k)_{n\geq1},(\wt{\Abf}_n F_k)_{n\geq1})_{k\in\Zb},\qquad F_k(c+\varphi):=F(c+\varphi+kR\arg),\]
such that $\wt{\Qc}^0(F,F)=\|\iota F\|_{\Hc_0^\infty}^2$. The condition $\wt{\Qc}^0(F_m-F_n,F_m-F_n)\to0$ implies that $(\iota F_m)_m$ is Cauchy in $\Hc_0^\infty$, so we can assume that $\iota F_m$ converges to some $(F_k^1,F_k^2,(F_{kn}^3)_{n\geq1},(F_{kn}^4)_{n\geq1})_{k\in\Zb}\in\Hc_0^\infty$. In particular, $\Abf_n F_m\to(F_{kn}^3)_{k\in\Zb}$ in $L^2(\DT)\cong\oplus_{k\in\Zb}\Hc_0$. For any $G\in\Ec(\DT)$, we have
\[\la(F_{kn}^3)_{k\in\Zb},G\ra=\lim_{m\to\infty}\la \Abf_n F_m,G\ra=\lim_{m\to\infty}\la F_m,\Abf_{-n}G\ra=\la F_0,\Abf_{-n}G\ra=0,\]
so $F_{kn}^3=0$ for all $k,n$. Similarly, one shows that $F_k^1=F_k^2=F_{kn}^4=0$ for all $k,n$. Thus $\iota F_m$ converges to $0$ in $\Hc_0^\infty$, i.e., $\wt{\Qc}^0(F_m,F_m)=\|\iota F_m\|_{\Hc_0^\infty}^2\to0$.
\end{proof}

\begin{lemma}\label{L_n^0_estimate}
There exists a constant $C>0$ (depending only on $Q$ and $R$) such that for all $n\in\Zb$ and all $F,G\in\Ec(\DT)$, we have
\begin{equation*}
    |\langle\mathbf L^0_nF,G\rangle_{L^2}|\leq C(1+|n|)\sqrt{\wt{\Qc}^0(F,F)\wt{\Qc}^0(G,G)}.
\end{equation*}
In particular, $\mathbf L_n^0$ extends to a bounded operator $\Dc(\wt{\Qc}^0)\to\Dc'(\wt{\Qc}^0)$. The same holds for $\widetilde{\mathbf L}_n^0$.
\end{lemma}
\begin{proof}
Since $(\mathbf L_n^{0,\beta})^*=\mathbf L_{-n}^{0,-\beta}$ and $\mathbf H^0=\mathbf L^{0}_0+\widetilde{\mathbf L}^0_0$, it suffices to consider the case $n> 0$.
 Recall \eqref{Qtilde0}, \eqref{eq_L} and that $\Abf_0=\tfrac{\i}{2}\pa_c- \frac12Q- \frac12Rk$. By Cauchy-Schwarz inequalities, we have
    \begin{equation*}
    \begin{split}
        |\langle\mathbf L^0_nF,G\rangle_2|={}&\left|(n+1)Q\langle\mathbf{A}_nF,G\rangle_2+2\langle\mathbf A_0\mathbf{A}_nF,G\rangle_2+\sum_{j=1}^{n-1}\langle\mathbf{A}_{n-j}F,\mathbf{A}_{-j}G\rangle_2+2\sum_{l\geq1}\langle\mathbf A_{n+l}F,\mathbf A_{l}G\rangle_2\right|\\
        \leq{}&\left(\ \|\mathbf A_nF\|_2^2+\|\partial_cF\|_2^2+\sum_{j=1}^{n-1}\|\mathbf A_j F\|^2_2+2\sum_{l\geq1}\|\mathbf A_{n+l}F\|_2^2\right)^{\frac12}\\
        &\times\left(\|(nQ-Rk)G\|_2^2+\|\mathbf A_{-n}G\|_2^2+\sum_{j=1}^{n-1}\|\mathbf A_{-j} G\|^2_2+2\sum_{l\geq1}\|\mathbf A_{l}G\|_2^2\right)^{\frac12}.
    \end{split}       
    \end{equation*}
  The second line above is bounded by $C\sqrt{\wt{\Qc}^0(F,F)}$. Using $[\mathbf{A}_j,\mathbf{A}_{-j}]=\frac j2$, we obtain $\|A_{-j} G\|^2_2=\|A_{j} G\|^2_2+\frac j2\|G\|_2^2$. Thus
  \begin{equation*}
      \begin{split}
          &\left(\|(nQ-Rk)G\|_2^2+\|\mathbf A_{-n}G\|_2^2+\sum_{j=1}^{n-1}\|\mathbf A_{-j} G\|^2_2+2\sum_{l\geq1}\|\mathbf A_{l}G\|_2^2\right)^{\frac12}\\
          ={}&\left(\left(|nQ-Rk|^2+\frac{n(n+1)}{4}\right)\|G\|_2^2+\sum_{j=1}^{n}\|\mathbf A_{j} G\|^2_2+2\sum_{l\geq1}\|\mathbf A_{l}G\|_2^2\right)^{\frac12},
      \end{split}
  \end{equation*}
  where the last line is bounded by $Cn\sqrt{\wt{\Qc}^0(G,G)}$.
\end{proof}


Let $\v = -\sum_{n=0}^\infty v_n z^{n+1} \pa_z$ be a Markovian holomorphic vector field on $\Db$. This lemma shows that the expression $\Hbf_\v^0=\sum_{n\geq0}v_n\,\Lbf_n^0+\sum_{n\geq0}\ol{v_n}\,\wt{\Lbf}_n^0$ converges in the space of bounded linear operators $\Dc(\wt{\Qc}^0)\to\Dc'(\wt{\Qc}^0)$. We define the associated (nonsymmetric) quadratic form $\Qc_\v^0(F,F):=\la \Hbf_\v^0F,F\ra_{L^2}$ on $\Dc(\wt{\Qc}^0)$. Then it is easy to see that $\wt{\Qc}^0=\Qc^0+Q^2\|\!\cdot\!\|_{L^2}^2$ where $\Qc^0:=\Qc_{-z\pa_z}^0$ is the special case $\v=-z\pa_z$. Using the lemma again as in the beginning of the proof of \cite[Theorem 2.8]{Virasoro_structure}, we obtain bounds (for some $K>0$)
\begin{equation}\label{Q_v^0_accretive}
\Re(\Qc_\v^0+\omega Q^2\|\!\cdot\!\|_{L^2}^2)\geq\left(\Re\omega-2K\sum_{n\geq1}n^2|v_n|\right)\wt{\Qc}^0,\qquad|\Qc_\v^0+\omega Q^2\|\!\cdot\!\|_{L^2}^2|\leq\left(2|\omega|+2K\sum_{n\geq1}n^2|v_n|\right)\wt{\Qc}^0.
\end{equation}
Assuming that $\Re\omega>0$ is large enough in the sense that
\begin{equation}\label{omega_condition}
\Re\omega > |\!\Im\omega| + 2K\sum_{n\geq1}n^2|v_n| + 1,
\end{equation}
we are in the situation described in \cref{HilleYosida_lemma}. The conclusion is that there exists a closed unbounded operator $\Hbf_\v^0$ associated to $\Qc_\v^0$ with domain $\Dc(\Hbf_\v^0)\subset\Dc(\wt{\Qc}^0)$, and it generates a strongly continuous semigroup of bounded operators $t\mapsto e^{-t\Hbf_\v^0}$. Its extension to a bounded operator $\Dc(\wt{\Qc}^0)\to\Dc'(\wt{\Qc}^0)$ coincides with the limit of $\sum_{n\geq0}v_n\,\Lbf_n^0+\sum_{n\geq0}\ol{v_n}\,\wt{\Lbf}_n^0$ above.

\begin{prop}[\textbf{Feynman--Kac formula, free case}]\label{prop_FK_free}
Let $\v = -\sum_{n=0}^\infty v_n z^{n+1} \pa_z$ be a Markovian holomorphic vector field on $\Db$ such that \eqref{omega_condition} holds. Let $t\mapsto f_t$ be the flow generated by $\v$. For $t>0$, we have $\Tbf_{f_t}^0 = e^{-t\Hbf_\v^0}$, i.e.,
\[(e^{-t\Hbf_\v^0}F)(\wt{\varphi}^k) = |f_t'(0)|^{\frac{R^2k^2-Q^2}2} \,\Eb\left[F\left((P\wt{\varphi}^k+X_\Db)\circ f_t\big|_\Tb + \i Q\log\tfrac{|f_t'|}{|f_t|}\big|_\Tb\right)\right].\]
\end{prop}
\begin{proof}
For $k\in\Zb$, let $\Hc_k:=L^2(\TR\times\Dc_0'(\Tb,\Rb)\times\{k\})\subset L^2(\DT)$, so that $L^2(\DT)=\bigoplus_{k\in\Zb}\Hc_k$. Clearly all the operators above preserve each $\Hc_k$.

First consider the case $k=0$. To show that $\Tbf_{f_t}^0 = e^{-t\Hbf_\v^0}$ on $\Hc_0$, it suffices to repeat the end of the proof of \cite[Theorem 2.8]{Virasoro_structure} (more precisely, Propositions 2.11 and 2.12), \textit{mutatis mutandis}. Indeed, the calculation of $\frac{\pa}{\pa t}\big|_{t=0}F\big(P\wt{\varphi}\circ f_t\big|_\Tb + Q\log\tfrac{|f_t'|}{|f_t|}\big|_\Tb\big)$ in Proposition 2.11 (\textit{ibid.}) holds for all $Q\in\Cb$ if $F$ is analytic in the zero mode $c$, and the calculation of $\frac{\pa}{\pa t}\big|_{t=0}\Eb\big[F\big(\wt{\varphi} + (X_\Db\circ f_t)\big|_\Tb\big)\big]$ in Proposition 2.12 (\textit{ibid.}) carries over verbatim since they concern the same GFF $X_\Db$.

For a general $k\in\Zb$, we have $\Tbf_{f_t}^0|_{\Hc_k} = |f_t'(0)|^{\frac{R^2k^2}{2}}\, \Tbf_{f_t}^0|_{\Hc_0} \circ \tau_{k,t}$, where we identify $\Hc_k$ and $\Hc_0$ via the natural isomorphism and $\tau_{k,t}$ is the translation on $\Hc_0$ defined by $(\tau_{k,t}F)(\wt{\varphi}) := F(\wt{\varphi} + Rkg_t|_\Tb)$, where $g_t(z):=\arg f_t(z)-\arg z=\Im\log\frac{f_t(z)}{z}$. Note that $g_t$ admits a continuous branch on $\Db$ and is well-defined up to a global constant in $2\pi\Zb$. Thus the generator of $\Tbf_{f_t}^0|_{\Hc_k}$ is
\[-\frac{R^2k^2}{2}\Re\omega+\Hbf_\v^0|_{\Hc_0} - \pa_t\big|_{t=0}\tau_{k,t}.\]
To compute the last term,
\[\pa_t\big|_{t=0}g_t(z) = \Im\frac{\pa_tf_t}{f_t}\bigg|_{t=0}(z) = \Im \frac{v(z)}{z} = -\Im\sum_{n=0}^\infty v_nz^n,\]
so
\[\pa_t\big|_{t=0}\tau_{k,t} = Rk\left(-\Im\omega\,\pa_c+\sum_{n\geq1}\frac{\i}{2}(v_n\pa_n-\ol{v_n}\pa_{-n})\right).\]
The rest of the proof is a straightforward algebraic calculation.
\end{proof}

\subsection{The Liouville case}\label{Sec_Virasoro_Liouville}

We turn to the general case $\mu\in\Cb$. Let $\v = -\sum_{n\geq0} v_n z^{n+1} \pa_z$ be a holomorphic vector field on $\Db$. We would like to make sense of the potential operator
\begin{equation}\label{defn_V_v}
V_\v:=e^{\i\beta c}\int_0^{2\pi} e^{\i\beta Rk\theta}\,e^{\i\beta\varphi(\theta)}\,\d\varrho_\v(\theta),\qquad \d\varrho_\v(\theta):=-\Re(e^{-\i\theta}v(e^{\i\theta}))\,\d\theta,
\end{equation}
where $e^{\i\beta\varphi(\theta)}\,\d\theta$ is understood as the imaginary GMC on $\Tb$. When $0<|\beta|<1$, $V_\v$ exists as a function and $\mathbb E[e^{\mu |V_\v|}]$ is finite, so that it acts as a multiplication operator on $L^2(\DT)$. However, when $1\leq\beta^2<2$, it is not properly defined and regularization is needed to make sense of it. For $0<\delta<1$, we define
\begin{equation*}
   V_\v^{(\delta)}:= e^{\i\beta c}\int_0^{2\pi} e^{\i\beta Rk\theta}\,e^{\i\beta\varphi^{\text{arc}}_\delta(\theta)+\frac{\beta^2}{2}\mathbb E[(\varphi^{\text{arc}}_\delta)^2(\theta)]}\,\d\varrho_\v(\theta),
\end{equation*}
where $\varphi^{\text{arc}}_\delta(\theta):=\frac{1}{2\delta}\int_{\theta-\delta}^{\theta+\delta}\varphi(\theta')\d\theta'$ is a.s.\ a continuous function on $\mathbb T$. Using the imaginary Cameron--Martin theorem (the proof is similar to those in  \cite[Theorem 6.11]{CILT}), we have for real analytic $L^2$ functions $F,G$ (in an appropriate space such as \eqref{eq_def_E_space})
\begin{equation*}
\langle V_\v^{(\delta)}F,G\rangle_{L^2}=\int_0^{2\pi R}e^{\i\beta c}\int_0^{2\pi} e^{\i\beta Rk\theta}\mathbb E[F(\widetilde{\varphi}+\i\beta\mathbb E[\varphi(\cdot)\varphi^{\text{arc}}_\delta(\theta)])\overline{G}(\widetilde{\varphi}+\i\beta\mathbb E[\varphi(\cdot)\varphi^{\text{arc}}_\delta(\theta)])\,\d\varrho_\v(\theta)\d c.
\end{equation*}

Then the operator $V_\v$ is defined by
\[\la V_\v F,G\ra_{L^2}:=\lim_{\delta\to0}\langle V_\v^{(\delta)}F,G\rangle_{L^2}\]
for $F,G\in\Ec(\DT)$.
\begin{remark}\label{rmk_tech_GMC}
    Let us explain the technical reasons why we choose this regularization instead of the usual harmonic cutoff that will be used in the paper about conformal bootstrap of CILT. In the proof of Feynman--Kac formula \cref{prop_FK_formula}, we need to find a uniform bound for the exponential moment of the regularized GMC like our \cref{lemma_gmc_delta_exp_bound} (in the real case, the potential is positive so they are free from this problem). If we use white noise decomposition as in \cref{lemma_gmc_delta_exp_bound} for the harmonic cutoff, we need to estimate the Fourier coefficient of the heat kernel accurately. Besides, the harmonic cutoff of $X\circ f_t$ can not produce a martingale, so we cannot use Jensen's inequality to finish the proof, unless we go through a detailed study of the Fourier coefficient of $\log|f_s(e^{\i\theta})-f_{s'}(e^{\i\theta'})|$. On the other hand, not every regularization is compatible with the Hilbert space $L^2(\DT)$. Note that $V_\v^{(\delta)}$ is measurable in $L^2(\DT)$, while the white noise decomposition and circle average are not. This is crucial in applying the Trotter's product formula in \cref{prop_FK_formula}. 
\end{remark}

\begin{lemma}\label{V_v_estimate}
Let $N\geq0$. For $F,G\in\Ec(\DT)$, we have
           \begin{equation}\label{eq_V_bound}
           \begin{split}
               \sup_{0<\delta<1}|\langle V^{(\delta)}_\v F,G\rangle_{L^2}|\leq {}&C\, (N+Q^2/2)^{\frac{\beta^2}{2}-1} \\
               &\sup_{\theta\in[0,2\pi]}|\Re(e^{-\i\theta}v(e^{\i\theta}))|\sqrt{(N\|F\|_{L^2}^2+\widetilde{\mathcal Q}^0(F,F))(N\|G\|_{L^2}^2+\widetilde{\mathcal Q}^0(G,G))},
           \end{split}
    \end{equation}
where $C>0$ only depends on $\beta$. In particular, $V_\v$ extends to a bounded operator $\Dc(\wt{\Qc}^0)\to\Dc'(\wt{\Qc}^0)$.
\end{lemma}
The main idea of the proof belongs to another paper studying conformal bootstrap of CILT (by C. Guillarmou, R. Rhodes and the authors), which is still in preparation. We put the proof in the Appendix, i.e. the proof of \cref{lem_App_1}. In short, we adapt the proof based on cut-off of harmonics to arc average renormalization here, and some basic Green's function estimates are greatly helped by AI.

In the same vein, for $n\in\Zb$, the operator
\[V_n:=e^{\i\beta c}\int_0^{2\pi} e^{\i\beta Rk\theta}\,e^{\i\beta\varphi(\theta)}\,e^{\i n\theta}\,\d\theta\]
extends to a bounded operator $\Dc(\wt{\Qc}^0)\to\Dc'(\wt{\Qc}^0)$, and we can define $\Lbf_n:=\Lbf_n^0+\frac{\mu}{2} V_n$, $\wt{\Lbf}_n:=\wt{\Lbf}_n^0+\frac{\mu}{2}V_{-n}$. One can check the formal adjoint relations (making explicit the dependence on $\mu,\beta$)
\begin{equation}\label{L_n_adjoint}
(\Lbf_n^{\mu,\beta})^*=\Lbf_{-n}^{\bar{\mu},-\beta},\qquad(\wt\Lbf_n^{\mu,\beta})^*=\wt\Lbf_{-n}^{\bar{\mu},-\beta},\qquad n\in\Zb,
\end{equation}
which hold in the sense that
\[\la \Lbf_n^{\mu,\beta}F, G\ra=\la F,\Lbf_{-n}^{\bar{\mu},-\beta}G\ra,\qquad \la \wt\Lbf_n^{\mu,\beta}F, G\ra=\la F,\wt\Lbf_{-n}^{\bar{\mu},-\beta}G\ra,\qquad\fr F,G\in\Dc(\wt{\Qc}^0).\]
It is more delicate to make sense of the commutation relations for $(\Lbf_n)_n,(\wt{\Lbf}_n)_n$; a precise statement will be formulated in a future work.

For $\v = -\sum_{n=0}^\infty v_n z^{n+1} \pa_z$ as above, the expression $\Hbf_\v=\sum_{n\geq0}v_n\,\Lbf_n+\sum_{n\geq0}\ol{v_n}\,\wt{\Lbf}_n$ converges in the space of bounded linear operators $\Dc(\wt{\Qc}^0)\to\Dc'(\wt{\Qc}^0)$, and the limit is precisely $\Hbf_\v^0+\mu V_\v$. We define the associated quadratic form $\Qc_\v$ by
\begin{equation*}
    \Qc_\v(F,G):= \la\Hbf_\v F,G\ra_{L^2} = \Qc_\v^0(F,G)+\mu\lim_{\delta\to0}\langle V_\v^{(\delta)}F,G\rangle_{L^2}.
\end{equation*}
To apply \cref{HilleYosida_lemma}, we need the following bounds: 

\begin{lemma}\label{Q_v_bound}
Suppose \eqref{omega_condition} holds. Then there exist $N,C>0$ (depending on $\v$) such that
\[\Re(\Qc_\v+N\|\!\cdot\!\|_2^2)\geq C^{-1}\wt{\Qc}^0,\qquad|\Qc_\v+N\|\!\cdot\!\|_2^2|\leq C\wt{\Qc}^0.\]
\end{lemma}
\begin{proof}
Note that $\sup_{\theta\in[0,2\pi]}|\Re(e^{-\i\theta}v(e^{\i\theta}))|\leq\sum_{n=0}^\infty|v_n|=|\omega|+\sum_{n\geq1}|v_n|$. Taking $N$ large enough in \eqref{eq_V_bound}, we have
\[|\mu \langle V_\v F,F\rangle_{L^2}|\leq\delta\,(N\|F\|_{L^2}^2+\widetilde{\mathcal Q}^0(F,F)),\]
where we can take $\delta$ arbitrarily small, say, $\delta=\min(\frac{1}{2},\frac{1}{2}\sup_{\theta\in[0,2\pi]}|\Re(e^{-\i\theta}v(e^{\i\theta}))|)$. Combining with \eqref{Q_v^0_accretive}, we obtain
\begin{gather*}\label{eq_Q_v_bound}
  \Re(\Qc_\v+N\|\!\cdot\!\|_{L^2}^2)\geq\left(\Re\omega-2K\sum_{n\geq1}n^2|v_n|-\delta\right)\widetilde{\mathcal Q}^0 + (N-\Re\omega\,Q^2-\delta N)\|\!\cdot\!\|_{L^2}^2,\\
|\Qc_\v+N\|\!\cdot\!\|_{L^2}^2|\leq\left(|\omega|+2K\sum_{n\geq1}n^2|v_n|+\delta\right)\widetilde{\mathcal Q}^0+(N+|\omega|Q^2+\delta N)\|\!\cdot\!\|_{L^2}^2.
\end{gather*}
Since $\wt{\Qc}^0\geq\frac{\min(1,Q^2)}2\|\!\cdot\!\|_{L^2}^2$, this proves the lemma taking $N>2\Re\omega\,Q^2$.
\end{proof}

Thus we obtain a closed unbounded operator $\Hbf_\v$ associated to $\Qc_\v$ with domain $\Dc(\Hbf_\v)\subset\Dc(\wt{\Qc}^0)$, and it generates a strongly continuous semigroup of bounded operators $t\mapsto e^{-t\Hbf_\v}$. Its extension to a bounded operator $\Dc(\wt{\Qc}^0)\to\Dc'(\wt{\Qc}^0)$ coincides with the limit of $\sum_{n\geq0}v_n\,\Lbf_n+\sum_{n\geq0}\ol{v_n}\,\wt{\Lbf}_n$ above.
\begin{prop}[\textbf{Feynman--Kac formula}]\label{prop_FK_formula}
Let $\v = -\sum_{n=0}^\infty v_n z^{n+1} \pa_z$ be a Markovian holomorphic vector field on $\Db$ such that \eqref{omega_condition} holds. Let $t\mapsto f_t$ be the flow generated by $\v$. For $t>0$, we have $\Tbf_{f_t} = e^{-t\Hbf_\v}$, i.e.,
\[(e^{-t\Hbf_\v}F)(\wt{\varphi}^k) = |f_t'(0)|^{\frac{R^2k^2-Q^2}2} \,\Eb_{\widetilde{\varphi}^k}\left[F\left((P\wt{\varphi}^k+X_\Db)\circ f_t\big|_\Tb + \i Q\log\tfrac{|f_t'|}{|f_t|}\big|_\Tb\right)\,e^{-\mu\int_{\Ab_{f_t}}|x|^{\beta Q}M_\beta(P\wt{\varphi}^k+X_\Db,\d x)}\right].\]
\end{prop}
\begin{proof}
We follow the proof in \cite[Theorem 3.1]{Virasoro_structure} and \cite[Section 5.2]{Guillarmou2024}. Denote by $X:=c + P\varphi + kR\arg+X_\Db$ the GFF restricted to the disk and $X_\epsilon:=\frac1{2\pi}\int_0^{2\pi}X(x+\epsilon e^{\i\theta}) \d\theta$ its regularisation. Pulling back by the diffeomorphism $[0,t]\times[0,2\pi]\mapsto\Ab_{f_t}, (s,\theta)\mapsto f_s(e^{\i\theta})$, we get by change of variables
\begin{equation}\label{eq_FK_1}
    \int_{\Ab_{f_t}}|x|^{\beta Q} \epsilon^{-\beta^2/2}e^{\i\beta X_\epsilon(x)}\ \d x=-\int_{0}^t\int_0^{2\pi} \epsilon^{-\beta^2/2}|f_s(e^{\i\theta})|^{\beta Q}e^{\i\beta X_\epsilon(f_s(e^{\i\theta}))}\Re\left(e^{-\i\theta}v(f_s(e^{\i\theta}))\overline{f_s'(e^{\i\theta})}\right)\ \d\theta\d s
\end{equation}
When $\epsilon\to0$, as a consequence of the uniqueness of GMC, we have
\begin{equation*}
    \epsilon^{-\beta^2/2}e^{\i\beta X_\epsilon(f_s(e^{\i\theta}))}\d\theta\d s\approx \epsilon_{s,\theta}^{-\beta^2/2} e^{\i\beta (X\circ f_s)_{\epsilon_{s,\theta}}(e^{\i\theta})}|f'_s(e^{\i\theta})|^{-\frac{\beta^2}{2}}\d\theta\d s,
\end{equation*}
where $\epsilon_{s,\theta}=\epsilon/|f'_s(e^{\i\theta})|$. Hence, the right hand side of \eqref{eq_FK_1} can be written as
\begin{equation*}
    -\int_{0}^t\int_0^{2\pi} \epsilon_{s,\theta}^{-\beta^2/2} e^{\i\beta ((X\circ f_s)_{\epsilon_{s,\theta}}(e^{\i\theta})+\i Q\log\frac{|f'_s(e^{\i\theta})|}{|f_s(e^{\i\theta})|})}\Re\left(\frac{e^{-\i\theta}v(f_s(e^{\i\theta}))}{f'_s(e^{\i\theta})}\right)\ \d\theta\d s.
\end{equation*}
By changing the cutoff $\epsilon_{s,\theta}$ and using \cite[Lemma 2.4]{Virasoro_structure}, we have
\begin{equation}\label{eq_Aft2V}
     \int_{\Ab_{f_t}}|x|^{\beta Q} e^{\i\beta X(x)}\ \d x=-\lim_{\epsilon\to0} \int_{0}^t\int_0^{2\pi} \epsilon^{-\beta^2/2} e^{\i\beta ((X\circ f_s)_{\epsilon}(e^{\i\theta})+\i Q\log\frac{|f'_s(e^{\i\theta})|}{|f_s(e^{\i\theta})|})}\Re\left(e^{-\i\theta}v(e^{\i\theta})\right)\ \d\theta\d s.
\end{equation}

\begin{lemma}
    For $F,G\in\Ec(\DT) $, $\lim_{t\to0}t^{-1}\langle F-\Tbf_{f_t}F,G\rangle_{L^2}=\langle \mathbf H_\v F,G\rangle_{L^2}$.
\end{lemma}
\begin{proof}
By definition of $\Ec(\DT)$, we can assume $F,G$ depend on a finite number of Fourier modes of GFF in terms of polynomials. Similar to the proof in \cite[Lemma 5.4]{Guillarmou2024}, and using the definition of the semigroup $\Tbf_{f_t}$ \eqref{eq_def_T_f}, $\langle F-\Tbf_{f_t}F,G\rangle_{L^2}$ is just
  \begin{equation}\label{eq_diff_T_f}
   \int_{0}^{2\pi R}\d c\ \Eb\left[\left(F(\wt{\varphi}^k)-|f_t'(0)|^{\frac{R^2k^2-Q^2}2} F\left((P\wt{\varphi}^k+X_\Db)\circ f_t\big|_\Tb + \i Q\log\tfrac{|f_t'|}{|f_t|}\big|_\Tb\right)\,e^{-\mu\int_{\Ab_{f_t}}|x|^{\beta Q}M_\beta(P\wt{\varphi}^k+X_\Db,\d x)}\right)\,\overline{G}(\wt{\varphi}^k)\right].
  \end{equation}
In the following proof, denote by $X$ the field on the disk $P\wt{\varphi}^k+X_\Db$. Recall that the GFF part $P{\varphi}+X_\Db$ has covariance $-\log|x-x'|,\forall x,x'\in\mathbb D$. Since $F,G\in\Ec(\DT) $, we have $\mathbb E[|F(\wt{\varphi}^k)G(\wt{\varphi}^k)|^n]<\infty$ for arbitrary $n$.
  
According to \cref{prop_FK_free}, \eqref{eq_diff_T_f} is further written as
  \begin{equation*}
  \begin{split}
   &t\Qc_\v^0(F,G)+o(t)+\int_{0}^{2\pi R}\d c\ \Eb\left[F(\wt{\varphi}^k)\,\overline{G}(\wt{\varphi}^k)\,\left(1-e^{-\mu\int_{\Ab_{f_t}}|x|^{\beta Q}M_\beta(X,\d x)}\right)\right]+\\
      &\int_{0}^{2\pi R}\d c\ \Eb\left[\left(F(\wt{\varphi}^k)-|f_t'(0)|^{\frac{R^2k^2-Q^2}2}F\left(X\circ f_t\big|_\Tb + \i Q\log\tfrac{|f_t'|}{|f_t|}\big|_\Tb\right)\right)\,\overline{G}(\wt{\varphi}^k)\,\left(e^{-\mu\int_{\Ab_{f_t}}|x|^{\beta Q}M_\beta(X,\d x)}-1\right)\right]\\
      &=:t\Qc_\v^0(F,G)+o(t)+D_t^{(1)}+D_t^{(2)}.
  \end{split}
  \end{equation*}
In the following we will prove $t^{-1}D_t^{(1)}\to \mu\langle V_\v F,G\rangle$ and $t^{-1}D_t^{(2)}\to 0$.

For $D_t^{(1)}$, we expand the exponential of GMC then we have
  \begin{equation*}
  \begin{split}
      D_t^{(1)}={}&-\int_{0}^{2\pi R}\d c\ \Eb\left[F(\wt{\varphi}^k)\,\overline{G}(\wt{\varphi}^k)\,\sum_{n\geq1}\frac{(-\mu)^n}{n!}\left(\int_{\Ab_{f_t}}|x|^{\beta Q}M_\beta(X,\d x)\right)^n\right].
  \end{split}
  \end{equation*}
  Now we apply the imaginary Cameron--Martin theorem to each term in the summation for a single vertex insertion, and we get
  \begin{equation*}
      \begin{split}
      D_t^{(1)}={}&-\int_{0}^{2\pi R}\int_{\Ab_{f_t}}e^{\i\beta c+\i\beta Rk\arg y}\\ &\ \ \ \ \ \ \ \ \ \ \ \ \ \ \Eb\left[F_{\text{shift}}(y,\wt{\varphi}^k)\,(\overline{G})_{\text{shift}}(y,\wt{\varphi}^k)\,\sum_{n\geq1}\frac{(-\mu)^n}{n!}\left(\int_{\Ab_{f_t}}|x|^{\beta Q}|x-y|^{\beta^2}M_\beta(X,\d x)\right)^{n-1}\right]|y|^{\beta Q}\,\d y\,\d c\\
      ={}& \mu\int_{0}^{2\pi R}\int_{\Ab_{f_t}}e^{\i\beta c+\i\beta Rk\arg y}\ \Eb\left[F_{\text{shift}}(y,\wt{\varphi}^k)\,(\overline{G})_{\text{shift}}(y,\wt{\varphi}^k)\right]|y|^{\beta Q}\,\d y\,\d c-\int_{0}^{2\pi R}\int_{\Ab_{f_t}}e^{\i\beta c+\i\beta Rk\arg y}\\ &\ \  \ \ \ \ \ \ \ \ \ \Eb\left[F_{\text{shift}}(y,\wt{\varphi}^k)\,(\overline{G})_{\text{shift}}(y,\wt{\varphi}^k)\,\sum_{n\geq2}\frac{(-\mu)^n}{n!}\left(\int_{\Ab_{f_t}}|x|^{\beta Q}|x-y|^{\beta^2}M_\beta(X,\d x)\right)^{n-1}\right]|y|^{\beta Q}\,\d y\,\d c,
  \end{split}
  \end{equation*}
where
\begin{equation*}
    F_{\text{shift}}(y,\wt{\varphi}^k):=F(\wt{\varphi}^k+\i\beta K_y),\ \ \ K_{r e^{\i\theta'}}(\theta):=\sum_{n\geq1}\frac{|r|^n}{n}\cos(n(\theta-\theta')).
\end{equation*}

For $t$ small enough, $\Ab_{f_t}$ are always the subset of some annulus $\{z:q<|z|<1\},q<1$. Define
\begin{equation*}
    I_t(y):=\left|\Eb\left[F_{\text{shift}}(y,\wt{\varphi}^k)\,(\overline{G})_{\text{shift}}(y,\wt{\varphi}^k)\,\sum_{n\geq2}\frac{(-\mu)^n}{n!}\left(\int_{\Ab_{f_t}}|x|^{\beta Q}|x-y|^{\beta^2}M_\beta(X,\d x)\right)^{n-1}\right]\right|,
\end{equation*}
and we claim that
\begin{equation}\label{eq_FK_claim}
    \lim_{t\to0}\sup_{q<|y|<1}   I_t(y)=0,
\end{equation}
which would imply $t^{-1}D_t^{(1)}-\mu\langle V_\v F,G\rangle\approx t^{-1}\int_{\Ab_{f_t}}  I_t(y)|y|^{\beta Q}\d y\to 0$.
To prove \eqref{eq_FK_claim}, we first apply the Cauchy-Schwarz inequality and obtain
\begin{equation*}
\begin{split}
    &I_t^2(y)\leq\Eb[|\int_{\Ab_{f_t}}|x|^{\beta Q}|x-y|^{\beta^2}M_\beta(X,\d x)|^{2}]\times\\
    &\Eb\left[\left|F_{\text{shift}}(y,\wt{\varphi}^k)\,(\overline{G})_{\text{shift}}(y,\wt{\varphi}^k)\,\sum_{n\geq2}\frac{(-\mu)^n}{n!}\left(\int_{\Ab_{f_t}}|x|^{\beta Q}|x-y|^{\beta^2}M_\beta(X,\d x)\right)^{n-2}\right|^2\right].
\end{split}
\end{equation*}
Since $|\sum_{n\geq2}\frac{(-\mu)^n}{n!}z^{n-2}|\leq|\mu|^2\sum_{n\geq0}\frac{|\mu z|^n}{(n+2)!}\leq |\mu|^2e^{|\mu z|}$, the second term above is further bounded, using Cauchy-Schwarz inequality, up to a constant depending on $\mu$ by
\begin{equation*}
    \Eb[|F_{\text{shift}}(\overline{G})_{\text{shift}}|^4]^{\frac12}\Eb[\exp(4|\mu||\int_{\Ab_{f_t}}|x|^{\beta Q}|x-y|^{\beta^2}M_\beta(X,\d x)|)]^{\frac12}.
\end{equation*}
For some constant $C$ independent of $t$ small enough, we have $\sup_{q<|y|<1}\mathbb E[|F_{\text{shift}}(\overline{G})_{\text{shift}}|^4]<C$ (since the Fourier modes of $K_y(\theta)$ are under control). According to \cite[Proposition 5.3]{CILT}, for some constant $C$ only depending on $\beta$, the exponential moment is bounded by
\begin{equation*}
    e^{C|\mu|}(1+2|\mu| u e^{C|\mu|^2 u^2}),
\end{equation*}
where
\begin{equation}\label{eq_metric_integral}
    u^2:=\iint_{(\Ab_{f_t})^2}|x|^{\beta Q}|x-y|^{\beta^2}|x'|^{\beta Q}|x'-y|^{\beta^2}|x-x'|^{-\beta^2}\ \d x\, \d x'.
\end{equation}
Hence $\Eb[\exp(C|\int_{\Ab_{f_t}}|x|^{\beta Q}|x-y|^{\beta^2}M_\beta(X,\d x)|)]$ is finite and uniformly bounded in $q<|y|<1$ and $t$ small enough. Then the study of $I_t^2$ boils down to the second moment
\begin{equation*}
    \Eb[|\int_{\Ab_{f_t}}|x|^{\beta Q}|x-y|^{\beta^2}M_\beta(X,\d x)|^{2}]\leq\iint_{(\Ab_{f_t})^2}|x|^{\beta Q}|x-y|^{\beta^2}|x'|^{\beta Q}|x'-y|^{\beta^2}|x-x'|^{-\beta^2}\ \d x\, \d x',
\end{equation*}
which goes to $0$ uniformly in $y$ as $t\to0$. 

Therefore \eqref{eq_FK_claim} is proved and we just have
 \begin{equation*}
      \begin{split}
      \lim_{t\to0}t^{-1}D_t^{(1)} &=   \lim_{t\to0}t^{-1}\mu\int_{0}^{2\pi R}\int_{\Ab_{f_t}}e^{\i\beta c+\i\beta Rk\arg y}\ \Eb\left[F_{\text{shift}}(y,\wt{\varphi}^k)\,(\overline{G})_{\text{shift}}(y,\wt{\varphi}^k)\right]|y|^{\beta Q}\,\d y\,\d c\\
      &=-\lim_{t\to0}t^{-1}\mu\int_{0}^{2\pi R}\int_0^t\int_0^{2\pi}e^{\i\beta c+\i\beta Rk\arg f_s(e^{\i\theta})}\ \Eb\left[F_{\text{shift}}(f_s(e^{\i\theta}),\wt{\varphi}^k)\,(\overline{G})_{\text{shift}}(f_s(e^{\i\theta}),\wt{\varphi}^k)\right]\\
      &\ \ \ \ \ \ \ \ \ \  \ \ \ \ \ \ \ \ \ \ \  \ \ \ \ \ \ \ \ \ |f_s(e^{\i\theta})|^{\beta Q}|f_s'(e^{\i\theta})|^2\Re\left(e^{-\i\theta}v(e^{\i\theta})\right)\, \d\theta\d s\ \d c\\
      &=-\mu\int_{0}^{2\pi R}\int_0^{2\pi}e^{\i\beta c+\i\beta Rk\theta}\ \Eb\left[F_{\text{shift}}(e^{\i\theta},\wt{\varphi}^k)\,(\overline{G})_{\text{shift}}(e^{\i\theta},\wt{\varphi}^k)\Re\left(e^{-\i\theta}v(e^{\i\theta})\right)\right]\ \d\theta\, \d c,
  \end{split}
  \end{equation*}
  where we used \eqref{eq_Aft2V} for the second line.

  Finally, we can focus on how $t^{-1}D_t^{(2)}$ goes to $0$. We can replace $F_{\text{shift}}$ in \eqref{eq_FK_claim} with
  \begin{equation*}
      F_{\text{shift}}(y,\wt{\varphi}^k)-|f_t'(0)|^{\frac{R^2k^2-Q^2}2}F\left(X\circ f_t\big|_\Tb +\i\beta K_{t,y}+ \i Q\log\tfrac{|f_t'|}{|f_t|}\big|_\Tb\right),\ \ K_{t,y}(\theta):=-\log|f_t(e^{\i\theta})-y|
  \end{equation*}as long as the second moment is still uniformly bounded in $y$ and $t$. Indeed, we have the free field semigroup $\|F\left(X\circ f_t\big|_\Tb +\i\beta K_{t,y}+ \i Q\log\tfrac{|f_t'|}{|f_t|}\big|_\Tb\right)\|_{L^2}\leq Ce^{Ct}\|F_{\text{shift}}\left(y,\wt{\varphi}^k\right)\|_{L^2}$ for some constant independent of $y$ and $t$. Since $f_t\to\id$, we have $\sup_{y\in\Ab_{f_t}}|\langle K_{t,y}-K_y,h\rangle|\to0$ for every $h\in C^\infty(\Tb)$. Thus
  \begin{equation*}
      \begin{split}
      \lim_{t\to0}t^{-1}D_t^{(2)} = -&  \lim_{t\to0}t^{-1}\mu\int_{0}^{2\pi R}\int_{\Ab_{f_t}}e^{\i\beta c+\i\beta Rk\arg y}\ \\
      &\Eb\left[\left(F_{\text{shift}}(y,\wt{\varphi}^k)-|f_t'(0)|^{\frac{R^2k^2-Q^2}2}F\left(X\circ f_t\big|_\Tb +\i\beta K_{t,y}+ \i Q\log\tfrac{|f_t'|}{|f_t|}\big|_\Tb\right)\right)\,(\overline{G})_{\text{shift}}(y,\wt{\varphi}^k)\right]|y|^{\beta Q}\, \d y\,\d c,
  \end{split}
  \end{equation*}
  where we can use Cauchy Schwarz inequality and \cite[Lemma 2.6, equation (2.27)]{Virasoro_structure} to conclude that $\lim_{t\to0}t^{-1}D_t^{(2)}=0$ and the proof is complete. 
\end{proof}
The above lemma says the quadratic form from the semigroup $\mathbf T_{f_t}$ is equal to $\Qc_\v$ for $F,G\in\Ec(\DT)$. Then we just need to prove the resolvents of $\mathbf T_{f_t}$ and $\Qc_\v$ are the same.

Define \begin{equation*}
    \Qc_\v^{(\delta)}(F,G):=\Qc_\v^0(F,G)+\mu\langle V_\v^{(\delta)}F,G\rangle_{L^2}.
\end{equation*}
Since $V_\v^{(\delta)}$ is bounded when $\delta$ is fixed, we can define the operator $\mathbf H^{(\delta)}_{\v}:=\mathbf H^0_{\v}+\mu V_\v^{(\delta)}$ associated with $ \Qc_\v^{(\delta)}(F,G)$. Then using the Trotter product formula \cite[Exercise 5.11 in Chapter III]{EN} and the fact that $V_\v^{(\delta)}$ is bounded, one can obtain the following formula by following the proof of \cite[Proposition 5.3]{Guillarmou2024}
\begin{equation}\label{eq_FK_cutoff}
    e^{-t\mathbf H^{(\delta)}_{\v}}F=|f_t'(0)|^{\frac{R^2k^2-Q^2}2} \,\Eb_{\wt{\varphi}^k}\left[F\left((P\wt{\varphi}^k+X_\Db)\circ f_t\big|_\Tb + \i Q\log\tfrac{|f_t'|}{|f_t|}\big|_\Tb\right)\,e^{-\mu\int_0^t V^{(\delta)}_\v((P\wt{\varphi}^k+X_\Db)\circ f_s + \i Q\log\tfrac{|f_s'|}{|f_s|})\d s}\right].
\end{equation}

We follow the strategy in \cite[Proposition 5.5]{Guillarmou2024}. First we have the representation of the resolvent
\begin{equation*}
    \mathbf R_{\v,\lambda}F:=\int_0^{\infty} e^{-\lambda t}\mathbf T_{f_t}F\ \d t,\ \ \ \ \mathbf R_{\v,\lambda}^{(\delta)}F:=\int_0^{\infty} e^{-\lambda t}e^{-t\mathbf H^{(\delta)}_{\v}}F\ \d t.
\end{equation*}
Note that $\|\mathbf T_{f_t}F\|\leq e^{Ct}\|F\|$ (see \cref{lemma_gmc_delta_exp_bound} below) so $\mathbf R_{\v,\lambda}F$ is well-defined if we consider $\lambda$ large enough. Now we want to show that $\mathbf R_{\v,\lambda}^{(\delta)}F$ converges to $\mathbf R_{\v,\lambda}F$ in $L^2(\DT)$. Using Cauchy-Schwarz inequality, we have
\begin{equation*}
    \int_{0}^{2\pi R}\mathbb E|\mathbf R_{\v,\lambda}F-\mathbf R_{\v,\lambda}^{(\delta)}F|^2\ \d c\leq C\int_{0}^{2\pi R}\int_0^{\infty} e^{-\lambda t}\ \mathbb E|(e^{-t\mathbf H_\v }-e^{-t\mathbf H^{(\delta)}_{\v}})F|^2\ \d t\d c.
\end{equation*}
Then by Jensen's inequality, one can obtain
\begin{equation}\label{eq_diff_GMC_delta}
\begin{split}
    &\mathbb E|(\mathbf T_{f_t}-e^{-t\mathbf H^{(\delta)}_{\v}})F|^2\\
    \leq{}& \left| |f_t'(0)|^{\frac{R^2k^2-Q^2}2}\right|^2\mathbb E\Bigg[ |F|^2\left(X\circ f_t\big|_\Tb + \i Q\log\tfrac{|f_t'|}{|f_t|}\big|_\Tb\right)\\
    &\ \ \ \ \ \ \ \ \ \ \ \ \ \ \ \ \ \ \ \ \left|e^{-\mu\int_0^t V_\v^{}(X\circ f_s\big|_\Tb + \i Q\log\tfrac{|f_s'|}{|f_s|}\big|_\Tb)\d s}-e^{-\mu\int_0^t V_\v^{(\delta)}(X\circ f_s\big|_\Tb + \i Q\log\tfrac{|f_s'|}{|f_s|}\big|_\Tb)\d s}\right|^2\Bigg].
\end{split}
\end{equation}
We can apply the Cauchy-Schwarz inequality again to separate $F$ and GMC exponential moment with the $F$ part controlled by $e^{Ct}\|F^2\|_{L^2}^2$. Thus we only need to show that the GMC exponential moment part converges and is bounded by a function of $t$ (in fact $e^{Ct},C>0$) independent of $\delta$. Since for all $t>0$, we have the $L^2$ convergence of GMC \cite{Lacoin2015}, the convergence of the resolvent is a consequence of the following lemma. 

Before the lemma, recall in \cite[Lemma 2.3]{Virasoro_structure}, as $s\to\infty$, $e^{\omega s}f_s(e^{\i\theta})$ and $e^{\omega s}f_s'(e^{\i\theta})$ converge uniformly in $\theta\in[0,2\pi]$ to some function $f_*(e^{\i\theta})$ and $f'_*(e^{\i\theta})$, with $f_*$ biholomorphic in $(1+\epsilon')\mathbb D$ for some $\epsilon'>0$.
\begin{lemma}\label{lemma_gmc_delta_exp_bound}
    For the GFF on the disk $X$, and a measurable function $h(\theta,s)$, there exists a constant $C$ depending on $\beta$ and $(v(z))_{z\in\mathbb D}$ only, such that when $\delta<1$,
    \begin{equation*}
        \mathbb E\left[\left|e^{-\mu \int_0^T \int_0^{2\pi}h(\theta,s)e^{\i\beta(X\circ f_s)^{\text{arc}}_\delta(e^{\i\theta})+\frac{\beta^2}{2}\mathbb E[(\varphi^{\text{arc}}_\delta)^2(\theta)]}\,\d\varrho_\v(\theta)\d s}\right|\right]< e^{C|\mu|}(1+2|\mu| u_\delta(T) e^{C|\mu|^2 u^2_{\delta}(T)}),
    \end{equation*}
    where
    \begin{equation}\label{eq_u_and_k}
    \begin{split}
        u^2_\delta(T):={}&\delta^{-2}\iint_{([0,T]\times[0,2\pi])^2}\int_{-\delta}^{\delta}\int_{-\delta}^{\delta}|h(\theta_1+\theta_1',s_1)||h(\theta_2+\theta_2',s_2)|\,\d \theta_1'\d\theta_2'\\
        &\ \ \ \ \ \ \ \ \  \ \ \ \ \ \ \ \ \ e^{-\frac{\beta^2\Re\omega}{2}(s_1+s_2) }|f_{s_1}(e^{\i\theta_1})-f_{s_2}(e^{\i\theta_2})|^{-\beta^2}\,\d s_1\d\varrho_\v(\theta_1)\d s_2\d\varrho_\v(\theta_2).
    \end{split}
    \end{equation}
\end{lemma}
\begin{proof}
    Without loss of generality, we can consider the decomposition of GFF on a disk of radius $\mathcal R>1$ into a Dirichlet GFF and a harmonic extension, hence we only need to focus on the Dirichlet GFF on $\mathcal R\mathbb D$ restricted to $\mathbb D$ and still denote it by $X$. 

Let $p_{t/2}(x,y)$ be the Dirichlet heat kernel on $\mathcal R\mathbb D$ and $W(\d y,\d t)$ be a white noise on $\mathcal R\mathbb D\times[0,\infty)$.   Using the white noise decomposition, we have
    \begin{equation*}
        (X\circ f_s)_{\delta,\epsilon}^{arc}(e^{\i\theta}):=\frac{\sqrt{2\pi}}{2\delta}\int_{-\delta}^{\delta}\int_{\epsilon^2}^\infty\int_{|y|<\mathcal R}p_{t/2}(f_s(e^{\i(\theta+\theta')}),y)W(\d y,\d t)\d \theta',
    \end{equation*}
    and $(e^{\i \beta (X\circ f_s)_{\delta,\epsilon}^{arc}(e^{\i\theta})+\frac{\beta^2}2\mathbb E[(X\circ f_s)_{\delta,\epsilon}^{arc}(e^{\i\theta})^2]})_{\epsilon>0}$ becomes a martingale. Define
    \begin{equation*}
        k_\delta(\theta,s):=h(\theta,s)\exp\left(\frac{\beta^2}{2}\left(\mathbb E[(\varphi^{\text{arc}}_\delta)^2(\theta)]-\mathbb E[(X\circ f_s)_{\delta}^{arc}(e^{\i\theta})^2]\right)\right).
    \end{equation*}
    Following the proof of \cite[Proposition 5.3]{CILT}, we compute the quadratic variation of  the martingale
    \begin{equation*}
        \mathcal M_{\epsilon}^{(\delta)}:=\int_0^{T}\int_0^{2\pi} k_\delta(\theta,s)\,e^{\i \beta (X\circ f_s)_{\delta,\epsilon}^{arc}(e^{\i\theta})+\frac{\beta^2}2\mathbb E[(X\circ f_s)_{\delta,\epsilon}^{arc}(e^{\i\theta})^2]}\ \d\varrho_\v(\theta)\d s,
    \end{equation*}
    which is
    \begin{equation*}
    \begin{split}
          &\langle \mathcal M^{(\delta)}\rangle_t-\langle \mathcal M^{(\delta)}\rangle_0\\
          ={}&\beta^2\int_0^t\iint_{([0,T]\times[0,2\pi])^2}k_\delta(\theta_1,s_1)k_\delta(\theta_2,s_2)\sin(\beta (X\circ f_{s_1})_{\delta,e^{-u}}^{arc}(e^{\i\theta_1}))\sin(\beta(X\circ f_{s_2})_{\delta,e^{-u}}^{arc}(e^{\i\theta_2}))\\
          &\ \ \  \ \  e^{\frac{\beta^2}2(\mathbb E[(X\circ f_{s_1})_{\delta,\epsilon}^{arc}(e^{\i\theta_1})^2]+\mathbb E[(X\circ f_{s_2})_{\delta,\epsilon}^{arc}(e^{\i\theta_2})^2])}\ \d s_1\d\varrho_\v(\theta_1)\d s_2\d\varrho_\v(\theta_2)\  \d\langle (X\circ f_{s_1})_{\delta,e^{-u}}^{arc}(e^{\i\theta_1}),(X\circ f_{s_2})_{\delta,e^{-u}}^{arc}(e^{\i\theta_2})\rangle_u,
    \end{split}
    \end{equation*}
    where for convenience we have assumed $\Im(h)=0$ and $\mu\in\mathbb R$. Next, the quadratic variation of the smooth field is
    \begin{equation*}
        \d\langle (X\circ f_{s_1})_{\delta,e^{-u}}^{arc}(e^{\i\theta_1}),(X\circ f_{s_2})_{\delta,e^{-u}}^{arc}(e^{\i\theta_2})\rangle_u=\frac{\pi e^{-2u}}{\delta^2}\int_{-\delta}^{\delta}\int_{-\delta}^{\delta}p_{e^{-2u}}(f_{s_1}(e^{\i(\theta_1+\theta_1')}),f_{s_2}(e^{\i(\theta_2+\theta_2')}))\ \d \theta_1'\d \theta_2'\d u.
    \end{equation*}
    Using the bound for the heat kernel \cite{Cheng1981},
    \begin{equation*}
        e^{-2u}p_{e^{-2u}}(f_{s_1}(e^{\i(\theta_1+\theta_1')}),f_{s_2}(e^{\i(\theta_2+\theta_2')}))\leq C e^{-{|f_{s_1}(e^{\i(\theta_1+\theta_1')})-f_{s_2}(e^{\i(\theta_2+\theta_2')})|^2e^{2u}}/{C}}
    \end{equation*}
    for some constant $C$, we can again bound the quadratic variation using the integral
 \begin{equation}\label{eq_M_bound}
    \begin{split}
          &\sup_{t\geq0}\langle \mathcal M^{(\delta)}\rangle_t-\langle \mathcal M^{(\delta)}\rangle_0\\
          \leq{}&\frac{C\beta^2}{\delta^2}\int_0^\infty\iint_{([0,T]\times[0,2\pi])^2}|k_\delta(s_1,\theta_1)||k_\delta(s_2,\theta_2)|e^{\frac{\beta^2}2(\mathbb E[(X\circ f_{s_1})_{\delta,e^{-u}}^{arc}(e^{\i\theta_1})^2]+\mathbb E[(X\circ f_{s_2})_{\delta,e^{-u}}^{arc}(e^{\i\theta_2})^2])}\\
          &\ \ \ \ \ \ \ \ \ \ \ \ \ \ \ \int_{-\delta}^{\delta}\int_{-\delta}^{\delta}e^{-|f_{s_1}(e^{\i(\theta_1+\theta_1')})-f_{s_2}(e^{\i(\theta_2+\theta_2')})|^2e^{2u}/{C}}\ \d \theta_1'\d\theta_2'\,\d s_1\d\varrho_\v(\theta_1)\d s_2\d\varrho_\v(\theta_2)\, \d u.
    \end{split}
    \end{equation}
When $t$ is small, we have the free heat kernel diagonal bound $ p_t(x,y)\leq\frac{1}{4\pi t}$, while for $t$ large $p_t$ decays exponentially fast in $t$. Thus we have when $u$ large
\begin{equation*}
    \mathbb E[ (X\circ f_{s_1})_{\delta,e^{-u}}^{arc}(e^{\i\theta_1})^2]=\frac{\pi}{2\delta^2}\int_{-\delta}^{\delta}\int_{-\delta}^{\delta}\int_{e^{-2u}}^\infty p_t(f_{s_1}(e^{\i(\theta_1+\theta')}),f_{s_1}(e^{\i(\theta_1+\theta'')}))\ \d t\ \d \theta'\d \theta''\leq C+u,
\end{equation*}
where $C$ is independent of $\delta,s_1,\theta_1$ hence we can replace the variance in \eqref{eq_M_bound} with $u$ uniformly. Finally, integrating out the white noise cut-off $\d u$, we are left with
\begin{equation*}
    \frac{C\beta^2}{\delta^2}\iint_{([0,T]\times[0,2\pi])^2}|k_\delta(s_1,\theta_1)||k_\delta(s_2,\theta_2)|\int_{-\delta}^{\delta}\int_{-\delta}^{\delta}|f_{s_1}(e^{\i(\theta_1+\theta_1')})-f_{s_2}(e^{\i(\theta_2+\theta_2')})|^{-\beta^2}\d \theta_1'\d\theta_2'\,\d s_1\d\varrho_\v(\theta_1)\d s_2\d\varrho_\v(\theta_2),
\end{equation*}
which is bounded by
\begin{equation}\label{eq_temp_in_GMC_bound}
    \frac{C\beta^2}{\delta^2}\iint_{([0,T]\times[0,2\pi])^2}\int_{-\delta}^{\delta}\int_{-\delta}^{\delta}|k_\delta(s_1,\theta_1+\theta_1')||k_\delta(s_2,\theta_2+\theta_2')|\,\d \theta_1'\d\theta_2'\,|f_{s_1}(e^{\i\theta_1})-f_{s_2}(e^{\i\theta_2})|^{-\beta^2}\,\d s_1\d\varrho_\v(\theta_1)\d s_2\d\varrho_\v(\theta_2).
\end{equation}
Besides, by Jensen's inequality,
\begin{equation*}
\begin{split}
    \exp\left(\frac{\beta^2}{2}(\mathbb E[(\varphi^{\text{arc}}_\delta)^2(\theta)]-\mathbb E[(X\circ f_s)_{\delta}^{arc}(e^{\i\theta})^2])\right)={}&\exp\left(\frac{\beta^2}{8\delta^2}\int_{-\delta}^\delta\int_{-\delta}^\delta \log\left|\frac {f_s(e^{\i(\theta+\theta_1)})-f_s(e^{\i(\theta+\theta_2)})}{e^{\i(\theta+\theta_1)}-e^{\i(\theta+\theta_2)}}\right|\d\theta_1\d\theta_2\right)\\
    \leq{}&\frac{1}{4\delta^2}\int_{-\delta}^\delta\int_{-\delta}^\delta \exp\left(\frac{\beta^2}{2}\log\left|\frac {f_s(e^{\i(\theta+\theta_1)})-f_s(e^{\i(\theta+\theta_2)})}{e^{\i(\theta+\theta_1)}-e^{\i(\theta+\theta_2)}}\right|\right)\d\theta_1\d\theta_2,
\end{split}
\end{equation*}
Since $f_s(z)$ is holomorphic on $\mathbb T$, we have 
\begin{equation*}
      |k_{\delta}(s,\theta)|\leq  |h(s,\theta)|\sup_{\theta'\in[0,2\pi]}|f'_s(e^{\i\theta'})|^{\frac{\beta^2}{2}}\leq C |h(s,\theta)|e^{-\frac{\beta^2\Re\omega}{2}s }.
\end{equation*}
We finish the proof by putting the above bound in \eqref{eq_temp_in_GMC_bound}.
\end{proof}
We stress that when $s$ becomes large, $|f_s(e^{\i\theta})|\approx|f_s'(e^{\i\theta})|\approx C e^{-s{\Re\omega} }$. Take in \cref{lemma_gmc_delta_exp_bound}
\begin{equation*}
    h(s,\theta)=\exp\left(-\frac{\beta Q}{2\delta}\int_{-\delta}^{\delta}\log\tfrac{|f_s'(e^{\i(\theta+\theta')})|}{|f_s(e^{\i(\theta+\theta')})|}\d\theta'\right),
\end{equation*}
and
\begin{equation*}
    h(s,\theta)\leq \frac{1}{2\delta}\int_{-\delta}^{\delta}\exp\left(-\beta Q\log\frac{|f_s'(e^{\i(\theta+\theta')})|}{|f_s(e^{\i(\theta+\theta')})|}\right)\d\theta'\leq \sup_{\theta'\in[-\delta_0,\delta_0]}\left|\frac{f_s'(e^{\i(\theta+\theta')})}{f_s(e^{\i(\theta+\theta')})}\right|^{-\frac{\beta^2}{2}+2},
\end{equation*}
where the above is bounded in $s>0$ and $\theta\in[0,2\pi]$. Now applying change of variable to \eqref{eq_u_and_k}, we have when $T$ large
\begin{equation*}
\begin{split}
     u^2_\delta(T)\leq{}&C\iint_{([0,T]\times[0,2\pi])^2}|f_{s_1}(e^{\i\theta_1})f_{s_2}(e^{\i\theta_2})|^{\frac{\beta^2}{2}}|f_{s_1}(e^{\i\theta_1})-f_{s_2}(e^{\i\theta_2})|^{-\beta^2}\,\d s_1\d\varrho_\v(\theta_1)\d s_2\d\varrho_\v(\theta_2)\\
     \leq{}&C\iint_{(\mathbb A_{f_T})^2}|z_1z_2|^{\beta Q}|z_1-z_2|^{-\beta^2}\,\d z_1\d z_2\leq C\iint_{Ce^{-T\Re\omega}\leq|z_1|,|z_2|\leq 1}|z_1z_2|^{\beta Q}|z_1-z_2|^{-\beta^2}\,\d z_1\d z_2.
\end{split}
\end{equation*}
Evaluating the last integral on $C^le^{-l\,T\Re\omega  }\leq|z_1|,|z_2|\leq C^{l-1}e^{-(l-1)\, T\Re\omega  },l\geq1$ shows that $ u^2_\delta(T)\leq  C(1+T\Re\omega) $. Therefore the bound in \cref{lemma_gmc_delta_exp_bound} is at most exponential in $T$ and independent of $\delta<\delta_0$. Then we can take the $\delta\to0$ limit in \eqref{eq_diff_GMC_delta} and obtain the $L^2$ convergence of the resolvent if we choose $\lambda$ large enough.

We first state a general lemma about convergence.
\begin{lem}\label{lemma_resolvent_technical}
    Let $A\geq c_0I,c_0>0$ be a positive self-adjoint operator on Hilbert space $\mathcal H$. Write $\mathcal W^s:=\mathcal D(A^{\frac s2})$, i.e. the domain under $\|u\|_s:=\|A^{\frac s2}u\|_{\mathcal H}$, with $\|\cdot\|_{\mathcal H}$ the Hilbert norm. Also there is a sequence of bounded operator $(B_n)_{n\geq1}$ satisfying
    \begin{equation*}
        \|e^{-t A}B_ne^{-s A}\|_{\mathcal H\to\mathcal H}\leq C_0(ts)^{-\alpha}e^{-c_1(t+s)},\ \ \ s,t>0,
    \end{equation*}
    where $C_0,c_1>0$ and $0\leq\alpha<\frac12$. Also consider a dense subspace $\mathcal C$ of $\mathcal W^1$ on which $\la B_n u,v\ra_{\mathcal H}$ converges for $u,v\in\mathcal C$ and we denote the limit by $\la B u,v\ra_{\mathcal H}$. If $u_n\to u$ strongly in $\mathcal H$ and $\sup_{n}\|u_n\|_{1}<\infty$, then $u\in \mathcal W^1$ and $\la B_n u_n,v\ra_{\mathcal H}$ converge to $\la B u,v\ra_{\mathcal H}$ for all $v\in\mathcal W^r,2\alpha<r<1$.
\end{lem}
\begin{proof}
   Fix $2\alpha<r<1$. Since $A^{-\frac r2}=\frac{1}{\Gamma(\frac r2)}\int_0^\infty t^{\frac r2-1}e^{-tA}\,\mathrm{d}t$, it is plain to check that 
   \begin{equation}\label{eq_temp_ABA}
       \|A^{-\frac{r}{2}}B_nA^{-\frac{r}{2}}\|_{\mathcal H\to\mathcal H}\leq\frac{C_0}{\Gamma(\frac r2)^2}\left(\int_0^{\infty}t^{\frac{r}{2}-\alpha-1}e^{-c_1t}\,\mathrm{d}t\right)^2=:C(r)<\infty,
   \end{equation}
   using the bound from $\|e^{-t A}B_ne^{-s A}\|_{\mathcal H\to\mathcal H}$. Hence 
   \begin{equation}\label{eq_temp_bound}
       |\la B_n u,v\ra_{\mathcal H}|\leq C(r)\|u\|_r\|v\|_r.
   \end{equation}
   Now let us check the first item. 
   \begin{equation*}
   \begin{split}
        |\la B_n u_n,v\ra_{\mathcal H}-\la Bu,v\ra_{\mathcal H}|\leq{}&|\la B_n (u_n-u),v\ra_{\mathcal H}|+|\la B_nu,v\ra_{\mathcal H}-\la Bu,v\ra_{\mathcal H}|\\
        \leq{}&C(r)\|u_n-u\|_r\|v\|_r+|\la B_nu,v\ra_{\mathcal H}-\la Bu,v\ra_{\mathcal H}|.
   \end{split}
   \end{equation*}
   According to \cite[Lemma 2.12]{Ma_1992}, the convergence of $u_n$ in $\mathcal H$ and boundedness in $\mathcal W^1$ imply that $u\in\mathcal W^1$. By H\"older's inequality, we have
   \begin{equation*}
       \|u_n-u\|_r\leq\|u_n-u\|_{\mathcal H}^{1-r}\|u_n-u\|_1^{r}\to0.
   \end{equation*}
   The convergence of $|\la B_nu,v\ra_{\mathcal H}-\la Bu,v\ra_{\mathcal H}|$ is again from \eqref{eq_temp_bound} with $\mathcal C$ approximation. 
\end{proof}

By the definition of $\Qc_\v^{(\delta)}$, for $G\in\Ec(\DT)$ we have
\begin{equation*}
    \lambda \langle \mathbf R_{\v,\lambda}^{(\delta)}F,G\rangle_{L^2}+\Qc_\v^{(\delta)}(\mathbf R_{\v,\lambda}^{(\delta)}F,G)=\langle F,G\rangle_{L^2}.
\end{equation*}
Especially one can take $G=\mathbf R_{\v,\lambda}^{(\delta)}F$ (by approximation) and get $\sup_{0<\delta<1} |\Qc_\v^{(\delta)}(\mathbf R_{\v,\lambda}^{(\delta)}F,\mathbf R_{\v,\lambda}^{(\delta)}F)|<\infty$. In \cref{Q_v_bound}, specifically in \cref{V_v_estimate}, the comparison between $\Qc_\v^{(\delta)}$ and $\widetilde{\Qc}_0$ is uniform in $\delta$. We thus have $\sup_{0<\delta<1} |\widetilde{\Qc}^0(\mathbf R_{\v,\lambda}^{(\delta)}F,\mathbf R_{\v,\lambda}^{(\delta)}F)|<\infty$ (for each $\delta$, $\Qc^0_\v(\mathbf R_{\v,\lambda}^{(\delta)}F,\mathbf R_{\v,\lambda}^{(\delta)}F)$ is finite since $V_\v^{(\delta)}$ is bounded). Then according to \cite[Lemma 2.12]{Ma_1992}, there is a subsequence of $ (\mathbf R_{\v,\lambda}^{(\delta)}F)_{\delta}$ that converges weakly in $\widetilde{\mathcal Q}^0$ and $\mathbf R_{\v,\lambda}F\in\mathcal D(\widetilde{\mathcal Q}^0)$. Recall \eqref{eq_estimate_PVP}, and we can easily use \cref{lemma_resolvent_technical} with $\mathbf H^0+Q^2$ to get the convergence of $\Qc_\v^{(\delta)}(\mathbf R_{\v,\lambda}^{(\delta)}F,G)$ to $\Qc_\v(\mathbf R_{\v,\lambda}F,G)$.

In the end, we have shown that for $\lambda$ large enough,
 \begin{equation*}
    \lambda \langle \mathbf R_{\v,\lambda}F,G\rangle_{L^2}+\Qc_\v(\mathbf R_{\v,\lambda}F,G)=\langle F,G\rangle_{L^2},
\end{equation*}
which means the resolvents from two approaches agree and therefore the proof is complete by the uniqueness of the resolvent.
\end{proof}

\section{Continuity and differentiability}\label{sec_con_diff}

In order to talk about the analytic properties of the quantized semigroup $f\in\Ssg\mapsto\Tbf_f\in\Lc(L^2(\DT))$, we need a topology on $\Ssg$. For $\varepsilon>0$, we define $\Ssg_\varepsilon$ to be the subset of maps in $\Ssg$ that are holomorphic in $(1+\varepsilon)\Db^\circ$ and smooth up to the boundary $(1+\varepsilon)\Tb$. 
Then $\Ssg=\bigcup_{\varepsilon>0}\Ssg_\varepsilon$. For each $\varepsilon>0$, we equip $\Ssg_\varepsilon$ with the Fréchet topology induced by the family of seminorms $(\|\!\cdot\!\|_{\varepsilon,\ell})_{\ell\geq0}$ defined for $f(z)=\sum_{n\in\Zb}f_nz^{n+1}$ by
\begin{equation}
\|f\|_{\varepsilon,\ell}^2 := \sum_{n\in\Zb}(1+\varepsilon)^{2|n|+2}(1+|n|)^{2\ell}|f_n|^2.
\end{equation}
Then $\Ssg$ is equipped with the direct limit topology, so that a map on $\Ssg$ is continuous if and only if its restriction to each $\Ssg_\varepsilon$ is continuous with respect to the Fréchet topology. 

For $\epsilon>0$, we also consider the space $\mathrm{Hol}_{\epsilon}(\mathbb D)$ that consists of holomorphic functions on the open disk $(1+\epsilon)\mathbb D^\circ$ that extend smoothly up to the boundary $(1+\varepsilon)\Tb$, also equipped with the seminorms $(\|\!\cdot\!\|_{\varepsilon,\ell})_{\ell\geq0}$. Then $\mathrm{Hol}(\mathbb D):=\bigcup_{\varepsilon>0}\mathrm{Hol}_{\epsilon}(\mathbb D)$. And we also define $\mathrm{Hol}_{\epsilon}^\bullet(\mathbb D)$ ($\mathrm{Hol}^\bullet(\mathbb D)$ resp.) as the closed subspace of $\mathrm{Hol}_{\epsilon}(\mathbb D)$ ($\mathrm{Hol}(\mathbb D)$ resp.) that consists of functions $f$ satisfying $f(0)=0$, equipped with the induced topology.

\begin{prop}\label{prop_Tf_continuity}
The map $f\in\Ssg\mapsto\Tbf_f\in\Lc(L^2(\DT))$ is continuous.
\end{prop}
\begin{proof}
Let $g_f$ be an admissible metric on $\Ab_f$ that depends in a $C^1$ fashion on $f$. Since the prefactor $W(f,g_f) $ in \cref{prop_semigroup_as_amplitude} is clearly continuous in $f$, we only need to show the continuity of the annulus amplitude $\Ac_{\Ab_f,g_f,\zetabf_f}$. Let $\widetilde{\boldsymbol{\varphi}}^k=(\widetilde{{\varphi}}^k_1,\widetilde{{\varphi}}^k_2)\in\DT^2$ with $k\in\Zb$. Recall the decompositions $\wt\varphi_1^k=c_1+\varphi_1+kR\arg$, $\wt\varphi_2^k=c_2+\varphi_2+kR\arg$ with $c_1,c_2\in\TR$, $\varphi_1,\varphi_2\in\Dc_0'(\Tb,\Rb)$. For $j\in\Zb$, recall $\cbf_j:=(c_1+2\pi Rj,c_2)$ and the $1$-form $\nu_{f,k}$ on $\Ab_f$
\begin{equation*}
    \nu_{f,k}:= kR\, d\arg+\,d P(0,k R(\arg f^{-1}-\arg)|_{f(\mathbb T)}),
\end{equation*}
where $P=P_{\Ab_f}$ is the Poisson operator on $\Ab_f$, i.e., $P(h_1,h_2)$ is the unique harmonic function on $\Ab_f$ that equals $h_1$ on $\mathbb T$ and $h_2$ on $f(\mathbb T)$. Then by definition,
    \begin{equation*}
    \begin{split}
        \Ac_{\Ab_f,g_f,\zetabf_f}(\widetilde{\boldsymbol{\varphi}}^k)={}&\det(\Delta_{\mathbb A_f,g_f,\mathrm{D}})^{-1/2}e^{-\frac1{4\pi} \int_{\Ab_f}|\nu_{f,k}|^2_{g_f}\d v_{g_f}}\sum_{j\in\mathbb Z}e^{-\frac{1}{2\pi}\int_{\partial\mathbb A_f}P(\cbf_j+\boldsymbol{\varphi})\,\nu_{f,k}(\nbf)\,\d\ell_{g_f}-\frac1{2}\la\cbf_j+\boldsymbol{\varphi},(\mathbf{D}_{\Ab_f}-\mathbf D)(\cbf_j+\boldsymbol{\varphi})\ra} \\
        \times&\mathbb E_{\widetilde{\boldsymbol{\varphi}}^k}[e^{-\frac{\i Q}{4\pi}\int_{\Ab_f}^{\mathrm{reg}}K_{g_f}X_{g_f}^{k,j}\,\d v_{g_f}-\mu\int_{\Ab_f}M_\beta(X_{g_f}^{k,j},\d v_{g_f})}].
    \end{split}
\end{equation*}
And
\[X_{g_f}^{k,j}:=I(\nu_{f,k})+P (\cbf_j+{\boldsymbol{\varphi}})+X_{g_f},\]
where $I(\nu_{f,k})(z):=\int_1^z \nu_{f,k}$ (see \cref{Sec_background} for the choice of arc $d_f$) and $X_{g_f}$ is the Dirichlet GFF on the annulus $\Ab_f$. Here the curvature term $\int_{\Ab_f}^{\mathrm{reg}}K_{g_f}X_{g_f}^{k,j}\,\d v_{g_f}$ is regularized as in \cite[Section 5]{BCILT} as
\begin{equation*}
    \int_{\Ab_f}^{\mathrm{reg}}K_{g_f}I(\nu_{f,k})\,\d v_{g_f}:= \int_{\Ab_f}K_{g_f}I(\nu_{f,k})\,\d v_{g_f}-4\pi Rk\int_{d_f}k_{g_f}\d \ell_{g_f},
\end{equation*}
where $k_{g_f}$ is the geodesic curvature and $\d\ell_{g_f}$ is the length measure under metric $g_f$.

    Consider a sequence $(f_n)_n$ that converges to $f$. Then there exists $\epsilon>0$ such that $f_n,f\in\Ssg_\epsilon$ for all $n$ large enough and $f_n\to f$ under $\|\!\cdot\!\|_{\epsilon,\ell}$ for all $\ell$. Then we can similarly introduce $X_{g_{f_n},k,j}$. In the following, we denote $X_{g_{f_n},k,j}$ by $X_{n}^{k,j}$ ($X^{k,j}:=X_{g_f}^{k,j}$ as well) and $X_n$ ($X$ resp.) for the Dirichlet GFF on $\Ab_{f_n}$ ($\Ab_f$ resp.). Similarly, we have $P_n(\cbf_j+\boldsymbol{\varphi})$ and $P(\cbf_j+\boldsymbol{\varphi})$ for the harmonic extension on $\mathbb A_{f_n}$ and $\mathbb A_{f}$ then $\nu_{f_n,k}^{n},d_{f_n},\nbf_n$ for the 1-form, the separating arc and outward normal vector on $\mathbb A_{f_n}$. We also replace the surface $(\mathbb A_{f_n},g_{f_n}),(\mathbb A_{f},g_f)$, the volume form $\d v_{g_{f_n}},\d v_{g_f}$, and the curvature $K_{g_{f_n}},K_{g_f}$,  with $(\Sigma_n,g_n),(\Sigma,g)$, $\d v_n,\d v$ and $K_n,K$. Take $0<\xi<\epsilon$ and consider the subset 
    \begin{equation*}
        \Sigma^{\xi}:=(\Sigma\backslash f((1+\xi)\mathbb D))\cap(1-\xi)\mathbb D.
    \end{equation*}
Let
\begin{equation*}
 \begin{split}
    \mathcal A_{j,n}^0(\widetilde{\varphibf}^k):={}&\det(\Delta_{\Sigma_n,g_n,\mathrm{D}})^{-1/2}\exp\Bigg(-\frac1{4\pi}\int_{\Sigma_n}|\nu_{f_n,k}|^2_{g_{n}}\d v_{n}\\
    -&\frac1{2\pi}\int_{\partial\Sigma_n}P_n(\cbf_j+\boldsymbol{\varphi})\,\nu_{f_n,k}(\nbf_n)\,\d\ell_{n}-\frac1{2}\la\cbf_j+\boldsymbol{\varphi},(\mathbf{D}_{\Sigma_n}-\mathbf D)(\cbf_j+\boldsymbol{\varphi})\ra\Bigg),
\end{split}
\end{equation*}
and $\mathcal A^0_j(\widetilde{\varphibf}^k)$ accordingly for $\Sigma$ and $f$. In the following proof, we will focus on a fixed $k\in\mathbb Z$. Indeed, first notice that $|\Ac_{\Ab_{f_n},g_{f_n},\zetabf_{f_n}}|\leq C\sum_{j\in\mathbb Z}\mathcal A^0_{j,n}$ and $|\Ac_{\Ab_{f},g_{f},\zetabf_{f}}|\leq C\sum_{j\in\mathbb Z}\mathcal A^0_j$ for some constant $C>0$ that does not depend on $\widetilde{\varphibf}$, $k$ and $n$, because of \cite[Proposition 5.3]{CILT}. According to the first item in  \cref{lem_continuity_tech}, the norm estimate of $\sum_{|k|\geq K}\Ac_{\Ab_{f_n},g_{f_n},\zetabf_{f_n}}(\widetilde{\boldsymbol{\varphi}}^k)$ for some large $K>0$ is dominated by the norm of $\sum_{|k|\geq K}\sum_{j\in\mathbb Z}\mathcal A^0_j$.

For $n$ large enough, we have $\Sigma^{\xi}\subset\Sigma_n$ and $ \Ac_{\Ab_{f_n},g_{f_n},\zetabf_{f_n}}(\widetilde{\boldsymbol{\varphi}}^k)=A_{n,\xi}^{(1)}+A_{n,\xi}^{(2)}$, where
\begin{equation*}
    A_{n,\xi}^{(1)}:=\sum_{j\in\mathbb Z}\mathcal  A_{j,n}^0(\widetilde{\varphibf}^k) \mathbb E_{\widetilde{\boldsymbol{\varphi}}^k}[e^{-\frac{\i Q}{4\pi}\int_{\Sigma_n}^{\mathrm{reg}}K_{n}X_{n}^{k,j}\,\d v_n-\mu\int_{\Sigma^{\xi}}M_\beta(X_{n}^{k,j},\d v_n)}],
\end{equation*}
and
\begin{equation*}
\begin{split}
     A_{n,\xi}^{(2)}:={}&\sum_{j\in\mathbb Z}\mathcal  A_{j,n}^0(\widetilde{\varphibf}^k) \mathbb E_{\widetilde{\boldsymbol{\varphi}}^k}[e^{-\frac{\i Q}{4\pi}\int_{\Sigma_n}^{\mathrm{reg}}K_{n}X_{n}^{k,j}\,\d v_n}(e^{-\mu\int_{\Sigma_n}M_\beta(X_{n}^{k,j},\d v_n)}-e^{-\mu\int_{\Sigma^{\xi}}M_\beta(X_{n}^{k,j},\d v_n)})].
\end{split}
\end{equation*}
By applying Cauchy-Schwarz inequality to the expectation in $ A_{n,\xi}^{(2)}$, we have
\begin{equation*}
    |A_{n,\xi}^{(2)}|\leq \sum_{j\in\mathbb Z}\mathcal  A_{j,n}^0(\widetilde{\varphibf}^k)  \mathcal K_{n,\xi}^{k,j}(\widetilde{\boldsymbol{\varphi}}^k)\mathbb E_{\widetilde{\boldsymbol{\varphi}}^k}[|e^{-\mu\int_{\Sigma^{\xi}}M_\beta(X_{n}^{k,j},\d v_n)}|^2]^{\frac12},
\end{equation*}
where
\begin{equation*}
    \mathcal K_{n,\xi}^{k,j}(\widetilde{\boldsymbol{\varphi}}^k):=\mathbb E_{\widetilde{\boldsymbol{\varphi}}^k}[|e^{-\mu\int_{\Sigma_n\backslash\Sigma^{\xi}}M_\beta(X_{n}^{k,j},\d v_n)}-1|^2]^{\frac12}.
\end{equation*}
Using the fact that $|e^z-1|^2\leq (e^{|z|}-1)^2\leq e^{2|z|}-1$ for all $z\in\mathbb C$, we can bound $ \mathcal K_{n,\xi}^{k,j}(\widetilde{\boldsymbol{\varphi}}^k)$ with
\begin{equation}\label{eq_K_bound}
    | \mathcal K_{n,\xi}^{k,j}(\widetilde{\boldsymbol{\varphi}}^k)|\leq \mathbb E_{\widetilde{\boldsymbol{\varphi}}^k}[e^{2|\mu\int_{\Sigma_n\backslash\Sigma^{\xi}}M_\beta(X_{n}^{k,j},\d v_n)|}-1]^{\frac12}.
\end{equation}
According to \cite[Proposition 5.3]{CILT}, for some constant $C$ only depending on $\beta$, the above is bounded by
\begin{equation*}
   \sqrt{ e^{C|\mu| w}(1+2|\mu| u e^{C|\mu|^2 u^2})-1},
\end{equation*}
where
\begin{equation}\label{eq_metric_integral_continuity}
    u^2:=\iint_{(\Sigma_n\backslash\Sigma^{\xi})^2}d_{g_n}(x,y)^{-\beta^2}\ dv_n(x)\d v_n(y), \ \ \ w:=\int_{\Sigma_n\backslash\Sigma^{\xi}}d_{g_n}(x,\partial\Sigma_n)^{-\frac{\beta^2}2}\ \d v_n(x).
\end{equation}

Then by \cite[Lemma 4.5]{Guillarmou2021}, the first and second item in \cref{lem_continuity_tech}, $A_{n,\xi}^{(2)}$ yields an operator on $\mathcal H$ with norm $(\iint |A_{n,\xi}^{(2)}|^2 (\d\widetilde{\varphi}^k)^{\otimes2})^{1/2}$ of order $\xi^{\frac12-\frac{\beta^2}{4}}$. 

Similarly, we write $ \Ac_{\Ab_{f},g_{f},\zetabf_{f}}(\widetilde{\boldsymbol{\varphi}}^k)=A_{\xi}^{(1)}+A_{\xi}^{(2)}$, where
\begin{equation*}
    A_{\xi}^{(1)}:=\sum_{j\in\mathbb Z}\mathcal A_{j}^0(\widetilde{\varphibf}^k) \mathbb E_{\widetilde{\boldsymbol{\varphi}}^k}[e^{-\frac{\i Q}{4\pi}\int_{\Sigma}^{\mathrm{reg}}KX^{k,j}\,\d v-\mu\int_{\Sigma^{\xi}}M_\beta(X^{k,j},\d v)}],
\end{equation*}
and again the norm of $A_{\xi}^{(2)}$ is under control. 
Now we estimate $A_{n,\xi}^{(1)}-A_{\xi}^{(1)}=B_{n,\xi}^{(1)}+B_{n,\xi}^{(2)}$, with
\begin{equation*}
\begin{split}
     B_{n,\xi}^{(1)}:=\sum_{j\in\mathbb Z}\mathcal  A_{j}^0(\widetilde{\varphibf}^k)\left(\mathbb E_{\widetilde{\boldsymbol{\varphi}}^k}[e^{-\frac{\i Q}{4\pi}\int_{\Sigma_n}^{\mathrm{reg}}K_{n}X_{n}^{k,j}\,\d v_n-\mu\int_{\Sigma^{\xi}}M_\beta(X_{n}^{k,j},\d v_n)}]
    -\mathbb E_{\widetilde{\boldsymbol{\varphi}}^k}[e^{-\frac{\i Q}{4\pi}\int_{\Sigma}^{\mathrm{reg}}KX^{k,j}\,\d v-\mu\int_{\Sigma^{\xi}}M_\beta(X^{k,j},\d v)}]\right),
\end{split}
\end{equation*}
and
\begin{equation*}
    \begin{split}
   B_{n,\xi}^{(2)}:={}&\sum_{j\in\mathbb Z}(\mathcal  A_{j,n}^0(\widetilde{\varphibf}^k)-\mathcal  A_{j}^0(\widetilde{\varphibf}^k))
   \mathbb E_{\widetilde{\boldsymbol{\varphi}}^k}[e^{-\frac{\i Q}{4\pi}\int_{\Sigma_n}^{\mathrm{reg}}K_{n}X_{n}^{k,j}\,\d v_n-\mu\int_{\Sigma^{\xi}}M_\beta(X_{n}^{k,j},\d v_n)}].
\end{split}
\end{equation*}
The estimate of $B_{n,\xi}^{(2)}$ is a consequence of the third item in \cref{lem_continuity_tech} together with \cite[Theorem 6.11]{CILT}. We will now focus on $B_{n,\xi}^{(1)}$. First using the imaginary Cameron--Martin theorem (the proof of \cite[Theorem 6.11]{CILT}), $B_{n,\xi}^{(1)}$ can be written as
\begin{equation}\label{eq_B1}
\begin{split}
     B_{n,\xi}^{(1)}={}&\sum_{j\in\mathbb Z}\mathcal  A_{j}^0(\widetilde{\varphibf}^k)(e^{-\frac{\i Q}{4\pi}\int_{\Sigma_n}^{\mathrm{reg}}K_{n}(X_n^{k,j}-X_n)\,\d v_n-\frac{Q^2}{32\pi^2}\int K_n\ \mathcal G_n\ \d v_n}\ \mathbb E_{\widetilde{\boldsymbol{\varphi}}^k}[e^{-\mu\int_{\Sigma^{\xi}}M_\beta(X_n^{k,j}-\frac{\i Q}{4\pi}\mathcal G_n,\d v_n)}]\\
    -&e^{-\frac{\i Q}{4\pi}\int_{\Sigma}^{\mathrm{reg}}K (X^{k,j}-X)\,\d v-\frac{Q^2}{32\pi^2}\int K\ \mathcal G\ \d v}\ \mathbb E_{\widetilde{\boldsymbol{\varphi}}^k}[e^{-\mu\int_{\Sigma^{\xi}}M_\beta(X^{k,j}-\frac{\i Q}{4\pi}\mathcal G,\d v)}]),
\end{split}
\end{equation}
where $\mathcal G_n(x):=\int_{\Sigma_n}K_n(x')G_n(x,x')\ \d v_n(x')$ and $(2\pi)^{-1}G_n$ is the Green's function on $\Sigma_n$ with Dirichlet boundary conditions ($\mathcal G(x)$ is similarly defined). Suppose $g_n=e^{\rho_n}g$ on $\Sigma^\xi$ and it is not hard to see that $\sup_{\Sigma^\xi}|\rho_n|\to0$. Recall that
\begin{equation*}
    e^{(\frac{\beta^2}{4}-1)\rho_n} M_\beta(X_n^{k,j}-\frac{\i Q}{4\pi}\mathcal G_n,\d v_n)= M_\beta(X_n^{k,j}-\frac{\i Q}{4\pi}\mathcal G_n,\d v),
\end{equation*}
where $e^{\frac{\beta^2}{4}\rho_n}$ is the effect of changing the metric regularisation for GFF. Then we can replace the term containing $\mathbb E_{\widetilde{\boldsymbol{\varphi}}^k}[e^{-\mu\int_{\Sigma^{\xi}}M_\beta(X_n^{k,j}-\frac{\i Q}{4\pi}\mathcal G_n,\d v_n)}]$ in \eqref{eq_B1} with $\mathbb E_{\widetilde{\boldsymbol{\varphi}}^k}[e^{-\mu\int_{\Sigma^{\xi}}M_\beta(X_n^{k,j}-\frac{\i Q}{4\pi}\mathcal G,\d v)}]$ if we can control
\begin{equation}\label{eq_B3}
\begin{split}
     B_{n,\xi}^{(3)}:={}&\sum_{j\in\mathbb Z}\mathcal  A_{j}^0(\widetilde{\varphibf}^k)(e^{-\frac{\i Q}{4\pi}\int_{\Sigma}^{\mathrm{reg}}K (X^{k,j}-X)\,\d v-\frac{Q^2}{32\pi^2}\int K\ \mathcal G\ \d v}\ \mathbb E_{\widetilde{\boldsymbol{\varphi}}^k}[e^{-\mu\int_{\Sigma^{\xi}}M_\beta(X_n^{k,j}-\frac{\i Q}{4\pi}\mathcal G,\d v)}]\\
    -&e^{-\frac{\i Q}{4\pi}\int_{\Sigma_n}^{\mathrm{reg}}K_{n}(X_n^{k,j}-X_n)\,\d v_n-\frac{Q^2}{32\pi^2}\int K_n\ \mathcal G_n\ \d v_n}\ \mathbb E_{\widetilde{\boldsymbol{\varphi}}^k}[e^{-\mu\int_{\Sigma^{\xi}}M_\beta(X_n^{k,j}-\frac{\i Q}{4\pi}\mathcal G_n,\d v_n)}]).
\end{split}
\end{equation}
For each $j\in\mathbb Z$, the convergence of the term $e^{-\frac{\i Q}{4\pi}\int_{\Sigma_n}^{\mathrm{reg}}K_{n}\ (X_n^{k,j}-X_n)\,\d v_n-\frac{Q^2}{32\pi^2}\int K_n\ \mathcal G_n\ \d v_n}$ above is already addressed in \cite[Proposition 3.11]{Segal_semigroup} (and the convergence of $\int_{d_{f_n}}k_{g_{f_n}}\d \ell_{g_{f_n}}$ is straightforward), then we apply the dominated convergence theorem. And according to \cite[Lemma 3.12]{Segal_semigroup} and \cite[Proposition 5.3]{CILT}, we can bound the GMC difference in \cref{eq_B3} since $\sup_{\Sigma^\xi}|\rho_n|\to0$ and $\sup_{\Sigma^\xi}|\mathcal G_n-\mathcal G|\to0$ \cite[Lemma 3.12 (3)]{Segal_semigroup}.

For the rest of the proof, we just need to show that
\begin{equation}\label{eq_goal_conv_B1}
    \iint \left|\sum_{j\in\mathbb Z}\mathcal  A_{j}^0(\widetilde{\varphibf}^k)|\mathbb E_{\widetilde{\boldsymbol{\varphi}}^k}[e^{-\mu\int_{\Sigma^{\xi}}M_\beta(X_n^{k,j}-\frac{\i Q}{4\pi}\mathcal G,\d v)}]-\mathbb E_{\widetilde{\boldsymbol{\varphi}}^k}[e^{-\mu\int_{\Sigma^{\xi}}M_\beta(X^{k,j}-\frac{\i Q}{4\pi}\mathcal G,\d v)}]|\right|^2(\d\widetilde{\varphi}^k)^{\otimes2}<\xi
 \end{equation}
for suitably large $n$. Note that we now compare two different GFFs so we cannot use \cite[Proposition 5.3]{CILT} directly. Instead, using the Taylor expansion of the exponential, we obtain
\begin{equation}\label{eq_int_expand}
\begin{split}
     &\mathbb E_{\widetilde{\boldsymbol{\varphi}}^k}[e^{-\mu\int_{\Sigma^{\xi}}M_\beta(X_n^{k,j}-\frac{\i Q}{4\pi}\mathcal G,\d v)}]=\sum_{s\geq0}\frac{(-\mu)^s}{s!}\mathbb E_{\widetilde{\boldsymbol{\varphi}}^k}[(\int_{\Sigma^{\xi}}e^{\frac{\beta Q}{4\pi}\mathcal G} M_\beta(X_n^{k,j},\d v))^s]\\
     ={}&\sum_{s\geq0}\frac{(-\mu)^s}{s!}\int_{(\Sigma^\xi)^s}e^{-\beta^2\sum_{r<l} G_n(x_r,x_l)}e^{\beta(\sum_{r=1}^s\frac{ Q}{4\pi}\mathcal G(x_r)+\i(X_n^{k,j}(x_r)-X_n(x_r))  )}e^{-\frac{\beta^2}{2}\sum_{r=1}^sW_{n}(x_r)}\ \prod_r\d v(x_r),
\end{split}
\end{equation}
where $W_{n}$ is the Robin mass of $X_n$ under the regularisation of the metric $g$, namely
\begin{equation*}
    W_{n}(x):=\lim_{\delta\to0}\mathbb E[X^2_{n,\delta}(x)]+\log\delta.
\end{equation*}
Let us consider the Dirichlet GFF $X^{\xi}$ on $\Sigma^\xi$. Since $\Sigma^\xi\subset \cap_{n\geq N}\Sigma_n$, we have the monotonicity of the $X^{\xi}$ Green's function $G^{\xi}\leq G_n$ on $(\Sigma^{\xi})^2$ and Robin mass $W^\xi\leq W_{n}$ on $\Sigma^{\xi}$. Thus the second line in \eqref{eq_int_expand} is further bounded by
\begin{equation}\label{eq_expansion_bound}
\begin{split}
    \sum_{s\geq0}\frac{|\mu|^s}{s!}\int_{(\Sigma^\xi)^s}e^{-\beta^2\sum_{r<l} G^\xi(x_r,x_l)}e^{\beta\sum_{r=1}^s\frac{ Q}{4\pi}\mathcal G(x_r)}e^{-\frac{\beta^2}{2}\sum_{r=1}^sW^\xi(x_r)}\ \d v(x)
    =\mathbb E[e^{|\mu|\int_{\Sigma^{\xi}}M_\beta(X^\xi-\frac{\i Q}{4\pi}\mathcal G,\d v)}].
\end{split}
\end{equation}
Therefore, for each moment in the expansion of \eqref{eq_int_expand}, we provide a uniform bound independent of $n,k,j$ and $\widetilde{\boldsymbol{\varphi}}^k$. For each $\xi>0$, we can pick a large $S$, such that 
\begin{equation*}
    \sum_{s\geq S}\frac{(-\mu)^s}{s!}\mathbb E_{\widetilde{\boldsymbol{\varphi}}^k}[(\int_{\Sigma^{\xi}}e^{\frac{\beta Q}{4\pi}\mathcal G} M_\beta(X_n^{k,j},\d v))^s]
\end{equation*}
 is arbitrarily small, uniformly over $n,k,j$ and $\widetilde{\boldsymbol{\varphi}}^k$. For $s<S$, we clearly have that the integral
 \begin{equation*}
     \int_{(\Sigma^\xi)^s}e^{-\beta^2\sum_{r<l} G_n(x_r,x_l)}e^{\beta\sum_{r=1}^s(\frac{ Q}{4\pi}\mathcal G(x_r)+\i(X_n^{k,j}-X_n))}e^{-\frac{\beta^2}{2}\sum_{r=1}^sW_{n}(x_r)}\ \d v(x)
 \end{equation*}
 converges to the one corresponding to $X$ since $\sup_{(\Sigma^\xi)^2}|G_n-G|$ and $\sup_{\Sigma^\xi}|P_n\widetilde{\boldsymbol{\varphi}}^k-P\widetilde{\boldsymbol{\varphi}}^k|$ goes to zero when $n$ gets large as stated in \cite[Lemma 3.12 (2),(4)]{Segal_semigroup} (semigroup lemma). We thus have \eqref{eq_goal_conv_B1} using dominated convergence theorem by \eqref{eq_expansion_bound}.
 \end{proof}



    


 Let us put some estimates used in the proof of \cref{prop_Tf_continuity}.
\begin{lem}\label{lem_continuity_tech}
Follow the notation in \cref{prop_Tf_continuity}. Fix $\epsilon>0$. Let $f_n\to f$ in $\mathcal S_\epsilon$ under Fréchet topology, and $(\Sigma_n,g_n):=(\mathbb A_{f_n},g_{f_n})$, $(\Sigma,g):=(\mathbb A_f,g_f)$. Similarly, we have $v_n,v,\ell_n,\ell$. Let $\Sigma^{\xi}:=(\Sigma\backslash f((1+\xi)\mathbb D))\cap(1-\xi)\mathbb D,0<\xi<\epsilon.$ 
Then:

    \begin{enumerate}[label=\emph{(\arabic*)}]
        \item Let 
\begin{equation*}
 \begin{split}
    \mathcal A_{j,n}^0(\widetilde{\varphibf}^k):={}&\det(\Delta_{\Sigma_n,g_n,\mathrm{D}})^{-1/2}\exp\Bigg(-\frac1{4\pi}\int_{\Sigma_n}|\nu_{f_n,k}|^2_{g_{n}}\d v_{n}\\
    -&\frac1{2\pi}\int_{\partial\Sigma_n}P_n(\cbf_j+\boldsymbol{\varphi})\,\nu_{f_n,k}(\nbf_n)\,\d\ell_{n}-\frac1{2}\la\cbf_j+\boldsymbol{\varphi},(\mathbf{D}_{\Sigma_n}-\mathbf D)(\cbf_j+\boldsymbol{\varphi})\ra\Bigg),
\end{split}
\end{equation*}
and $\mathcal A^0_j(\widetilde{\varphibf}^k)$ accordingly for $\Sigma$ and $f$. Then
\begin{equation*}
   \lim_{n\to\infty}\left\|\sum_{j\in\mathbb Z} \mathcal A_{j,n}^0(\widetilde{\varphibf}^k)-\sum_{j\in\mathbb Z}\mathcal A_{j}^0(\widetilde{\varphibf}^k)\right\|_{\mathcal H^{\otimes2}}=0.
\end{equation*}
        \item For $\xi$ small and $n$ large enough (depending on $\xi$), there exists some $C>0$ depending on $f,\beta$,
\begin{equation*}
            \iint_{(\Sigma_n\backslash\Sigma^{\xi})^2}d_{g_n}(x,y)^{-\beta^2}\ dv_n(x)\d v_n(y)\leq C\xi^{2-\beta^2},\ \ \ \int_{\Sigma_n\backslash\Sigma^{\xi}}d_{g_n}(x,\partial\Sigma_n)^{-\frac{\beta^2}2}\ \d v_n(x)\leq C \xi^{1-\frac{\beta^2}{2}}.
        \end{equation*}
        \item Fix $k\in\mathbb Z$ and let $\mathcal H_k$ be the subspace of $\mathcal H$ with winding $k$. Then there exists $C>0,\delta_n\to0^+$ independent of $j$ s.t. when $n$ large enough,
\begin{equation*}
    \|A_{j,n}^0(\widetilde{\varphibf}^k)-A_{j}^0(\widetilde{\varphibf}^k)\|_{\mathcal H_k^{\otimes2}}\leq \delta_ne^{-Cj^2}.
\end{equation*}
    \end{enumerate}
\end{lem}
\begin{proof}

 \begin{enumerate}[label=(\arabic*)]

\item Notice that $\mathcal A^0_n$ is just a free field annulus amplitude and $\| \mathcal A_{n}^0(\widetilde{\varphibf}^k)-\mathcal A_{}^0(\widetilde{\varphibf}^k)\|_{\mathcal H^{\otimes2}}$ consists of partition functions of gluing $\Sigma_n$ (or $\Sigma$) with $\Sigma$ (or $\Sigma_n$). Take the gluing of $\Sigma_n$ and reflected $\Sigma_n$ as an example. The convergence of $(\Sigma_n,g_n)$ makes sure that the determinant and the harmonic representatives of the monodromy also converge, so the partition function converges to that of $\Sigma$ gluing reflected $\Sigma$ as well. Thus we obtain the convergence.

\item  Recall that $h_n(f_n(z))=\log\frac{|zf_n'(z)|}{|f_n(z)|}$ on a neighborhood of $f(\mathbb D)$ and $g_n=e^{-2 h_n\chi}\frac{|\d z|^2}{|z|^2}$, where $\chi\in C^{\infty}(\cup_{n\geq N}\Sigma_n)$ for some large $N$ is equal to $1$ near $f(\mathbb T)$ and $0$ near $\mathbb T$. It is not hard to see that $h_n$ converges to $h$ uniformly on $\cap_{n\geq N}\Sigma_n$. 

    When $n>N$, we have $C^{-1} e^{-2h}\leq e^{-2 h_n\chi}\leq Ce^{-2 h}$ in the region $\cup_{n\geq N}\Sigma_n\backslash\Sigma^{\xi_0}\cap(1-\xi_0)\mathbb D$ for a fixed $\xi_0>0$, where the constant $C$ does not depend on $\xi$ and $n$. First notice that $\Sigma_n\backslash\Sigma^{\xi}$ contains two collars: the inner one $\Sigma_n\backslash\Sigma^{\xi}\cap(1-\xi)\mathbb D$ and the outer one near $\mathbb T$. The two collars do not intersect when $\xi$ is sufficiently small, and we will only estimate the integral for the inner one, since the other ones are simple. Besides, when $n$ is large enough, $f_n^{-1}((\Sigma_n\backslash\Sigma^{\xi})\cap(1-\xi)\mathbb D)\subset (1+C'\xi)\mathbb D$ and the two integrals in \eqref{eq_metric_integral_continuity} restricted to $(\Sigma_n\backslash\Sigma^{\xi})\cap(1-\xi)\mathbb D)$ are further bounded by
\begin{equation*}
   \iint_{(\mathbb A_{1,1+C'\xi})^2}|x-y|^{-\beta^2}\ \d x\d y, \ \ \ \int_{1}^{1+C'\xi}|r-1|^{-\frac{\beta^2}2}r\ \d r,
\end{equation*}
where $\mathbb A_{1,1+C'\xi}:=\{z:1\leq|z|\leq 1+C'\xi\}$ ($C'\xi<\epsilon$). Therefore we complete the proof.
\item It is plain to see that $\nu_{f_n,1}(\nbf_n)$ is smooth on $\partial\Sigma_n$ and converges in $C^\infty$ to $\nu_{f,1}(\nbf)$, therefore the proof follows from \cite[Lemma 3.12 (6)]{Segal_semigroup}. Especially,
\begin{equation*}
    |\log A_{j,n}^0(\widetilde{\varphibf}^k)-\log A_{j}^0(\widetilde{\varphibf}^k)|\leq \delta'_n(1+j^2+k^2+\|\varphibf\|_{H^{-s}}^2),
\end{equation*}
where $\delta'_n\to0^+$ and $\|\varphibf\|_{H^{-s}}$ is the norm on Sobolev space $H^{-s}(\mathbb T)^{2},s>0$.\qedhere
 \end{enumerate}
\end{proof}
One major difference compared to the real Liouville quadratic form is that the form $ \Qc_\v$ here is not positive. Therefore, some items (mainly (1), (7) and (8)) in \cite[Assumption 4.7]{Segal_semigroup} need modifications.
We point out that our $\mathcal S$ is the $\mathcal S_>$ defined in \cite[Section 3.3]{Segal_semigroup}, i.e. analytic on $\mathbb T$.
\begin{assumption}\label{assmp_diff}
   \mbox{}
    \begin{enumerate}[label=(\arabic*), ref=\arabic*]
        \item\label{itm_diff_1} There exists a measure space $(M,\mu_0)$ such that, if $\mathcal H=L^2(M,\mu_0)$ is the Hilbert space of square integrable functions, we are given a general quadratic form $\mathcal Q$ and a symmetric positive one $\widetilde{\mathcal Q}^0$ with $\mathcal D(\mathcal Q)=\mathcal D(\widetilde{\mathcal Q}^0)\subset \mathcal H$. The form $\mathcal Q$ generates a continuous semigroup $(e^{-t\mathbf H})_{t\geq0}$ on $\mathcal H$, with a non self-adjoint generator $\mathbf H$. The semigroup also maps $\mathcal D'(\widetilde{\mathcal Q}^0)\to\mathcal D(\widetilde{\mathcal Q}^0)$ continuously. The form $\widetilde{\mathcal Q}^0$ corresponds to a positive self-adjoint operator $\widetilde{\mathbf H}^0$. The Sobolev-like norm $\la F,G\ra_{\widetilde{\mathbf H}^0,s}:=\la ({\widetilde{\mathbf H}^0})^sF,G\ra_{L^2}$ induces space $W^s$ as the closure of $F\in\Ec(\DT)$ with $\|F\|_{\widetilde{\mathbf H}^0,s}$. 
        \item We are given a closed subset $\mathcal {AS}$ of $\mathcal S$, stable by composition, and then we consider the space $\mathcal{AS}_\epsilon:=\mathcal {AS}\cap \mathcal{S}_\epsilon$ equipped with the induced topology. Let $\mathcal A\mathrm{Hol}^{\bullet}(\mathbb D)\subset \mathrm{Hol}^{\bullet}(\mathbb D)$ be the tangent space of $\mathcal{AS}$ and $\mathcal A\mathrm{Hol}^{\bullet}_\epsilon(\mathbb D):=\mathcal A\mathrm{Hol}^{\bullet}(\mathbb D)\cap \mathrm{Hol}^{\bullet}_\epsilon(\mathbb D)$.
        \item The subspace $\mathcal {AS}$ is preserved by the flow. More precisely, let $v\in\mathcal{A}\mathrm{Hol}^{\bullet}_\epsilon(\mathbb D)$ Markovian and $f_0\in\mathcal{AS}_{\epsilon}$ (resp. $f_0(z)=z$). Then there exists $\epsilon'<\epsilon$ (locally uniform in $v$ and depending on $f_0$) s.t. the solution $f_t$ to the flow $\partial_t f_t=v(f_t)$ belongs to $\mathcal{AS}_{\epsilon'}$ for all $t\geq0$ (resp. for all $t>0$).
        \item If $f\in\mathcal{AS}_\epsilon$ then for $t_0>0$ small enough, $e^{t_0}f(e^{t_0}\cdot)\in\mathcal{AS}_{\epsilon'}$, where $(1+\epsilon')e^{t_0}=1+\epsilon$.
        \item\label{itm_5} For $f\in\mathcal {AS}$, the operator $\mathbf T_f$ extends to a bounded operator on $\mathcal H$, i.e. $\mathbf T_f\in\mathcal L(\mathcal H)$ and obeys the composition rule $\mathbf T_{f_1}\circ\mathbf T_{f_2}=\mathbf T_{f_1\circ f_2}$.
        \item\label{itm_6} The mapping $f\in\mathcal {AS}\mapsto \mathbf T_f$ is continuous in $\mathcal L(\mathcal H)$.
        \item\label{itm_diff_7} There exists a  linear map $v\in\cup_{\epsilon>0}\mathcal A {\mathrm{Hol}^\bullet_\epsilon(\mathbb D) }\to \mathcal Q_\v$ where $\v=v(z)\partial_z$ and $\mathcal Q_\v$ is a continuous sesquilinear form on $\mathcal D(\widetilde{\mathcal Q}^0)^2$, equipped with $\widetilde{\mathcal Q}^0$-norm. For all $\epsilon>0$, the restriction of this map $v\in\mathcal A {\mathrm{Hol}^\bullet_\epsilon(\mathbb D) }\to \mathcal Q_\v$ is $\|\cdot\|_{\epsilon,s}$-continuous (for some $s\geq0$).
        \item\label{itm_coercivity} Let $\v=v(z)\partial_z$ with $v(z)=-\sum_{n\geq0}v_n z^{n+1}\in\mathcal A {\mathrm{Hol}^\bullet_\epsilon(\mathbb D)}$. If $\Re\omega>0$ is large enough in the sense that there exists $C'>0$ and $s>0$ such that $v$ satisfies
\begin{equation*}
    D(v):=\Re\omega-C'(|\Im\omega|^2+\sum_{n\geq-1,n\not=0}(1+|n|)^{2s}|v_n|^2)^{1/2}>0,
\end{equation*}
then the bilinear form $\mathcal Q_\v$ generates a semigroup $(e^{-t\mathbf H_\v})_{t\geq0}$ on $\mathcal H$, for some bounded operator $\mathbf H_\v:\mathcal D(\widetilde{\mathcal Q}^0)\to\mathcal D'(\widetilde{\mathcal Q}^0)$. Furthermore, we can find $s'>0,\xi>0$, and positive constants $C_1,C_2$ (that depend on $\beta$ only) such that for all $ F\in\mathcal D(\widetilde{\mathcal Q}^0)$ and $N>0$,
\begin{equation}\label{eq_assump_N_lower}
    \Re(\mathcal Q_\v(F,F))+N\|F\|_{L^2}^2\geq (D(v)-C_1N^{-\xi}\|v\|_{\epsilon=0,s'}) \ \widetilde{\mathcal Q}^0(F,F)+(N-C_1N^{1-\xi}\|v\|_{\epsilon=0,s'}-C_2\Re\omega)\|F\|^2_{L^2}.
\end{equation}
In particular, for $v(z)=-z$, we have $\mathcal Q_\v=\mathcal Q$.
        \item\label{itm_9} If $v\in\mathcal A {\mathrm{Hol}^\bullet_\epsilon(\mathbb D) }$ satisfies the coercivity condition of item \eqref{itm_coercivity}, and $f_t$ denotes its flow with $f_0(z)=z$, then $e^{-t\mathbf H_\v}=\mathbf T_{f_t}$.
        \item \label{itm_10}If $f\in\mathcal {AS}_\epsilon$ and $v\in\mathcal A {\mathrm{Hol}^\bullet_\epsilon(\mathbb D) }$ such that $v/f'$ satisfies the coercivity condition of item \eqref{itm_coercivity}, we set $f_t(z)=f^{-1}(f(z)+tv(z))$. We assume that for $\delta>0$ small enough and for some fixed $\epsilon'\in(0,\epsilon)$, the map
        \begin{equation*}
            h:s,t\in[0,\delta]\mapsto \frac{v\circ f^{-1}_t(e^sz)}{f'(e^sz)}\in\mathcal A {\mathrm{Hol}^\bullet_{\epsilon'}(\mathbb D)}
        \end{equation*}
        is continuous.
        \item If $f\in\mathcal{AS}_\epsilon$ and $v\in\mathcal A {\mathrm{Hol}^\bullet_\epsilon(\mathbb D) }$, then $-\omega zf'(z)+v(z)\in\mathcal A{\mathrm{Hol}^\bullet_\epsilon(\mathbb D) }$ for all $\omega>0$ large enough.
\end{enumerate}
\end{assumption}
\begin{prop}\label{prop_general_T_f_diff}
   Under \cref{assmp_diff}, the map $f\in\mathcal {AS}_\epsilon\mapsto\Tbf_f\in\Lc(\mathcal H)$ is differentiable with differential
\[D_\v\Tbf_f = -\Tbf_f\circ\Hbf_{\rm w},\qquad {\rm w} = \tfrac{v(z)}{f'(z)}\,\pa_z,\]
where $\v=v(z)\pa_z\in\mathcal A {\mathrm{Hol}^\bullet_\epsilon(\mathbb D)} $. More precisely, given a differentiable family $t\mapsto f_t\in\mathcal {AS}_\epsilon$ with $f_0=f$ and $\pa_t|_{t=0}f_t=\v$ (in the sense that $\pa_t|_{t=0}f_t(z)=v(z)$ for $z\in\Db$), we have
\[\lim_{t\to0}\left\|\frac{\Tbf_{f_t}F-\Tbf_fF}{t} +\Tbf_f\Hbf_{\rm w}F\right\|_{\Dc(\wt{\Qc}^0)}=0,\qquad\fr F\in\Dc(\wt{\Qc}^0).\]

\end{prop}

\begin{proof}


We still need to handle the non-autonomous flows as in \cite[Section 4.1.2]{Segal_semigroup}. Consider a non-autonomous flow
\begin{equation*}
    \partial_t f_t=w_t(f_t), \ \ \ \text{with }f_0(z)=z,\ \ \ \text{where }w_t(z)=\frac{v(f_t^{-1}(z))}{f'(z)}. 
\end{equation*}
Consider $v\in\mathcal A\mathrm{Hol}^\bullet_\epsilon(\mathbb D)$ such that $v/f'\in\mathcal A\mathrm{Hol}^\bullet_\epsilon(\mathbb D)$ and satisfies the coercivity condition of item \eqref{itm_coercivity}. According to \cite[Lemma 4.6]{Segal_semigroup}, with item \eqref{itm_10}, for some $\epsilon'<\epsilon$, we can then find $\delta>0$ and $\eta>0$ such that
\begin{equation}\label{eq_con4-7}
\begin{split}
    &\forall t\leq\delta,\forall s\in[0,\eta]\ \ \ z\mapsto w_t(e^s z)\in\mathcal A\mathrm{Hol}^\bullet_{\epsilon'}(\mathbb D),\text{there are some uniform }D,N>0,\\
    &\text{ such that} \Re(\mathcal Q_{w_t(e^s z)\partial_z}(F,F))+N\|F\|_{L^2}^2\geq D\widetilde{\mathcal Q}^0(F,F),
\end{split}    
\end{equation}
and
\begin{equation}\label{eq_con4-8}
   \text{for some }\epsilon''\in(\epsilon',\epsilon),\text{ the mapping }t\in[0,\delta]\mapsto f_t^{-1}\in\mathcal S_{\epsilon''}\text{ is continuous.}
\end{equation}
Fix $n>0$ and set $t_l:=l\delta/n$ for $0\leq l\leq n$. We construct recursively the flow $f_t^n$ for $t\in[0,\delta]$ as follows: we set $f_t^n=f_t^{n,l}$ for $t\in[t_l,t_{l+1}]$ where $f^{n,l}$ is constructed inductively on $l$ as the solution of the autonomous flow
\begin{equation*}
    \forall t\in[t_l,t_{l+1}],\ \ \partial_t f_t^{n,l}=w_{t_l}(f_t^{n,l}), \ \ \text{with initial condition }f_{t_l}^{n,l}(z)=f_{t_l}^{n,l-1}(z).
\end{equation*}

In the following, we pick the $N>0$ in \eqref{eq_con4-7} (as we have \eqref{eq_assump_N_lower}) such that the estimate for $e^{-Nt}\mathbf T_{f_t^n}$ can be adapted from \cite[Lemma 4.9]{Segal_semigroup}.
\begin{lemma}
    For each $F\in\mathcal D(\widetilde{\mathcal{Q}}^0)$, the mapping $t\in[0,\delta]\mapsto e^{-Nt}\mathbf T_{f_t^n}F$ converges in $C([0,\delta],L^2)\cap L^2([0,\delta],\mathcal D(\widetilde{\mathcal Q}^0))$, as $n\to\infty$, to the mapping $t\in[0,\delta]\mapsto e^{-Nt}\mathbf T_{f_t}F$ satisfying, for ${\mathrm w}_t=w_t(z)\partial_z$,
    \begin{equation}\label{eq_increment_id}
        \forall G\in\mathcal D(\widetilde{\mathcal{Q}}^0),\ \ \ \la e^{-Nt}\mathbf T_{f_t}F,G\ra_{L^2}-\la F,G\ra_{L^2}=-\int_0^t\mathcal Q_{{\mathrm w}_r}(e^{-Nr}\mathbf T_{f_r}F,G)+N\la e^{-Nr}\mathbf T_{f_r}F,G\ra\ \d r.
    \end{equation}
\end{lemma}
\begin{proof}
    Similar to the discussion at the beginning of \cite[Lemma 4.9]{Segal_semigroup}, $T_{f_t^n}F$ converges to $T_{f_t}F$ in $L^2$ because of item \eqref{itm_6} and \eqref{eq_con4-8}. 

    On each interval $[t_l,t_{l+1}]$, we still have for the autonomous flow using item \eqref{itm_diff_7}, \eqref{itm_coercivity}
    \begin{equation*}
        \forall G\in\mathcal{D}(\widetilde{\mathcal Q}^0),\ \ \ \la e^{-Nt}\mathbf T_{f_t^n}F,G\ra_{L^2}-\la e^{-Nt_l}\mathbf T_{f_{t_l}^n}F,G\ra_{L^2}=-\int_{t_l}^t\mathcal Q_{{\mathrm w}_{t_l}}(e^{-Nr}\mathbf T_{f_r^n}F,G)+N\la e^{-Nr}\mathbf T_{f_r^n}F,G\ra_{L^2}\ \d r.
    \end{equation*}
    Furthermore, we have
    \begin{equation}\label{eq_TfnTfm}
    \begin{split}
        \la e^{-Nt}\mathbf T_{f_t^n}F,e^{-Nt}\mathbf T_{f_t^m}F\ra_{L^2}-\|F\|^2_{L^2}&=-\int_{0}^t\mathcal Q_{{\mathrm w}_{r}^n}(e^{-Nr}\mathbf T_{f_r^n}F,e^{-Nr}\mathbf T_{f_r^m}F)+N\la e^{-Nr}\mathbf T_{f_r^n}F,e^{-Nr}\mathbf T_{f_r^m}F\ra_{L^2}\ \d r\\
        &-\int_{0}^t\overline{\mathcal Q}_{{\mathrm w}_{r}^m}(e^{-Nr}\mathbf T_{f_r^n}F,e^{-Nr}\mathbf T_{f_r^m}F)+N\la e^{-Nr}\mathbf T_{f_r^n}F,e^{-Nr}\mathbf T_{f_r^m}F\ra_{L^2}\ \d r,
    \end{split}
    \end{equation}
    where $w_t^n(z):=w_{t_l}(z)$ for $t\in[t_l,t_{l+1}]$. Especially, we have the a priori estimate
    \begin{equation*}
        \|e^{-Nt}\mathbf T_{f_t^n}F\|_{L^2}^2+2\int_{0}^t\Re \mathcal Q_{{\mathrm w}_{r}^n}(e^{-Nr}\mathbf T_{f_r^n}F,e^{-Nr}\mathbf T_{f_r^n}F)+N\| e^{-Nr}\mathbf T_{f_r^n}F\|_{L^2}^2\ \d r=\|F\|^2_{L^2},
    \end{equation*}
 and applying \eqref{eq_con4-7} to the left-hand side leads to
    \begin{equation}\label{eq_apriori_estimate}
    \sup_{t\in[0,\delta]}\|e^{-Nt}\mathbf T_{f_t^n}F\|_{L^2}^2+C\int_0^\delta \widetilde{\mathcal Q}^0(e^{-Nr}\mathbf T_{f_r^n}F,e^{-Nr}\mathbf T_{f_r^n}F)\ \d r\leq 2\|F\|^2_{L^2}.
    \end{equation}
Denote by $\Delta^{n,m}_t:=e^{-Nt}\mathbf T_{f_t^n}F-e^{-Nt}\mathbf T_{f_t^m}F$. Using \eqref{eq_TfnTfm} and item \eqref{itm_diff_7}, we obtain
\begin{equation*}
    \begin{split}
        \|\Delta_{t}^{n,m}\|^2_{L^2}&+2\int_{0}^t\Re \mathcal Q_{{\mathrm w}_{r}^n}(\Delta_{r}^{n,m},\Delta_{r}^{n,m})+N\| \Delta_{r}^{n,m}\|_{L^2}^2\ \d r=-2\int_0^t\Re(\mathcal Q_{{\mathrm w}_{r}^n}-\mathcal Q_{{\mathrm w}_{r}^m})(e^{-Nr}\mathbf T_{f_r^m}F,\Delta_{r}^{n,m})\ \d r\\
        &\leq 2C' \sup_{t\in[0,\delta]}\|w_t^n-w_t^m\|_{\epsilon',s}\int_0^t\sqrt{\widetilde{\mathcal Q}^0(e^{-Nr}\mathbf T_{f_r^m}F,e^{-Nr}\mathbf T_{f_r^m}F)\widetilde{\mathcal Q}^0(\Delta_{r}^{n,m},\Delta_{r}^{n,m})}\ \d r\\
        &\leq \frac1\lambda(C')^2 \sup_{t\in[0,\delta]}\|w_t^n-w_t^m\|^2_{\epsilon',s}\int_0^t{\widetilde{\mathcal Q}^0(e^{-Nr}\mathbf T_{f_r^m}F,e^{-Nr}\mathbf T_{f_r^m}F)\, \d r+\lambda\int_0^t\widetilde{\mathcal Q}^0(\Delta_{r}^{n,m},\Delta_{r}^{n,m})}\, \d r,
    \end{split}
\end{equation*}
where $\lambda$ is arbitrarily small. For the left-hand side of the first line above, we use \eqref{eq_assump_N_lower}; for the last line we use equation \eqref{eq_apriori_estimate}, then we arrive at
\begin{equation*}
    \sup_{t\in[0,\delta]}\|\Delta_{t}^{n,m}\|_{L^2}^2+\int_0^\delta\widetilde{\mathcal Q}^0(\Delta_{r}^{n,m},\Delta_{r}^{n,m})\ \d r\leq C'' \sup_{t\in[0,\delta]}\|w_t^n-w_t^m\|^2_{\epsilon',s}\|F\|^2_{L^2}.
\end{equation*}
Now the rest of the proof is the same as \cite[Lemma 4.9]{Segal_semigroup}.
\end{proof}
Recall the norm $\|\cdot\|_{\widetilde{\mathbf H}^0,s}$ defined in item \eqref{itm_diff_1} of Assumption \ref{assmp_diff}. It is not hard to see that $W^{-1}=\mathcal D'(\widetilde{\mathcal Q}^0),W^0=\mathcal H,W^1=\mathcal D(\widetilde{\mathcal Q}^0)$. Moreover, we have
\begin{equation*}
    \forall F,G\in\mathcal D(\widetilde{\mathcal Q}^0), \ \ \ |\widetilde{\mathcal Q}^0(F,G)|\leq\|F\|_{\widetilde{\mathbf H}^0,1}\|G\|_{\widetilde{\mathbf H}^0,1},
\end{equation*}
and
\begin{equation*}
    \forall F\in\mathcal D(\widetilde{\mathcal Q}^0),\forall G\in\mathcal D'(\widetilde{\mathcal Q}^0) \ \ \ |\la F,G\ra_{L^2}|\leq\|F\|_{\widetilde{\mathbf H}^0,1}\|G\|_{\widetilde{\mathbf H}^0,-1}.
\end{equation*}
We also have the Sobolev inequality
\begin{equation*}
    \forall F\in\mathcal D(\widetilde{\mathcal Q}^0),\ \ \ \|F\|_{\widetilde{\mathbf H}^0,-1}\leq C\|F\|_{\widetilde{\mathbf H}^0,0}\leq C\|F\|_{\widetilde{\mathbf H}^0,1},
\end{equation*}
where $C$ is some constant independent of $F$. Recall $\eta>0$ appeared in \eqref{eq_con4-7}. Below we provide a weaker version of \cite[Lemma 4.10]{Segal_semigroup} since we only have $e^{-t\mathbf H}:W^{-1}\to W^1$.
\begin{lemma}\label{lemma_Tf_prop}
    For $f\in\mathcal A \mathcal S_\epsilon$, there is some operator $\mathbf T^{\eta}_f$ such that $\mathbf T_f=\mathbf T_f^{\eta} e^{-\eta\mathbf H}$ with operator norm
    \begin{equation*}
        \|\mathbf T_f\|_{W^{-1}\to W^{1}}<+\infty\ \ \ \ \text{and}\ \ \ \ \|\mathbf T_f^\eta\|_{W^{-1}\to W^{1}}<+\infty.
    \end{equation*}
    The same property and operator norm estimates hold for the adjoints $\mathbf T^*_f$ and $(\mathbf T^\eta_f)^*$.
\end{lemma}
The proof for the above lemma is exactly the same as in \cite[Lemma 4.10]{Segal_semigroup}. Recall the $\delta$ in \eqref{eq_con4-7}. Now we claim
\begin{lemma}\label{lemma_TF-Id_is_sqt}
    There exists a constant $C>0$ such that for $t\in(0,\delta)$,
    \begin{equation*}
        \|e^{-Nt}\mathbf T_{f_t}-{\mathrm{Id}}\|_{\mathcal D(\widetilde{\mathcal{Q}}^0)\to \mathcal H}\leq C \sqrt{t}.
    \end{equation*}
\end{lemma}
\begin{proof}
    Let us follow the proof of \cite[Lemma 4.11]{Segal_semigroup}. From \eqref{eq_increment_id}, we have
    \begin{equation*}
        \|e^{-Nt}\mathbf T_{f_t} F-F\|_{L^2}^2=-2\int_{0}^t\Re \left(\mathcal Q_{{\mathrm w}_r}(e^{-Nr}\mathbf T_{f_r} F,e^{-Nr}\mathbf T_{f_r} F-F)+N\la e^{-Nr}\mathbf T_{f_r} F,e^{-Nr}\mathbf T_{f_r} F-F\ra_{L^2}\right)\ \d r.
    \end{equation*}
    Using \eqref{eq_con4-7}, we get
    \begin{equation*}
    \begin{split}
        &\|e^{-Nt}\mathbf T_{f_t} F-F\|_{L^2}^2+C\int_{0}^t\widetilde{\mathcal Q}^0(e^{-Nr}\mathbf T_{f_r} F-F,e^{-Nr}\mathbf T_{f_r} F-F)\ \d r\\
        \leq{}&-2\int_{0}^t\Re \left(\mathcal Q_{{\mathrm w}_r}(F,e^{-Nr}\mathbf T_{f_r} F-F)+N\la F,e^{-Nr}\mathbf T_{f_r} F-F\ra_{L^2}\right)\ \d r.
    \end{split}
    \end{equation*}
   Based on item \eqref{itm_diff_7} and item \eqref{itm_10}, the right-hand side above can be further bounded by
    \begin{equation*}
        \frac t\lambda C\|F\|_{\widetilde{\mathbf H}^0,1}^2+\lambda(C+N)\int_0^t\|e^{-Nr}\mathbf T_{f_r} F-F\|_{\widetilde{\mathbf H}^0,1}^2\ \d r+\frac t\lambda N\|F\|_{\widetilde{\mathbf H}^0,-1}^2.
    \end{equation*}
    Make $\lambda$ sufficiently small and using Sobolev inequality, we complete the proof.
\end{proof}
We follow the rest of the proof in \cite[Section 4.1.3]{Segal_semigroup}. Since $t^{-1}\la e^{-Nt}F-F,G\ra$ clearly exists, it is enough to study $\la\mathbf e^{-Nt}T_{f_t}F,G\ra_{L^2}-\la F,G\ra_{L^2}$.  First using \eqref{eq_increment_id}, we have
\begin{equation*}
\begin{split}
    \la\mathbf e^{-Nt}T_{f_t}F,G\ra_{L^2}-\la F,G\ra_{L^2}&=-\int_0^t\mathcal Q_{{\mathrm w}_r}(e^{-Nr}\mathbf T_{f_r}F,G)+N\la e^{-Nr}\mathbf T_{f_r}F,G\ra\ \d r\\
    ={}&-t\mathcal Q_{{\mathrm w}_0}(F,G)-Nt\la F,G\ra_{L^2}-\int_0^t(\mathcal Q_{{\mathrm w}_r}(F,G)-\mathcal Q_{{\mathrm w}_0}(F,G))\ \d r\\
    &-\int_0^t\mathcal Q_{{\mathrm w}_r}(e^{-Nr}\mathbf T_{f_r}F-F,G)+N\la e^{-Nr}\mathbf T_{f_r}F-F,G\ra_{L^2}\ \d r\\
    =&:-t\mathcal Q_{{\mathrm w}_0}(F,G)-Nt\la F,G\ra_{L^2}-R^1_t(G)-R^2_t(G).
\end{split}
\end{equation*}
Especially, we put $\mathbf T_f^* G$ with $ G\in\mathcal D'(\widetilde{\mathcal Q}_0)$ in the above equation, and 
\begin{equation*}
     \la e^{-Nt}\mathbf T_{f}\mathbf T_{f_t}F,G\ra_{L^2}-\la \mathbf T_{f}F,G\ra_{L^2}=-t\mathcal Q_{{\mathrm w}_0}(F,\mathbf T_f^*G)-Nt\la F,\mathbf T_f^*G\ra_{L^2}-R^1_t(\mathbf T_f^*G)-R^2_t(\mathbf T_f^*G).
\end{equation*}
Using item \eqref{itm_diff_7} for $\mathcal Q_{{\mathrm w}_0}$ and \cref{lemma_Tf_prop}, we have
\begin{equation*}
    |\mathcal Q_{{\mathrm w}_0}(F,\mathbf T_f^*G)|\leq C\|F\|_{\widetilde{\mathbf H}^0,1}\|\mathbf T_f^*G\|_{\widetilde{\mathbf H}^0,1}\leq C\|F\|_{\widetilde{\mathbf H}^0,1}\|G\|_{\widetilde{\mathbf H}^0,-1}.
\end{equation*}
By the Riesz representation theorem, we can find an element $D_{\v}\mathbf T_f F\in\mathcal D(\widetilde{\mathcal{Q}}^0)$ such that $\la D_{\v}\mathbf T_f F,G\ra_{L^2}=-\mathcal Q_{{\mathrm w}_0}(F,\mathbf T_f^*G)$ for all $G\in\mathcal D'(\widetilde{\mathcal Q}^0)$. Now we just prove $R^1_t,R^2_t$ are $o(t)$ in operator norm $\mathcal D(\widetilde{\mathcal{Q}}^0)\to \mathcal D(\widetilde{\mathcal{Q}}^0)$.

For $R_t^2$, using \cite[Lemma 4.12]{Segal_semigroup} (the proof is the same except that $\mathcal Q_\v(F,e^{-t\mathbf H^*}G)=\mathcal Q_{\v_t}(e^{-t\mathbf H}F,G)$ where $v_t(z):=e^{-t}v(e^tz)$) and \cref{lemma_Tf_prop}, we have
\begin{equation*}
    |R_t^2(\mathbf T_f^*G)|=\left|\int_0^t\mathcal Q_{{\mathrm w}_r}(e^{-Nr}\mathbf T_{f_r}F-F,e^{-\eta\mathbf H^*}(\mathbf T_f^\eta)^*G)+N\la e^{-Nr}\mathbf T_{f_r}F-F,\mathbf T_f^*G\ra_{L^2} \d r\right|,
\end{equation*}
which leads to the same estimate in \cite{Segal_semigroup} using \cite[Lemma 4.12]{Segal_semigroup} ($\mathcal Q_\v(F,e^{-t\mathbf H^*}G)=\mathcal Q_{\v_t}(e^{-t\mathbf H}F,G)$), item \eqref{itm_diff_7}, and \cref{lemma_TF-Id_is_sqt}
\begin{equation*}
    |R_t^2(\mathbf T_f^*G)|\leq C\int_0^t\|e^{-Nr}\mathbf T_{f_r}F-F\|_{L^2}\|G\|_{\widetilde{\mathbf H}^0,-1}\ \d r\leq C t^{3/2}\|F\|_{\widetilde{\mathbf H}^0,1}\|G\|_{\widetilde{\mathbf H}^0,-1}.
\end{equation*}
The proof of $R_t^1$ is the same except that in \cite[Lemma 4.14]{Segal_semigroup}, we bound the difference of quadratic form using $\widetilde{\mathcal Q}^0$.

 The rest of the proof is the same as those in \cite[Section 4.1]{Segal_semigroup}.
\end{proof}
Now let us verify \cref{assmp_diff} for CILT. Especially we only need to check items \eqref{itm_diff_1}, \eqref{itm_5}, \eqref{itm_6}, \eqref{itm_diff_7}, \eqref{itm_coercivity}, \eqref{itm_9}, whereas the rest of \cref{assmp_diff} is made on the flow thus the same as the real case (see \cite[Lemma 4.5, Lemma 4.6]{Segal_semigroup}, and note that $\{(z,v):\Re(\bar z v(z))<0\}$ is open).

Item \eqref{itm_diff_1} is done in \cref{HilleYosida_lemma}. Item \eqref{itm_5} is proved in \cref{T_f_semigroup}. Item \eqref{itm_6} is just \cref{prop_Tf_continuity}.

Now we check Item \eqref{itm_diff_7}. Let $v(z)=-\sum_{n\geq0}v_n z^{n+1}\in\mathrm{Hol}_\epsilon^{\bullet}(\mathbb D)$. The quadratic form $ \Qc_\v$ is $\widetilde{\mathcal Q}^0$-continuous with bound
    \begin{equation}\label{eq_vnorm_bound_Q}
        |\Qc_\v(F,G)|\leq C\|v\|_{\epsilon=0,3}\sqrt{\widetilde{\mathcal Q}^0(F,F)\widetilde{\mathcal Q}^0(G,G)}
    \end{equation}
for some constant $C$ depending on $\mu$ and $\beta$. To see this, we point out that
\begin{equation*}
    |\Qc_\v^0(F,G)|\leq C\|v\|_{\epsilon=0,3}\sqrt{\widetilde{\mathcal Q}^0(F,F)\widetilde{\mathcal Q}^0(G,G)}
\end{equation*}
which comes from  \cref{L_n^0_estimate}, and the bound in \cref{V_v_estimate} with $N$ set to $0$,
    \begin{equation*}
        |\langle V_\v F,G\rangle_{L^2}|\leq C\sup_{\theta\in[0,2\pi]}|\Re(e^{-\i\theta}v(e^{\i\theta}))|\,\sqrt{\widetilde{\mathcal Q}^0(F,F)\widetilde{\mathcal Q}^0(G,G)},
    \end{equation*}
    where the constant $C$ only depends on $\beta$.

Items \eqref{itm_coercivity} and \eqref{itm_9} are consequences of \cref{Q_v_bound} and \cref{prop_FK_formula}. Finally we have the following conclusion.
\begin{coro}\label{T_f_differentiable}
The map $f\in\Ssg\mapsto\Tbf_f\in\Lc(L^2(\DT))$ is differentiable with differential
\[D_\v\Tbf_f = -\Tbf_f\circ\Hbf_{\rm w},\qquad {\rm w} = \tfrac{v(z)}{f'(z)}\,\pa_z,\]
where $\v=v(z)\pa_z$ is any holomorphic vector field on $\Db$ satisfying $v(0)=0$. More precisely, given a differentiable family $t\mapsto f_t\in\Ssg$ with $f_0=f$ and $\pa_t|_{t=0}f_t=\v$ (in the sense that $\pa_t|_{t=0}f_t(z)=v(z)$ for $z\in\Db$), we have
\[\lim_{t\to0}\left\|\frac{\Tbf_{f_t}F-\Tbf_fF}{t} +\Tbf_f\Hbf_{\rm w}F\right\|_{\Dc(\wt{\Qc}^0)}=0,\qquad\fr F\in\Dc(\wt{\Qc}^0).\]
\end{coro}

Note that by \cref{HilleYosida_lemma}, the operator $\Tbf_f$ extends to a bounded operator $\Dc'(\wt{\Qc}^0)\to\Dc(\wt{\Qc}^0)$, so $\Tbf_f\Hbf_{\rm w}F\in\Dc(\wt{\Qc}^0)$ for $F\in\Dc(\wt{\Qc}^0)$.

\section{Variation of amplitudes and Ward identities}\label{sec:ward}

As an application of the semigroup theory, we now know how to differentiate CILT amplitudes and can then use this to prove the Ward identities. These results are crucial to the conformal bootstrap program, which will be investigated in future works.

First we need to extend the definition \eqref{H_v} to holomorphic vector fields $\v=-\sum_{n\in\Zb}v_nz^{n+1}\pa_z$ defined only on (a neighborhood of) $\Tb$, i.e., we set
\[\Hbf_\v:=\sum_{n\in\Zb}v_n\,\Lbf_n+\sum_{n\in\Zb}\ol{v_n}\,\wt{\Lbf}_n.\]
Using the analog of \cref{V_v_estimate}, it is easy to see that this series converges in the space of bounded operators $\Dc(\wt{\Qc}^0)\to\Dc'(\wt{\Qc}^0)$.

We would like to compute the differential of the amplitude of a general surface under a deformation of the boundary. Extracting an annulus along each boundary circle and applying the gluing axiom, we see that it suffices to consider the variation of amplitudes of annuli. Let $\Sc(\Ab)$ be the set of holomorphic embeddings $\Tb\to\Db^\circ$. As before, to an element $f\in\Sc(\Ab)$, we associate the annulus $\Ab_f\subset\Db$ bounded by $\Tb$ and $f(\Tb)$, with boundary parametrizations $\zetabf_f:=(\id_\Tb,f|_\Tb)$.

\begin{lem}\label{annulus_differential}
Let $t\mapsto f_t\in\Ssg(\Ab)$ be a differentiable family with $f_0:z\mapsto rz$ ($0<r<1$) and $\pa_t|_{t=0}f_t=\v=-\sum_{n\in\Zb}v_nz^{n+1}\pa_z$ a holomorphic vector field on $\Tb$. Let $g_{f_t}$ be an admissible metric on $\Ab_{f_t}$ depending in a $C^1$ fashion on $t$ with $g_{f_0}=g_\Ab=|dz|^2/|z|^2$. Then the map $t\mapsto\Ac_{\Ab_{f_t},g_{f_t},\zetabf_{f_t}}\in\Lc(L^2(\DT))$ is differentiable at $t=0$ with differential (in the same sense as in \cref{T_f_differentiable})
\[\pa_t\big|_{t=0}\Ac_{\Ab_{f_t},g_{f_t},\zetabf_{f_t}} = -c_\L\left(-\frac{\Re v_0}{12r} + \pa_t\big|_{t=0} S_\L^0(\Ab_{f_t},g_{f_t},g_\Ab)\right)\Ac_{\Ab_{f_0},g_\Ab,\zetabf_{f_0}}-\Ac_{\Ab_{f_0},g_\Ab,\zetabf_{f_0}}\circ\Hbf_{\v/r}.\]
\end{lem}
\begin{proof}
This follows from \cref{T_f_differentiable} in the same way as in \cite[Section 4.3]{Segal_semigroup}, noting that taking the adjoint of CILT amplitudes requires replacing $\mu,\beta$ by $\bar{\mu},-\beta$. We give a self-contained proof here for the sake of completeness.

First we need to consider amplitudes of outer annuli. Let $\iota:z\mapsto 1/\bar{z}$ and $\Db^*:=\iota(\Db)=\hat{\Cb}\setminus\Db^\circ$ the closed outer unit disk. Note that $\iota|_\Tb=\id_\Tb$. For $f\in\Ssg$, the map $f^\iota:=\iota\circ f\circ\iota$ is a holomorphic embedding $\Db^*\to(\Db^*)^\circ$ with $f^\iota(\infty)=\infty$, and we associate to it the outer annulus $\Ab_{f^\iota}:=\iota(\Ab_f)=\Db^*\setminus f^\iota((\Db^*)^\circ)$ with boundary parametrizations $\zetabf_{f^\iota}:=(f^\iota|_\Tb,\id_{\Tb})$. Then for an admissible metric $g_f$ on $\Ab_f$, we have
\[\Ac_{\Ab_{f^\iota},\iota_* g_f,\zetabf_{f^\iota}}^{\mu,\beta} = (\Ac_{\Ab_f,g_f,\zetabf_f}^{\bar{\mu},-\beta})^*,\]
where we have made explicit the dependence on $\mu,\beta$. Indeed, by diffeomorphism invariance, we have for $\wt\varphi_1^k,\wt\varphi_2^k\in\DT$
\[\Ac_{\Ab_{f^\iota},\iota_* g_f,\zetabf_{f^\iota}}^{\mu,\beta}(\wt\varphi_1^k,\wt\varphi_2^k) = \Ac_{\Ab_{f},g_f,(f|_\Tb,\id_\Tb)}^{\mu,\beta}(\wt\varphi_1^k,\wt\varphi_2^k) = \Ac_{\Ab_{f},g_f,\zetabf_f}^{\mu,\beta}(\wt\varphi_2^k,\wt\varphi_1^k) =  \ol{\Ac_{\Ab_{f},g_f,\zetabf_f}^{\bar{\mu},-\beta}(\wt\varphi_2^k,\wt\varphi_1^k)}.\]
If $t\mapsto f_t\in\Ssg$ is the flow on $\Db$ generated by a Markovian holomorphic vector field $\v=-\sum_{n=0}^\infty v_nz^{n+1}\pa_z$, then $t\mapsto f_t^\iota$ is the flow on $\Db^*$ generated by $\v^\iota:=\sum_{n=0}^\infty\ol{v_n}z^{-n+1}\pa_z$, so that by \cref{prop_semigroup_as_amplitude},
\[\frac{e^{-\frac{c_\L}{12}W(f_t,g_{f_t})}}{\sqrt{2}\pi}\Ac_{\Ab_{f_t^\iota},\iota_* g_{f_t},\zetabf_{f_t^\iota}}^{\mu,\beta} = (\Tbf_{f_t}^{\bar{\mu},-\beta})^* = e^{-t(\Hbf_\v^{\bar{\mu},-\beta})^*} = e^{t\Hbf_{\v^\iota}^{\mu,\beta}},\]
where $g_{f_t}$ is any admissible metric on $\Ab_{f_t}$ and we have used \eqref{L_n_adjoint}. Similarly, defining $\Tbf_{f^\iota}^{\mu,\beta}:=(\Tbf_f^{\bar{\mu},-\beta})^*$ for $f\in\Ssg$, we have analogously to \cref{T_f_differentiable}
\[\pa_t\big|_{t=0}\Tbf_{f_t^\iota}=\Hbf_{{\rm w}^\iota}\circ\Tbf_{f^\iota},\qquad{\rm w}=\tfrac{v(z)}{f'(z)}\pa_z,\]
where $t\mapsto f_t\in\Ssg$ is a differentiable family with $f_0=f$ and $\pa_t|_{t=0}f_t=\v$.

The proof relies on the observation that an annulus in $\Ssg(\Ab)$ can be decomposed into an inner annulus and an outer annulus, both parametrized by $\Ssg$. Indeed, for $f\in\Sc(\Ab)$, consider the closed Riemann surface $S_f:=(\Db^*\cup\Ab_f)\sqcup\Db/\!\sim$ where the quotient relation is $z\in\pa\Db\sim f(z)\in\pa(\Db^*\cup\Ab_f)$. This is topologically a sphere, so there exists a uniformizing biholomorphism $\Phi_f:S_f\to\hat\Cb$, which is unique if we require that $\Phi_f(0)=0$, $\Phi_f'(0)=a$, $\Phi_f(\infty)=\infty$, where $r<a<1$ is fixed. We apply this with $f=f_t\in\Ssg(\Ab)$ as in the statement of the lemma and let $f_{{\rm out},t}:=\iota\circ(\Phi_{f_t}|_{\Ab_{f_t}}\circ\id_\Tb)\circ\iota$, $f_{{\rm in},t}:=\Phi_{f_t}|_{\Ab_{f_t}}\circ f_t|_\Tb$, which both extend to holomorphic maps on $\Db$. We have $\Phi_{f_0}=\frac{a}{r}\id$ on $\Ab_{f_0}$, so for $t$ small, we have $\Tb\subset\Phi_{f_t}(\Ab_{f_t}^\circ)$, so that $f_{{\rm out},t},f_{{\rm in},t}\in\Ssg$. Note that $f_{{\rm out},0}=\frac{r}{a}\id$, $f_{{\rm in},0}=a\;\!\id$. Let $\v_{\rm out}:=\pa_t|_{t=0}f_{{\rm out},t}$, $\v_{\rm in}:=\pa_t|_{t=0}f_{{\rm in},t}$. Differentiating the relation $f_{{\rm out},t}^\iota\circ f_t=f_{{\rm in},t}$ on $\Tb$ at $t=0$, we obtain
\begin{equation}\label{v_in_v_out_relation}
\frac{a^2}{r^2}v_{\rm out}^\iota(rz)+\frac{a}{r}v(z)=v_{\rm in}(z),\qquad z\in\Tb.
\end{equation}
By conformal covariance, we may assume that
$(\Phi_{f_t})_*g_{f_t}=g_\Ab$ near $\Tb$.
Indeed, this can be achieved by a smooth conformal change
supported away from the original boundary components and vanishing
at $t=0$. Since $g_{f_0}=g_\Ab$ is flat, the corresponding Weyl
anomaly is $O(t^2)$, so this modification does not change the
derivative at $t=0$. Now by gluing,
\[\Ac_{\Ab_{f_t},g_{f_t},\zetabf_{f_t}}=\frac{1}{\sqrt{2}\pi}\,\Ac_{\Ab_{f_{{\rm out},t}^\iota},\iota_*g_{f_{{\rm out},t}},\zetabf_{f_{{\rm out},t}^\iota}}\circ\Ac_{\Ab_{f_{{\rm in},t}},g_{f_{{\rm in},t}},\zetabf_{f_{{\rm in},t}}},\]
where $g_{f_{{\rm out},t}}:=\iota_*(\Phi_{f_t}|_{\Ab_{f_t}})_*g_{f_t}|_{\Ab_{f_{{\rm out},t}^\iota}}$, $g_{f_{{\rm in},t}}:=(\Phi_{f_t}|_{\Ab_{f_t}})_*g_{f_t}|_{\Ab_{f_{{\rm in},t}}}$. Thus
\begin{equation}\label{eq_annulus_differential}
\begin{aligned}
\pa_t\big|_{t=0}\Ac_{\Ab_{f_t},g_{f_t},\zetabf_{f_t}} &= \frac{c_\L}{12}\big(\pa_t\big|_{t=0}(W(f_{{\rm out},t},g_{f_{{\rm out},t}})+W(f_{{\rm in},t},g_{f_{{\rm in},t}}))\big)\Ac_{\Ab_{f_0},g_\Ab,\zetabf_{f_0}}\\
&\hspace{8em}+\Hbf_{\frac{a}{r}\v_{\rm out}^\iota}\Ac_{\Ab_{f_0},g_\Ab,\zetabf_{f_0}}-\Ac_{\Ab_{f_0},g_\Ab,\zetabf_{f_0}}\Hbf_{\frac{1}{a}\v_{\rm in}},
\end{aligned}
\end{equation}
where we used \cref{prop_semigroup_as_amplitude,T_f_differentiable} and the remarks above on outer annuli.

We claim that
\[\Hbf_{\rm w}\Tbf_{f_0}=\Tbf_{f_0}\Hbf_{\frac{1}{r}{\rm w}(r\cdot)}\]
for any holomorphic vector field ${\rm w}=-\sum_{n\in\Zb}w_nz^{n+1}\pa_z$ on $\Ab_{f_0}$, so that the last two terms in \eqref{eq_annulus_differential} add up to $-\Ac_{\Ab_{f_0},g_\Ab,\zetabf_{f_0}}\Hbf_{\v/r}$ by \eqref{v_in_v_out_relation}. This is equivalent to the commutation relation $\Lbf_ne^{-t\Hbf}=e^{-nt}e^{-t\Hbf}\Lbf_n$ ($n\in\Zb$) with $t=-\log r$. By duality, we may assume $n\geq0$, i.e., ${\rm w}$ is holomorphic on $\Db$. By linearity, it suffices to consider the case where ${\rm w}$ is Markovian. Letting $t\mapsto h_t$ be the flow generated by ${\rm w}$, we have $-\Hbf_{\rm w}\Tbf_{f_0}=\pa_t|_{t=0}\Tbf_{h_t}\Tbf_{f_0}=\pa_t|_{t=0}\Tbf_{h_t\circ f_0}$, so it suffices to compute $\pa_t|_{t=0}(h_t\circ f_0)={\rm w}(r\cdot)$ and apply again \cref{T_f_differentiable}.

It remains to compute the factor of the first term in \eqref{eq_annulus_differential}. By definition,
\[S_\L^0(\Ab_{f_{{\rm out},t}^\iota},\iota_*g_{f_{{\rm out},t}},g_\Ab)+S_\L^0(\Ab_{f_{{\rm in},t}},g_{f_{{\rm in},t}},g_\Ab) = S_\L^0(\Ab_{f_t},g_{f_t},(\Phi_{f_t}|_{\Ab_{f_t}})^*g_\Ab).\]
Suppose $g_\Ab=e^{\omega_t}g_{f_t}$, $(\Phi_{f_t}|_{\Ab_{f_t}})^*g_\Ab=e^{\omega_t+u_t}g_{f_t}$. Then
\[S_\L^0(\Ab_{f_t},g_{f_t},(\Phi_{f_t}|_{\Ab_{f_t}})^*g_\Ab)-S_\L^0(\Ab_{f_t},g_{f_t},g_\Ab) = \frac{1}{96\pi}\int_{\Ab_{f_t}}\left(2\la d\omega_t,du_t\ra_{g_{f_t}}+|du_t|_{g_{f_t}}^2+2K_{g_{f_t}}u_t\right)\d v_{g_{f_t}}.\]
Since $\omega_0=u_0=0$, $K_{g_{f_0}}=K_{g_\Ab}=0$, we obtain $\pa_t\big|_{t=0}S_\L^0(\Ab_{f_t},g_{f_t},(\Phi_{f_t}|_{\Ab_{f_t}})^*g_\Ab)=\pa_t\big|_{t=0}S_\L^0(\Ab_{f_t},g_{f_t},g_\Ab)$.
By \eqref{v_in_v_out_relation} again,
\[\pa_t\big|_{t=0}\big(\log|f_{{\rm out},t}'(0)|+\log|f_{{\rm in},t}'(0)|\big) = \Re\Big(\frac{a}{r}v_{\rm out}'(0)+\frac{1}{a}v_{\rm in}'(0)\Big) = -\frac{\Re v_0}{r}.\]
Finally, we have $\pa_t\big|_{t=0}\big\langle\Dbf\log\frac{|f_{{\rm out},t}'|}{|f_{{\rm out},t}|}\big|_\Tb,\log\frac{|f_{{\rm out},t}'|}{|f_{{\rm out},t}|}\big|_\Tb\big\rangle=0$ since $\Dbf\log\frac{|f_{{\rm out},0}'|}{|f_{{\rm out},0}|}\big|_\Tb=0$, and similarly for $f_{{\rm in},t}$.
\end{proof}

We are ready to formulate and prove the Ward identities for CILT amplitudes:

\begin{prop}[\textbf{Ward identity}]\label{ward_cilt}
Let $\Sigma$ be a compact Riemann surface with real-analytic boundary parametrizations $\zetabf=(\zeta_1,\ldots,\zeta_\b)$. Suppose $\b=\b^++\b^-$ where $\zeta_j$ is outgoing for $j=1,\ldots,\b^+$ and incoming for $j=\b^++1,\ldots,\b$. Let $\zbf=(z_1,\ldots,z_\s)$ be distinct marked points on $\Sigma^\circ$, $\vbf=(v_1,\ldots,v_\s)$ with $v_j\in T_{z_j}\Sigma\setminus\{0\}$ for each $j$, $\alphabf=(\alpha_1,\ldots,\alpha_\s)\in\Rb^\s$ electric charges such that $\alpha_j>Q$, $\alpha_jR\in\Zb$ for all $j$, and $\mbf=(m_1,\ldots,m_\s)\in\Zb^\s$ magnetic charges. Let $g$ be an admissible metric on $\Sigma$. We have the amplitude operator $\Ac_{\Sigma,g,\vbf,\alphabf,\mbf,\zetabf}:\Hc^{\otimes\b^-}\to\Hc^{\otimes\b^+}$ with $\Hc=L^2(\DT)$. Finally, let ${\rm w}$ be a holomorphic vector field on $\Sigma$ (analytic up to the boundary). For each $j=1,\ldots,\b$, let ${\rm w}_j:=\zeta_j^*{\rm w}=-\sum_{n\in\Zb}w_{jn}z^{n+1}\pa_z$, where we extend $\zeta_j$ analytically so that ${\rm w}_j$ is holomorphic on an open neighborhood of $\Tb$. Then the following identity holds:
\begin{align*}
0 &= \sum_{j=1}^{\b^+}\Hbf_{{\rm w}_j}^{(j)}\Ac_{\Sigma,g,\vbf,\alphabf,\mbf,\zetabf} - \sum_{j=1}^{\b^-}\Ac_{\Sigma,g,\vbf,\alphabf,\mbf,\zetabf}\Hbf_{{\rm w}_{\b^++j}}^{(j)}
+ \sum_{j=1}^{\s}\left(\Delta_{\alpha_j,m_j}(\div_gW)(z_j)+W_\vbf^{(j)}\right)\Ac_{\Sigma,g,\vbf,\alphabf,\mbf,\zetabf} \\
&\qquad+ c_\L\left(-\sum_{j=1}^{\b^+}\frac{\Re w_{j0}}{12}+\sum_{j=\b^++1}^{\b}\frac{\Re w_{j0}}{12}-\frac{1}{96\pi}\int_\Sigma 2K_g\div_gW\,\d v_g\right)\Ac_{\Sigma,g,\vbf,\alphabf,\mbf,\zetabf},
\end{align*}
where $\Hbf_{{\rm w}}^{(j)}:=\id_\Hc \otimes \cdots \otimes \underset{j}{\underbrace{\Hbf_{\rm w}}} \otimes \cdots \otimes \id_\Hc$ acts on the $j$-th $\Hc$, $W={\rm w}+\bar{{\rm w}}$, and $W_\vbf^{(j)}$ denotes differentiation with respect to $v_j$ in the direction of $W$ at $z_j$.
\end{prop}
\begin{proof}
See \cite[Corollary 4.8]{block} for the counterpart in real Liouville CFT; both the statement and the proof are parallel.
Consider the closed Riemann surface $\hat{\Sigma}$ obtained by gluing a disk $\Db$ or $\Db^*$ along each boundary circle $\zeta_j(\Tb)\subset\pa\Sigma$ via $\zeta_j$. We extend ${\rm w}$ to a neighborhood of $\Sigma$ in $\hat{\Sigma}$. For $t$ small, the flow $t\mapsto f_t$ of ${\rm w}$ exists near $\Sigma$, and $(f_t)_*g$ is well-defined. By diffeomorphism invariance, $\Ac_{f_t(\Sigma),(f_t)_*g,(f_t)_*\vbf,\alphabf,\mbf,f_t\circ\zetabf}=\Ac_{\Sigma,g,\vbf,\alphabf,\mbf,\zetabf}$ is constant in $t$. We shall compute the derivative $\pa_t\big|_{t=0}\Ac_{f_t(\Sigma),(f_t)_*g,(f_t)_*\vbf,\alphabf,\mbf,f_t\circ\zetabf}=0$ in another way. For notational simplicity, suppose $\b^+=0$, i.e., all boundaries are incoming. The general case follows from the analog of \cref{annulus_differential} for outer annuli, obtained by taking suitable adjoints.

Now each $\zeta_j$ extends to a holomorphic embedding $\Db\to\hat{\Sigma}$. Take $r<1$ such that each $\zeta_j$ extends to a holomorphic embedding $r^{-1}\Db\to\hat{\Sigma}$. Let $\xibf=(\xi_1,\ldots,\xi_\b)$ where $\xi_j=\zeta_j(r^{-1}\cdot)$ for each $j$. Let $\Sigma_0:=\Sigma\setminus\bigcup_{j=1}^\b\xi_j(\Db^\circ)$ with boundary parametrizations $\xibf|_{\Tb}$. By admissibility and shrinking $r$ if necessary, we may assume that $\xi_j^*g=g_\Ab$ on $\{r-\varepsilon<|z|<1+\varepsilon\}$ for some $\varepsilon>0$, so that in particular $g|_{\Sigma_0}$ is admissible. For each $j$, let $f_{j,t}:=\xi_j^{-1}\circ f_t\circ\zeta_j\in\Ssg(\Ab)$ and $g_{f_{j,t}}$ an admissible metric on $\Ab_{f_{j,t}}$ depending in a $C^1$ fashion on $t$ with $g_{f_{j,0}}=g_\Ab=|dz|^2/|z|^2$. We have $f_{j,0}:z\mapsto rz$ and $\pa_t|_{t=0}f_{j,t}=r{\rm w}_j$. Let $g_t$ be the metric on $f_t(\Sigma)$ such that $g_t|_{\Sigma_0}=g|_{\Sigma_0}$ and $\xi_j^*g_t=g_{f_{j,t}}$ for each $j$. By gluing,
\[\Ac_{f_t(\Sigma),g_t,(f_t)_*\vbf,\alphabf,\mbf,f_t\circ\zetabf} = C\Ac_{\Sigma_0,g|_{\Sigma_0},(f_t)_*\vbf,\alphabf,\mbf,\xibf|_\Tb} \circ \bigotimes_{j=1}^\b \Ac_{\Ab_{f_{j,t}},g_{f_{j,t}},\zetabf_{f_{j,t}}},\]
where $C$ is a constant depending only on $\b$. Thus
\begin{align*}
0 &= \pa_t\big|_{t=0}\Ac_{f_t(\Sigma),(f_t)_*g,(f_t)_*\vbf,\alphabf,\mbf,f_t\circ\zetabf}\\
&= \left(c_\L\pa_t\big|_{t=0}S_\L^0(f_t(\Sigma),g_t,(f_t)_*g)+\sum_{j=1}^\s\big(-\Delta_{\alpha_j,m_j}\pa_t\big|_{t=0}\rho_t(z_j)+W_\vbf^{(j)}\big)\right) \Ac_{\Sigma,g,\vbf,\alphabf,\mbf,\zetabf}\\
&\phantom{={}}+\sum_{j=1}^\b\left(-c_\L\left(-\frac{\Re w_{j0}}{12}+\pa_t\big|_{t=0}S_\L^0(\Ab_{f_{j,t}},g_{f_{j,t}}, g_\Ab)\right)\Ac_{\Sigma,g,\vbf,\alphabf,\mbf,\zetabf} - \Ac_{\Sigma,g,\vbf,\alphabf,\mbf,\zetabf}\circ\Hbf_{{\rm w}_j}^{(j)}\right),
\end{align*}
where $(f_t)_*g=e^{\rho_t}g_t$ and we used the Weyl anomaly formula (\cite[Proposition 8.12]{CILT}, \cite[Theorem 7.6]{BCILT}) and \cref{annulus_differential}. Note that $t\mapsto f_t$ is also the flow generated by the real vector field $W=2\Re(w(z)\pa_z)$. Since $g_t=g_0$ on $\Sigma_0$, we have $\pa_t\big|_{t=0}\rho_t=-\div_gW$ on $\Sigma_0$. It remains to collect the $S_\L^0$ terms, which give $c_\L$ times
\[\pa_t\big|_{t=0}\left(S_\L^0(f_t(\Sigma),g_t,(f_t)_*g)-\sum_{j=1}^\b S_\L^0(\Ab_{f_{j,t}},g_{f_{j,t}}, g_\Ab)\right).\]
For each $j$,
\[S_\L^0(\xi_j(\Ab_{f_{j,t}}),g_t,(f_t)_*g)-S_\L^0(\Ab_{f_{j,t}},g_{f_{j,t}}, g_\Ab) = S_\L^0(\xi_j(\Ab_{f_{j,t}}),g_t,(f_t)_*g)-S_\L^0(\xi_j(\Ab_{f_{j,t}}),g_t,g) = S_\L^0(\xi_j(\Ab_{f_{j,t}}),g,(f_t)_*g),\]
where we used $\xi_j^*g=g_\Ab$ on $\Ab_{f_{j,t}}$ for $t$ small. Since $g_t=g$ on $\Sigma_0$, the above is equal to
\[\pa_t\big|_{t=0}S_\L^0(f_t(\Sigma),g,(f_t)_*g)=-\frac{1}{96\pi}\int_\Sigma 2K_g\div_gW\,\d v_g.\qedhere\]
\end{proof}

\section{Semigroup theory in Boundary CILT}\label{sec_BCILT}

In this section, we develop the analogous semigroup theory for Boundary CILT. Since the theory is largely parallel, we shall only state the main results and point out where proofs differ.

Segal's semigroup in the Boundary CFT is the subset of elements $f\in\Ssg$ such that $f(\Db\cap\Rb)\subset\Rb$ and $f'(0)>0$. Then $f(\bar{z})=\ol{f(z)}$ for $z\in\Cb$ and $f(\Db^+)\subset\Db^+$. Such maps constitute a subsemigroup of $\Sc$ and we will denote it by $\Ssg^+$. To an element $f\in\Ssg^+$, we associate the upper semiannulus $\Ab_f^+:=\Ab_f\cap\Hb=\ol{\Db^+\setminus f(\Db^+)}$ with Neumann boundary $\Ab_f^+\cap\Rb=[-1,f(-1)]\sqcup[f(1),1]$ and Dirichlet boundary $\Tb^+\sqcup f(\Tb^+)$ parametrized by $\zetabf_f^+:=(\id_{\Tb^+},f|_{\Tb^+})$ (\Cref{fig:segal-semiannulus}). Again, composition on $\Sc^+$ corresponds to gluing the semiannuli. One-parameter semigroups in $\Ssg^+$ are generated by Markovian holomorphic vector fields $\v=v(z)\pa_z$ on $\Db$ satisfying $v(\bar{z})=\ol{v(z)}$ for $z\in\Cb$, i.e., $v(z)=-\sum_{n=0}^\infty v_nz^{n+1}$ with $v_n\in\Rb$ for all $n$. We shall call such vector fields \textit{symmetric}.

The quantization of $\Ssg^+$ by BCILT is defined as follows. Fix parameters $\mubf=(\mu,\mu^{\rm R},\mu^{\rm L})\in\Cb^3$. For $f\in\Ssg^+$, we define the operator $\Tbf_f^+$ on $L^2(\DTh)$ as follows. First let $\Ec(\DTh)$ be the linear space spanned by functions $F:\DTh\to\Cb$ of the form
\begin{equation}
F(\wt{\varphi}^\h) = e^{\i nc/R}\, P(\la\varphi^\h,h_1\ra,\ldots,\la\varphi^\h,h_m\ra),
\end{equation}
where $n\in\Zb$, $m\in\Nb$, $P$ is a complex polynomial, $h_1,\ldots,h_m\in C^\infty(\Tb)$. Here the bracket is still $\la\varphi^\h,h_j\ra:=\frac{1}{2\pi}\int_\Tb\varphi^\h\,\ol{h_j}\,\d\theta$, where $\varphi^\h$ is viewed as an even distribution on $\Tb$, i.e., $\varphi^\h(-\theta)=\varphi^\h(\theta)$.
For $F\in\Ec(\DTh)$, we define
\begin{equation}\label{eq_def_T_f^+}
(\Tbf_f^+F)(\wt{\varphi}^\h) := |f'(0)|^{-\frac{Q^2}4} \,\Eb\left[F\left((P\wt{\varphi}^\h+X_{\Db^+})\circ f\big|_{\Tb^+} + \i Q\log\tfrac{|f'|}{|f|}\big|_{\Tb^+}\right)\,e^{-\int_{\Ab_f^+}M_\beta^{\mubf}(P\wt{\varphi}^\h+X_{\Db^+},\d x)}\right],
\end{equation}
where $X_{\Db^+}$ is the GFF on $\Db^+$ with Dirichlet boundary $\Tb^+$ and Neumann boundary $\Db^+\cap\Rb=[-1,1]$, the expectation is over $X_{\Db^+}$, and the potential term is (with $X=P\wt{\varphi}^\h+X_{\Db^+}$)
\begin{equation}
M_\beta^{\mubf}(X,\d x):=\mu\,|x|^{\beta Q}M_\beta(X,\d x) + \mu^{\rm R}\,{\bf1}_{x>0}\,|x|^{\frac{\beta Q}{2}}L_\beta(X,\d x) + \mu^{\rm L}\,{\bf1}_{x<0}\,|x|^{\frac{\beta Q}{2}}L_\beta(X,\d x),
\end{equation}
where $M_\beta$ is the bulk imaginary GMC and $L_\beta$ is the boundary imaginary GMC.


The infinitesimal generator of the quantized semigroup $f\mapsto\Tbf_f^+$ is given as follows. We define the following operators on $\Ec(\DTh)$:
\begin{equation}
\partial_0:=\partial_c,
\qquad\partial_n:=\partial_{\varphi_n^\h}=\sqrt {2n}\partial_{x_n^\h},
\qquad n>0,
\end{equation}
\begin{equation}
\Abf_0^+:=\i(\pa_c+\i\tfrac{Q}{2}),
\qquad\mathbf A_n^+:=\tfrac\i2(\partial_n+2n\varphi_{-n}^\h\textbf{1}_{n<0}),\qquad n\neq0.
\end{equation}
For $n\in\Zb$, we define
\begin{equation}
\begin{gathered}
\Lbf_n^+ := (n+1)Q\Abf_n^+ + \sum_{m\in{\mathbb Z}}:\!\Abf_{n-m}^+\Abf_m^+\!: +\; V_n^{\mubf},\\
V_n^\mubf := \mu \,e^{\i\beta c}\int_0^{\pi}e^{\i\beta\varphi^\h(\theta)}\,(2|\!\sin\theta|)^{\frac{\beta^2}{2}}\cos n\theta\,\d\theta + \mu^{\rm R}\,e^{\i\frac{\beta}2c}\,e^{\i\frac{\beta}2\varphi^\h(0)} + \mu^{\rm L}\,(-1)^n\,e^{\i\frac{\beta}2c}\,e^{\i\frac{\beta}2\varphi^\h(\pi)},
\end{gathered}
\end{equation}
where the rigorous definition of $V_n^\mubf$ will be given shortly. For a symmetric holomorphic vector field $\v = -\sum_{n=0}^\infty v_n z^{n+1} \pa_z$, we define
\begin{equation}
\Hbf_\v^+:=\sum_{n=0}^\infty v_n\,\Lbf_n^+.
\end{equation}
The case $\v=-z\pa_z$
\begin{equation}
\Hbf^+:=\Hbf_{-z\pa_z}^+=\Lbf_0^+
\end{equation}
is an additional Hamiltonian in BCILT. As always, we add a superscript $^0$ to denote the result of setting $\mubf=0$. To make sense of these expressions, we introduce the positive symmetric quadratic form $\wt{\Qc}^{+,0}$ on $L^2(\DTh)$ given by
\begin{equation}\label{Qtilde+0}
\wt{\Qc}^{+,0}(F,F) :=  \|\pa_cF\|_{L^2}^2+\frac{Q^2}4\|F\|_{L^2}^2+2\sum_{n\geq1}\|\Abf_n^+F\|_{L^2}^2.
\end{equation}
One checks that it is closable and satisfies $\wt{\Qc}^{+,0}(F,G)=\la \Hbf^{+,0}F,G\ra_{L^2}+\frac{Q^2}2\la F,G\ra_{L^2}$ for $F,G\in\Ec(\DTh)$. We have similar bounds for $\Lbf_n^{+,0}$ in terms of $\wt{\Qc}^{+,0}$ and a constant $C$ depending on $Q,R$,
\begin{equation*}
\forall F,G\in\Ec(\DTh),\ \     |\langle\mathbf L^{+,0}_nF,G\rangle_{L^2}|\leq C(1+|n|)\sqrt{\wt{\Qc}^{+,0}(F,F)\wt{\Qc}^{+,0}(G,G)}
\end{equation*}
as in \cref{L_n^0_estimate}, so that the operators $\Lbf_n^{+,0}$ ($n\in\Zb$) extend to bounded operators $\Dc(\wt{\Qc}^{+,0})\to\Dc'(\wt{\Qc}^{+,0})$, and the series $\Hbf_\v^{+,0}=\sum_{n=0}^\infty v_n\,\Lbf_n^{+,0}$ converges in the space of such operators. 

Now let us define the potential operators. For a symmetric holomorphic vector field $\v = -\sum_{n\geq0} v_n z^{n+1} \pa_z$ on $\Db$, we define $V_\v^\mubf$ by the following regularization procedure. For $\delta>0$, we define the arc average $\widetilde{\varphi}^{\text{h,arc}}_\delta(\theta):=\frac{1}{\int_{[\theta-\delta,\theta+\delta]\cap[0,\pi]}\d\theta'}\int_{[\theta-\delta,\theta+\delta]\cap[0,\pi]}\widetilde{\varphi}(\theta')\d\theta'$ and
\begin{gather*}
V_\v^{\cap,(\delta)}:=\int_\delta^{\pi-\delta}\,e^{\i\beta\widetilde{\varphi}^{\text{h,arc}}_\delta(\theta)+\frac{\beta^2}{2}\mathbb E[{\varphi}^{\text{h,arc}}_\delta(\theta)^2]}(2|\!\sin\theta|)^{\frac{\beta^2}{2}}\,\d\varrho_\v(\theta),\\
V^{+,{\rm R},(\delta)}:=e^{\i\frac\beta2\widetilde{\varphi}^{\text{h,arc}}_\delta(0)+\frac{\beta^2}{8}\mathbb E[\varphi^{\h,\text{arc}}_\delta(0)^2]},\qquad V^{+,{\rm L},(\delta)}:=e^{\i\frac\beta2\widetilde{\varphi}^{\text{h,arc}}_\delta(\pi)+\frac{\beta^2}{8}\mathbb E[\varphi^{\h,\text{arc}}_\delta(\pi)^2]},
\end{gather*}
where $\d\varrho_\v$ is as in \eqref{defn_V_v}. Then $V_\v^\mubf$ is defined by
\[\la V_\v^\mubf F,G\ra_{L^2}:=\lim_{\delta\to0}\langle V_\v^{\mubf,(\delta)}F,G\rangle_{L^2}\]
for $F,G\in\Ec(\DTh)$, where
\[
V_\v^{\mubf,(\delta)}:=\mu\, V_\v^{\cap,(\delta)}- \mu^{\rm R}\,v(1)\,V^{+,{\rm R},(\delta)}+\mu^{\rm L}\,v(-1)\,V^{+,{\rm L},(\delta)}.\]
The existence of these limits is a simple Cameron--Martin calculation. In fact, the limits $V^{+,{\rm R}}:=\lim_{\delta\to0}V^{+,{\rm R},(\delta)}$ and $V^{+,{\rm L}}:=\lim_{\delta\to0}V^{+,{\rm L},(\delta)}$ satisfy for $F,G\in\Ec(\DTh)$,
\begin{equation}
\begin{aligned}
\la V^{+,{\rm R}}F,G\ra_{L^2} &= \la e^{\i\frac\beta2 c}F(\cdot-\i\beta\log|e^{\i\theta}-1|),G(\cdot+\i\beta\log|e^{\i\theta}-1|)\ra_{L^2},\\
\la V^{+,{\rm L}}F,G\ra_{L^2} &= \la e^{\i\frac\beta2 c}F(\cdot-\i\beta\log|e^{\i\theta}+1|),G(\cdot+\i\beta\log|e^{\i\theta}+1|)\ra_{L^2}.
\end{aligned}
\end{equation}
The operators $V_n^\mubf$ ($n\in\Zb$) are defined similarly. We have the following analog of \cref{V_v_estimate}:

\begin{lemma}\label{Q_v^+_bound}
Let $N\geq0$. For $F,G\in\Ec(\DTh)$, we have
\begin{equation}\label{eq_V_bound_bdy}
|\langle V_\v^\mubf  F,G\rangle_{L^2}|\leq C\, (N+Q^2/4)^{\frac{\beta^2}{2}-1} \sup_{\theta\in[0,2\pi]}|\Re(e^{-\i\theta}v(e^{\i\theta}))|\,\sqrt{(N\|F\|_{L^2}^2+\widetilde{\mathcal Q}^{+,0}(F,F))(N\|G\|_{L^2}^2+\widetilde{\mathcal Q}^{+,0}(G,G))},
\end{equation}
where $C>0$ only depends on $\beta$ and $\mubf$. In particular, $V_\v^\mubf$ extends to a bounded operator $\Dc(\wt{\Qc}^{+,0})\to\Dc'(\wt{\Qc}^{+,0})$.
\end{lemma}
See \cref{lem_App_2} for the proof.

Let $\v= -\sum_{n=0}^\infty v_n z^{n+1} \pa_z$ be a symmetric Markovian holomorphic vector field on $\Db$ such that \eqref{omega_condition} holds for some $K>0$ large enough. Then one can proceed as in \cref{Q_v_bound} and apply \cref{HilleYosida_lemma} to obtain a closed unbounded operator $\Hbf_\v^+$ with domain $\Dc(\Hbf_\v^+)\subset\Dc(\wt{\Qc}^{+,0})$ such that it generates a strongly continuous semigroup of bounded operators $t\mapsto e^{-t\Hbf_\v^+}$ on $L^2(\DTh)$ and that its extension to a bounded operator $\Dc(\wt{\Qc}^{+,0})\to\Dc'(\wt{\Qc}^{+,0})$ coincides with the limit of $\sum_{n\geq0}v_n\,\Lbf_n^+$.

The following theorem summarizes the main results on the boundary semigroup theory:

\begin{thm}[\textbf{Boundary semigroup theory}]\mbox{}\label{thm_BCILT_semigroup}
\begin{enumerate}[label=\emph{(\arabic*)}]
\item Let $f\in\Ssg^+$ and $g_f$ a Neumann extendible admissible metric on $\Ab_f^+$ (i.e., it is the restriction to $\Ab_f^+$ of an even admissible metric on $\Ab_f$). Viewing the BCILT amplitude $\Ac_{\Ab_f^+,g_f,\zetabf_f^+}$ of $(\Ab_f^+,g_f,\zetabf_f^+)$ as an operator $L^2(\DTh)\to L^2(\DTh)$, we have
\[\Tbf_f^+=\frac{e^{-\frac{c_\L}{12}W^+(f,g_f)}}{(2\pi)^{3/4}}\Ac_{\Ab_f^+,g_f,\zetabf_f^+},\]
where $W^+(f,g_f):=\frac{1}{2}W(f,g_f)$ with $W$ as in \cref{prop_semigroup_as_amplitude}.
\item The map $f\in\Ssg^+\mapsto \Tbf_f^+\in\Lc(L^2(\DTh))$ is a semigroup homomorphism, i.e., $\Tbf_{f_1}^+\circ\Tbf_{f_2}^+=\Tbf_{f_1\circ f_2}^+$ for $f_1,f_2\in\Ssg^+$.
\item Let $\v= -\sum_{n=0}^\infty v_n z^{n+1} \pa_z$ be a symmetric Markovian holomorphic vector field on $\Db$ such that \eqref{omega_condition} holds for some $K>0$ large enough. Let $t\mapsto f_t$ be the flow generated by $\v$.
For $t>0$, we have $e^{-t\Hbf_\v^+}=\Tbf_{f_t}^+$.
\item The map $f\in\Ssg^+\mapsto \Tbf_f^+\in\Lc(L^2(\DTh))$ is continuous and differentiable with differential
\[D_\v\Tbf_f^+ = -\Tbf_f^+\circ\Hbf_{\rm w}^+,\qquad {\rm w} = \tfrac{v(z)}{f'(z)}\,\pa_z.\]
\end{enumerate}
\end{thm}
\begin{proof}[Sketch of proofs]
(1): The proof is the same Girsanov transform calculation as in \cref{prop_semigroup_as_amplitude}, upon applying the gluing axiom for the free field to $\Ab_f^+\#f(\Db^+)=\Db^+$. The only nontrivial part is the determinant term $\frac{\det'\Delta_{g|_{f(\Db^+)}}}{\det'\Delta_{g|_{\Db^+}}}$ where $g$ is a Neumann extendible metric on $\Db^+$ extending $g_f$. Now by the Polyakov formula \cite[Proposition 5.10]{Boundary_Segal} (for Neumann extendible metrics on surfaces with both Neumann and Dirichlet boundaries), we have
\[\frac{\det'\Delta_{g|_{f(\Db^+)}}}{\det'\Delta_{g|_{\Db^+}}}=e^{-2S_\L^0(\Db^+,g,f^*g)}=e^{-S_\L^0(\Db,g,f^*g)}=\left(\frac{\det'\Delta_{g|_{f(\Db)}}}{\det'\Delta_{g|_{\Db}}}\right)^{1/2},\]
where we view $g$ as an even metric on $\Db$. Thus we can apply \eqref{det_formula} to get that this is precisely $e^{\frac{1}{6}W^+(f,g_f)}$, as desired.

(2): That $\Tbf_f^+$ is bounded on $L^2(\DTh)$ follows from (1) and the boundedness of the BCILT amplitude (\cite[Theorem 7.5]{BCILT}). The rest follows from \eqref{W_additivity} as in the proof of \cref{T_f_semigroup}.

(3): We emphasize that, in the proof of Proposition \ref{prop_FK_formula}, we only used the existence of the GMC exponential moment, the convergence of $\mathbb E[\exp(|\mu \int_{\mathbb A_t}|x|^{\beta Q} |x-y|^{\beta^2}M_\beta(P\wt{\varphi}^k+X_{\Db},\d x)|)]$ to $1$ (uniformly over $y$ in a thin annulus), and the convergence of exponential moment of GMC cut-off (\cref{lemma_gmc_delta_exp_bound}). Similar to \cite[Proposition 7.6]{boundary_bootstrap}, we have
\begin{equation*}
     \int_{\Ab_{f_t}^+}|x|^{\beta Q} M_\beta(X^+,\d x)=-\int_{0}^t\int_0^{\pi} e^{\i\beta ((X^+\circ f_s)(e^{\i\theta})+\i Q\log\frac{|f'_s(e^{\i\theta})|}{|f_s(e^{\i\theta})|})}\Re\left(e^{-\i\theta}v(e^{\i\theta})\right)\, \d\theta\d s,
\end{equation*}
where $X^+:=P\wt{\varphi}^\h+X_{\Db^+}$. Also we have
\begin{equation*}
    \begin{split}
        &\int_{[f_t(1),1]}|x|^{\frac{\beta Q}{2}}L_\beta(X^+,\d x)=-v(1)\int_0^te^{\i\frac\beta2 ((X^+\circ f_s)(1)+\i Q\log\frac{|f'_s(1)|}{|f_s(1)|})}\d s,\\
        &\int_{[-1,f_t(-1)]}|x|^{\frac{\beta Q}{2}}L_\beta(X^+,\d x)=v(-1)\int_0^te^{\i\frac\beta2 ((X^+\circ f_s)(-1)+\i Q\log\frac{|f'_s(-1)|}{|f_s(-1)|})}\d s
    \end{split}
\end{equation*}
The control of the boundary GMC exponential moment can be found in \cite[Proposition 6.2]{BCILT}. Specifically, given a general function $f(x)$, for some constant $C$ depending on $\beta$,  we have
\begin{equation*}
    \begin{split}
        &\mathbb E\left[\exp\left(\left|\int_{\Ab_{f_t}^+} f(x)M_\beta(X^+,\d x)\right|+\left|\int_{\Ab_{f_t}^+\cap\mathbb R} f(x)L_\beta(X^+,\d x)\right|\right)\right]\\
        \leq{}& e^{C V(f)}\sqrt{\left(1+C\sqrt{U_1(\Sigma,f)}e^{CU_1(\Sigma,f)}\right)\left(1+C\sqrt{U_2(\paN\Sigma,f)}e^{CU_2(\paN\Sigma,f)}\right)},
    \end{split}
    \end{equation*}
     where
    \begin{equation}\label{eq_BGMC_bound}
    \begin{split}
        &U_1(\Sigma,f):=\iint_{(\Ab_{f_t}^+)^2}|f(x)f(y)|(|x-y|^{-\beta^2}+|x-\bar y|^{-\beta^2})\d x\d y,\\
        &U_2(\paN\Sigma,f):=\iint_{(\Ab_{f_t}^+\cap\mathbb R)^2}|f(x)f(y)||x-y|^{-\beta^2/2}\d x\d y,\\
        &V(f):=\int_{\Ab_{f_t}^+}|f(x)|\sup_{y\in\partial\Ab_{f_t}^+}|x-y|^{-\beta^2/2}\d x+\int_{\Ab_{f_t}^+\cap\mathbb R}|f(x)|\sup_{y\in\partial(\Ab_{f_t}^+\cap\mathbb R)}|x-y|^{-\beta^2/4}\,\d x.
    \end{split}
    \end{equation}  
And we can prove the boundary version of \cref{lemma_gmc_delta_exp_bound} by using again the martingale argument.

(4): The continuity follows the discussion in \cite[Proposition 3.11]{Segal_semigroup}, \cite[Lemma 7.10]{boundary_bootstrap} and \cref{prop_Tf_continuity}. Here we point out the main difference in the proof of a BCILT version compared to those. First, there is no magnetic operator on the boundary, so the amplitude becomes simpler. Next is the estimate of GMC exponential moment when the area is small. Similar to \eqref{eq_BGMC_bound}, for some $C,\xi>0$, we bound the small area GMC exponential moment with
\begin{equation*}
\begin{split}
    &\iint_{(\mathbb A_{1,1+C\xi})^2}|x-y|^{-\beta^2}\d x\d y,\ \  \iint_{[1,1+C\xi]^2}|x-y|^{-\frac{\beta^2}{2}}\d x\d y,\\
    &\int_{\mathbb A_{1,1+C\xi}^+}(||x|-1|\wedge\Im x)^{-\frac{\beta^2}2}\d x,\ \ \int_1^{1+C\xi}|x-1|^{-\frac{\beta^2}{4}}\d x,
\end{split}
\end{equation*}
where $\mathbb A_{1,1+C\xi}:=\{z:1\leq|z|\leq 1+C\xi\}$ and $\mathbb A_{1,1+C\xi}^+:=\mathbb A_{1,1+C\xi}\cap \mathbb H$. The integrals above can be controlled by $\xi^{a}$ for some $a>0$.

Next let us consider a sequence of surfaces $(\Sigma_n^+,g)_{n\geq1}$ and a surface $\Sigma^{+,\xi}\subset\cap_{n\geq N}\Sigma_n^+$ for large $N$. Denote by $\paD\Sigma_n^+$ the Dirichlet boundaries of $\Sigma_n^+$ and $\paN\Sigma_n^+$ the Neumann boundaries of $\Sigma_n^+$. Then we can construct GFF $X_n^+$ on $\Sigma_n^+$ with corresponding boundary conditions. We now show how to provide a bound for $\mathbb E_{\widetilde{\boldsymbol{\varphi}}^h}[e^{-\mu\int_{\Sigma^{+,\xi}}f(x)M_\beta(X_n^{+}+Z,\d v_g(x))}]$ uniformly in $n$ and in $Z$ ($Z$ must be real and independent of $X_n^+$) or $\widetilde{\boldsymbol{\varphi}}^h$ the value on the Dirichlet boundary $\paD\Sigma_n^+$.

Denote by $\Sigma_n$ and $\Sigma^\xi$ the Neumann double of $\Sigma^+_n$ and $\Sigma^{+,\xi}$ under reflection $\tau$. We have respectively the Dirichlet Green's function $(2\pi)^{-1}G_n$ and $(2\pi)^{-1}G^\xi$ and mixed Green's function $(2\pi)^{-1}G_n^+(x,y)=(2\pi)^{-1}(G_n(x,y)+G_n(x,\tau(y)))$ and $(2\pi)^{-1}G^{+,\xi}(x,y)=(2\pi)^{-1}(G^\xi(x,y)+G^\xi(x,\tau(y)))$. Since we have $G_n\geq G^\xi$, we also have $G_n^+\geq G^{+,\xi}$. Again by Taylor's expansion we have
\begin{equation*}
    \left|\mathbb E_{\widetilde{\boldsymbol{\varphi}}^h}[e^{-\mu\int_{\Sigma^{+,\xi}}f(x)M_\beta(X_n^{+}+Z,\d v_g(x))}]\right|\leq\mathbb E[e^{|\mu|\int_{\Sigma^{+,\xi}}|f(x)|M_\beta(X^{+,\xi},\d v_g(x))}],
\end{equation*}
where $X^{+,\xi}$ is the GFF on $\Sigma^{+,\xi}$ with covariance $G^{+,\xi}$. Then we just follow the proof of \cite[Lemma 7.10]{boundary_bootstrap} and \cref{prop_Tf_continuity}.
By verifying \cref{assmp_diff}, we can get differentiability from \cref{prop_general_T_f_diff} just as the proof of  \cite[Proposition 7.9]{boundary_bootstrap} in \cite[Section 7.5.1]{boundary_bootstrap} does.
\end{proof}

We also state the Ward identities for BCILT amplitudes. The proof is analogous to that in Section \ref{sec:ward}, so we omit it.

\begin{prop}[\textbf{Ward identity, boundary version}]\label{ward_bcilt}
Let $\Sigma$ be an extended surface (see \cite[Definition 1.1]{BCILT} for the notations) with real-analytic boundary parametrizations $\zetabf=(\zeta_1^\D,\ldots,\zeta_{\b_\D}^\D,\zeta_1^{\MD},\ldots,\zeta_{\b^\MD}^\MD)$. Suppose $\b^\D=\b^{\D,+}+\b^{\D,-}$ where $\zeta_j^\D$ is outgoing for $j=1,\ldots,\b^{\D,+}$ and incoming for $j=\b^{\D,+}+1,\ldots,\b^\D$, and similarly for $\b^{\MD}=\b^{\MD,+}+\b^{\MD,-}$. Let $\zbf=(z_1,\ldots,z_\s)$ be distinct marked points on $\Sigma^\circ$, $\vbf=(v_1,\ldots,v_\s)$ with $v_j\in T_{z_j}\Sigma\setminus\{0\}$ for each $j$, $\alphabf=(\alpha_1,\ldots,\alpha_\s)\in\Rb^\s$ electric charges such that $\alpha_j>Q$, $\alpha_jR\in\Zb$ for all $j$, and $\mbf=(m_1,\ldots,m_\s)\in\Zb^\s$ magnetic charges. Let $\xbf=(x_1,\ldots,x_\t)$ be distinct marked points on $\paN\Sigma\setminus\paD\Sigma$ and $\etabf=(\eta_1,\ldots,\eta_\t)\in\Rb^\t$ boundary electric charges such that $\eta_j>Q$, $\frac{\eta_j}2R\in\Zb$ for all $j$. Let $g$ be a Neumann extendible admissible metric on $\Sigma$. We have the amplitude operator $\Ac_{\Sigma,g,\vbf,\xbf,\alphabf,\mbf,\etabf,\zetabf}:\Hc^{\otimes\b^{\D,-}}\otimes(\Hc^+)^{\otimes\b^{\MD,-}}\to\Hc^{\otimes\b^{\D,+}}\otimes(\Hc^+)^{\otimes\b^{\MD,+}}$ with $\Hc=L^2(\DT)$, $\Hc^+=L^2(\DTh)$. Finally, let ${\rm w}$ be a Neumann extendible holomorphic vector field on $\Sigma$ (analytic up to the boundary). For each $j=1,\ldots,\b^\D$, let ${\rm w}_j^\D:=(\zeta_j^\D)^*{\rm w}=-\sum_{n\in\Zb}w_{jn}^\D z^{n+1}\pa_z$, where we extend $\zeta_j^\D$ analytically so that ${\rm w}_j^\D$ is holomorphic on an open neighborhood of $\Tb$. Similarly, for each $j=1,\ldots,\b^\MD$, let ${\rm w}_j^\MD:=(\zeta_j^\MD)^*{\rm w}=-\sum_{n\in\Zb}w_{jn}^\MD z^{n+1}\pa_z$, which is symmetric. Then the following identity holds:
\begin{align*}
0 &= \sum_{j=1}^{\b^{\D,+}}\Hbf_{{\rm w}_j^\D}^{(j)}\Ac_\Sigma - \sum_{j=1}^{\b^{\D,-}}\Ac_\Sigma\Hbf_{{\rm w}_{\b^{\D,+}+j}^\D}^{(j)} + \sum_{j=1}^{\b^{\MD,+}}\Hbf_{{\rm w}_j^\MD}^{+,(j)}\Ac_{\Sigma} - \sum_{j=1}^{\b^{\MD,-}}\Ac_{\Sigma}\Hbf_{{\rm w}_{\b^{\MD,+}+j}^\MD}^{+,(j)}\\
&\phantom{={}}+ \left(\sum_{j=1}^{\s}\left(\Delta_{\alpha_j,m_j}(\div_gW)(z_j)+W_\vbf^{(j)}\right)+\sum_{j=1}^{\t}\left(\frac{\Delta_{\eta_j}}{2}(\div_gW)(x_j)+W_\xbf^{(j)}\right)\right)\Ac_{\Sigma} \\
&\phantom{={}}+ c_\L\left(-\sum_{j=1}^{\b^{\D,+}}\frac{\Re w_{j0}^\D}{12}+\sum_{j=\b^{\D,+}+1}^{\b^\D}\frac{\Re w_{j0}^\D}{12}-\sum_{j=1}^{\b^{\MD,+}}\frac{w_{j0}^{\MD}}{24}+\sum_{j=\b^{\MD,+}+1}^{\b^\MD}\frac{w_{j0}^\MD}{24}-\frac{1}{96\pi}\int_\Sigma 2K_g\div_gW\,\d v_g\right)\Ac_{\Sigma},
\end{align*}
where $\Ac_\Sigma:=\Ac_{\Sigma,g,\vbf,\xbf,\alphabf,\mbf,\etabf,\zetabf}$ for short and $\Hbf_{{\rm w}}^{(j)},\Hbf_{\rm w}^{+,(j)},W,W_\vbf^{(j)},W_\xbf^{(j)}$ are defined similarly to \cref{ward_cilt}.
\end{prop}





\appendix
\section{Boundary GMC estimates}\label{Sec_app_GMC}
This section contains the proof of \cref{V_v_estimate} and \cref{Q_v^+_bound}. The main idea of the proof belongs to another paper studying conformal bootstrap of CILT (by C. Guillarmou, R. Rhodes and the authors), which is still in preparation. We put them here for completeness.

The original proof is based on cut-off of harmonics instead of arc average regularization we use here, since for us the flow deformation complicates the framework. Estimates \cref{lemma_lowerB_DGreen,lemma_upper_Green_DOZZ,lem_AI3} are technical and helped by AI. 
\begin{prop}\label{lem_App_1}
Let $N\geq0$. For $F,G\in\Ec(\DT)$, we have
           \begin{equation*}
           \begin{split}
               \sup_{0<\delta<1}|\langle V^{(\delta)}_\v F,G\rangle_{L^2}|\leq {}&C\, (N+Q^2/2)^{\frac{\beta^2}{2}-1} \\
               &\sup_{\theta\in[0,2\pi]}|\Re(e^{-\i\theta}v(e^{\i\theta}))|\sqrt{(N\|F\|_{L^2}^2+\widetilde{\mathcal Q}^0(F,F))(N\|G\|_{L^2}^2+\widetilde{\mathcal Q}^0(G,G))},
           \end{split}
    \end{equation*}
where $C>0$ only depends on $\beta$. In particular, $V_\v$ extends to a bounded operator $\Dc(\wt{\Qc}^0)\to\Dc'(\wt{\Qc}^0)$.
\end{prop}
\begin{proof}
All we need to prove is
\begin{equation}\label{eq_estimate_PVP}
    \|e^{-s\mathbf H^0}V^{(\delta)}_\v e^{-t\mathbf H^0}\|_{L^2\to L^2}\leq e^{\frac{Q^2-R^2k^2}{2}(s+t)}\,C_\beta \sup_{\theta\in[0,2\pi]}|\Re(e^{-\i\theta}v(e^{\i\theta}))|(1-e^{-s})^{-\frac{\beta^2}{4}}(1-e^{-t})^{-\frac{\beta^2}{4}}.
\end{equation}
Then using identity $(\mathbf H^0+\lambda)^{-\frac12}=\frac1{\sqrt{\pi}}\int_0^\infty t^{-\frac12}e^{-\lambda t} e^{-t\mathbf H^0}\,\mathrm{d}t$ when $\lambda$ is sufficiently large and Cauchy-Schwarz inequality would finish the proof. From now on we will only focus on \eqref{eq_estimate_PVP}.
Let $\widetilde{\varphi}^{\text{arc}}_\delta(\theta):=\frac{1}{2\delta}\int_{\theta-\delta}^{\theta+\delta}\widetilde{\varphi}(\theta')\d\theta'$. Denote by
\begin{equation*}
    T_{s,t,\delta}(\theta):=-e^{-s\mathbf H^0}e^{\i\beta\widetilde{\varphi}^{\text{arc}}_\delta(\theta)+\frac{\beta^2}{2}\mathbb E[(\varphi^{\text{arc}}_\delta)^2(\theta)]} e^{-t\mathbf H^0}\Re(e^{-\i\theta}v(e^{\i\theta})), \,\,\,\int_0^{2\pi}T_{s,t,\delta}(\theta)\,\d\theta=e^{-s\mathbf H^0}V^{(\delta)}_\v e^{-t\mathbf H^0}.
\end{equation*}
Now let us focus on $\|(T_{s,t,\delta}(\theta'))^* T_{s,t,\delta}(\theta)\|^{\frac12}_{L^2\to L^2}$. Denote by $\widetilde{\varphi}^k_t(\theta):=(P\widetilde{\varphi}^k+X_{\mathbb D})(e^{-t+\i\theta})$ the restriction of the GFF to a circle with radius $e^{-t}$. Now we take two i.i.d. Dirichlet GFFs $X_\mathbb D^1$ and $X_\mathbb D^2$ and we have $\widetilde{\varphi}^{k,j}_t(\theta):=(P\widetilde{\varphi}^k+X^j_{\mathbb D})(e^{-t+\i\theta}),\ j=1,2$. In fact, let
\begin{equation*}
    X_{g_0}(e^{-t+\i\theta}):=\widetilde{\varphi}^{k,1}_t(\theta),X_{g_0}(e^{t+\i\theta}):=\widetilde{\varphi}^{k,2}_t(\theta),\ t>0,\theta\in[0,2\pi].
\end{equation*}
which is the GFF on the sphere with a DOZZ metric $g_0=\frac {|\d z|^2}{1\vee |z|^4}$. Also define $\widetilde{\varphi}^{k,j}_{s,\delta}(\theta):=\frac{1}{2\delta}\int_{\theta-\delta}^{\theta+\delta}\widetilde{\varphi}^{k,j}_{s}(\theta')\,\d\theta'$. Then
\begin{equation*}
\begin{split}
    &\langle  T_{s,t,\delta}(\theta) F,T_{s,t,\delta}(\theta')G\rangle_{L^2}=e^{{(Q^2-R^2k^2)}(s+t)}\\
    &\int_0^{2\pi R}\Re(e^{-\i\theta}v(e^{\i\theta}))\Re(e^{-\i\theta'}v(e^{\i\theta'}))\mathbb E[e^{\i\beta\widetilde{\varphi}^{k,1}_{s,\delta}(\theta)+\frac{\beta^2}{2}\mathbb E[(\varphi^{\text{arc}}_\delta)^2(\theta)]}F(\widetilde{\varphi}^{k,1}_{s+t})e^{-\i\beta\widetilde{\varphi}^{k,2}_{s,\delta}(\theta')+\frac{\beta^2}{2}\mathbb E[(\varphi^{\text{arc}}_\delta)^2(\theta)]}\overline{G(\widetilde{\varphi}^{k,2}_{s+t})}]\,\d c.
\end{split}
\end{equation*}
Denote by $P(\widetilde{\varphi}^{k,1}_{s+t},\widetilde{\varphi}^{k,2}_{s+t})(z)$ the harmonic extension from $\widetilde{\varphi}^{k,1}_{s+t}(\theta)$ on $|z|=e^{-s-t}$ and $\widetilde{\varphi}^{k,2}_{s+t}(\theta)$ on $|z|=e^{s+t}$, to the annulus $e^{-s-t}\leq |z|\leq e^{s+t}$. By the domain Markov property, we can write $X_{g_0}(z)$ as $X_{g_0}=P(\widetilde{\varphi}^{k,1}_{s+t},\widetilde{\varphi}^{k,2}_{s+t})+X_{D,s+t}$ when $e^{-s-t}\leq |z|\leq e^{s+t}$, where $X_{D,s+t}$ is a Dirichlet GFF on $e^{-s-t}\leq |z|\leq e^{s+t}$ that is independent of $\widetilde{\varphi}^{k,1}_{s+t},\widetilde{\varphi}^{k,2}_{s+t}$.

Conditioning on $\widetilde{\varphi}^{k,1}_{s+t},\widetilde{\varphi}^{k,2}_{s+t}$, we have
\begin{equation}\label{eq_condi_robin}
    \begin{split}
        &\left|\mathbb E[e^{\i\beta\widetilde{\varphi}^{k,1}_{s,\delta}(\theta)+\frac{\beta^2}{2}\mathbb E[(\varphi^{\text{arc}}_\delta)^2(\theta)]}F(\widetilde{\varphi}^{k,1}_{s+t})e^{-\i\beta\widetilde{\varphi}^{k,2}_{s,\delta}(\theta')+\frac{\beta^2}{2}\mathbb E[(\varphi^{\text{arc}}_\delta)^2(\theta)]}\overline{G(\widetilde{\varphi}^{k,2}_{s+t})}|\widetilde{\varphi}^{k,1}_{s+t},\widetilde{\varphi}^{k,2}_{s+t}]\right|\\
        \leq{}&\left|F(\widetilde{\varphi}^{k,1}_{s+t})\overline{G(\widetilde{\varphi}^{k,2}_{s+t})}\right|e^{\beta^2\mathbb E[(\varphi^{\text{arc}}_\delta)^2(\theta)]-\frac{\beta^2}{2}\mathbb E[(\frac{1}{2\delta}\int_{\theta-\delta}^{\theta+\delta}X_{D,s+t}(e^{-s+\i\theta''})\,\d\theta''-\frac{1}{2\delta}\int_{\theta'-\delta}^{\theta'+\delta}X_{D,s+t}(e^{s+\i\theta''})\,\d\theta'')^2]}.
    \end{split}
\end{equation}
Denote by $G_{D,s+t}(\cdot,\cdot)$ the covariance kernel of $X_{D,s+t}$. To give a bound for \eqref{eq_condi_robin}, we check that
\begin{equation*}
\begin{split}
    &-\frac{\beta^2}{2}\mathbb E\left[\left(\frac{1}{2\delta}\int_{\theta-\delta}^{\theta+\delta}X_{D,s+t}(e^{-s+\i\theta''})\,\d\theta''-\frac{1}{2\delta}\int_{\theta'-\delta}^{\theta'+\delta}X_{D,s+t}(e^{s+\i\theta''})\,\d\theta''\right)^2\right]\\
    ={}&-\frac{\beta^2}{8\delta^2}\Bigg(\int_{\theta-\delta}^{\theta+\delta}\int_{\theta-\delta}^{\theta+\delta}G_{D,s+t}(e^{-s+\i\theta_1},e^{-s+\i\theta_2})\,\d\theta_1\d\theta_2+\int_{\theta'-\delta}^{\theta'+\delta}\int_{\theta'-\delta}^{\theta'+\delta}G_{D,s+t}(e^{s+\i\theta_1},e^{s+\i\theta_2})\,\d\theta_1\d\theta_2\\
    &\ \ \ \ \ \ \ \ \ \ \ \ \ -2\int_{\theta'-\delta}^{\theta'+\delta}\int_{\theta-\delta}^{\theta+\delta}G_{D,s+t}(e^{-s+\i\theta_1},e^{s+\i\theta_2})\,\d\theta_1\d\theta_2\Bigg),
\end{split}
\end{equation*}
\begin{lemma}\label{lemma_lowerB_DGreen}
    For $(2\pi)^{-1}G_{D,s+t}(\cdot,\cdot)$ the Dirichlet Green's function on $e^{-s-t}\leq |z|\leq e^{s+t}$, uniformly in $s,t\geq0,\delta<1,\theta\in[0,2\pi]$,
    \begin{equation*}
        \frac{1}{4\delta^2}\int_{\theta-\delta}^{\theta+\delta}\int_{\theta-\delta}^{\theta+\delta}G_{D,s+t}(e^{-s+\i\theta_1},e^{-s+\i\theta_2})\,\d\theta_1\d\theta_2\geq\mathbb E[(\varphi^{\text{arc}}_\delta)^2(\theta)]+\log|1-e^{-t}|+C,
    \end{equation*}
    where $C$ is independent of $s,t>0,\theta\in[0,2\pi],0<\delta<1$.
\end{lemma}
\begin{proof}
    Consider the Dirichlet Green's function $(2\pi)^{-1}G_{\theta,s,t}$ of a disk centered at $a=e^{-s+\i\theta}$ with radius $r=e^{-s}-e^{-s-t}$
\begin{equation*}
    G_{\theta,s,t}(x,y)=\log\left|\frac{r^2-(x-a)\overline{(y-a)}}{r(x-y)}\right|.
\end{equation*}
To study the integral, we need
\begin{equation*}
     G_{\theta,s,t}(e^{-s+\i\theta_1},e^{-s+\i\theta_2})=\log\left|\frac{(e^{-\i\theta_2}-e^{-\i\theta})(e^{\i\theta_1}-e^{\i\theta})-|1-e^{-t}|^2}{(1-e^{-t})(e^{\i\theta_1}-e^{\i\theta_2})}\right|.
\end{equation*}
Denote by $\delta_t:=\delta\wedge2\arcsin(\frac{1-e^{-t}}{2})$. Then
\begin{equation}\label{eq_bound_DGreen}
\begin{split}
     \frac{1}{4\delta^2}\int_{\theta-\delta}^{\theta+\delta}\int_{\theta-\delta}^{\theta+\delta}G_{D,s+t}(e^{-s+\i\theta_1},e^{-s+\i\theta_2})\,\d\theta_1\d\theta_2\geq{}&\frac{1}{4\delta^2}\int_{\theta-\delta_t}^{\theta+\delta_t}\int_{\theta-\delta_t}^{\theta+\delta_t}\log\left|\frac{(e^{-\i\theta_2}-e^{-\i\theta})(e^{\i\theta_1}-e^{\i\theta})}{|1-e^{-t}|^2}-1\right|\,\d\theta_1\d\theta_2\\
     +&\frac{\delta^2_t}{\delta^2}\log|1-e^{-t}|-\frac{1}{4\delta^2}\int_{\theta-\delta_t}^{\theta+\delta_t}\int_{\theta-\delta_t}^{\theta+\delta_t}\log|e^{\i\theta_1}-e^{\i\theta_2}|\,\d\theta_1\d\theta_2\\
     =:&I_1+\frac{\delta^2_t}{\delta^2}\log|1-e^{-t}|+I_2.
\end{split}
\end{equation}
We first compare the singular part $I_2$. That is
\begin{equation}\label{eq_bound_DGI2}
    I_2+\frac{1}{4\delta^2}\int_{\theta-\delta}^{\theta+\delta}\int_{\theta-\delta}^{\theta+\delta}\log|e^{\i\theta_1}-e^{\i\theta_2}|\,\d\theta_1\d\theta_2
    \geq\frac{\delta_t^2}{\delta^2}(-\log\delta_t-C)+\log\delta-C\geq-\frac{\delta_t^2}{\delta^2}\log\frac{\delta_t}\delta+(1-\frac{\delta_t^2}{\delta^2})\log\delta+C,
\end{equation}
where $C$ is independent of $\delta$ since $|\mathbb E[(\varphi^{\text{arc}}_\delta)^2(\theta)]+\log\delta|$ is bounded uniformly in $\delta$ small. 

Next, since $e^{\i\theta_1}-1=2\i e^{\i\frac{\theta_1}2}\sin(\frac{\theta_1}2)$, we have
\begin{equation*}
    \left|\frac{(e^{-\i\theta_2}-1)(e^{\i\theta_1}-1)}{|1-e^{-t}|^2}-1\right|\geq\left|\cos\left(\frac{\theta_1-\theta_2}{2}\right)\frac{4\sin\left(\frac{\theta_1}2\right)\sin\left(\frac{\theta_2}2\right)}{|1-e^{-t}|^2}-1\right|.
\end{equation*}
Therefore
\begin{equation*}
\begin{split}
     I_1={}&\frac{1}{4\delta^2}\int_{-\delta_t}^{\delta_t}\int_{-\delta_t}^{\delta_t}\log\left|\frac{(e^{-\i\theta_2}-1)(e^{\i\theta_1}-1)}{|1-e^{-t}|^2}-1\right|\,\d\theta_1\d\theta_2\\
     \geq{}&\frac{1}{2\delta^2}\int_{0}^{\delta_t}\int_{0}^{\delta_t}\log\left|\cos\left(\frac{\theta_1-\theta_2}{2}\right)\frac{4\sin\left(\frac{\theta_1}2\right)\sin\left(\frac{\theta_2}2\right)}{|1-e^{-t}|^2}-1\right|\,\d\theta_1\d\theta_2+C\frac{\delta_t^2}{2\delta^2},
\end{split}
\end{equation*}
where we can replace $\int_{0}^{\delta_t}\int_{-\delta_t}^{0}\cdots\d\theta_1\d\theta_2$ with $C\delta_t^2$ because $\cos\left(\frac{\theta_1-\theta_2}{2}\right)\geq-\frac12$ (since $\delta_t\leq2\arcsin(\frac{1-e^{-t}}{2})\leq\frac{2\pi}{3}$). Now $I_1$ is further bounded by
\begin{equation*}
    I_1\geq\frac{1}{2\delta^2}\int_{0}^{\delta_t}\int_{0}^{\delta_t}\log\left|\frac{4\sin\left(\frac{\theta_1}2\right)\sin\left(\frac{\theta_2}2\right)}{|1-e^{-t}|^2}-1\right|\,\d\theta_1\d\theta_2.
\end{equation*}
Let $u=2\frac{\sin\left(\frac{\theta_1}2\right)}{1-e^{-t}},v=2\frac{\sin\left(\frac{\theta_2}2\right)}{1-e^{-t}}$, and
\begin{equation}\label{eq_DGI1}
\begin{split}
     I_1\geq{}&\frac{(1-e^{-t})^2}{2\delta^2}\int_0^{2\frac{\sin\left(\frac{\delta_t}2\right)}{1-e^{-t}}\wedge1}\int_0^{2\frac{\sin\left(\frac{\delta_t}2\right)}{1-e^{-t}}\wedge1}\log|1-uv|\,\frac{\d u\d v}{\sqrt{1-\left(\frac{1-e^{-t}}{2}u\right)^2}\sqrt{1-\left(\frac{1-e^{-t}}{2}v\right)^2}}\\
     \geq{}&\frac23\int_0^{1}\int_0^{1}\log|1-uv|\,{\d u\d v}\geq C.
\end{split}
\end{equation}
Finally with \eqref{eq_bound_DGI2} and \eqref{eq_DGI1} we get back to \eqref{eq_bound_DGreen}, which is greater than
\begin{equation*}
    I_1+\frac{\delta^2_t}{\delta^2}\log|1-e^{-t}|+I_2\geq \mathbb E[(\varphi^{\text{arc}}_\delta)^2(\theta)]+\log|1-e^{-t}|+(1-\frac{\delta_t^2}{\delta^2})\log\frac{\delta}{|1-e^{-t}|}-\frac{\delta_t^2}{\delta^2}\log\frac{\delta_t}{\delta}+C,
\end{equation*}
where $(1-\frac{\delta_t^2}{\delta^2})\log\frac{\delta}{|1-e^{-t}|}$ is positive (when $\delta>\delta_t$, use $\arcsin x\geq x$).
\end{proof}
\begin{lemma}\label{lemma_upper_Green_DOZZ}
    \begin{equation*}
        -\frac{1}{4\delta^2}\int_{-\delta}^{\delta}\int_{-\delta}^{\delta}\log|e^{-2s+\i(\theta-\theta'+\theta_1-\theta_2)}-1|\,\d\theta_1\d\theta_2\leq-\log|e^{\i\theta}-e^{\i\theta'}|+C,
    \end{equation*}
    with $C$ independent of $\theta,\theta'\in[0,2\pi],s>0,0<\delta<1$.
\end{lemma}
\begin{proof}
Since for $t\in\mathbb R$, $|e^{\i t}-1|\leq 2|e^{-2s+\i t}-1|$, we get
    \begin{equation*}
         -\frac{1}{4\delta^2}\int_{-\delta}^{\delta}\int_{-\delta}^{\delta}\log{|e^{-2s+\i(\theta-\theta'+\theta_1-\theta_2)}-1|}\,\d\theta_1\d\theta_2\leq  -\frac{1}{4\delta^2}\int_{-\delta}^{\delta}\int_{-\delta}^{\delta}\log{|e^{\i(\theta-\theta'+\theta_1-\theta_2)}-1|}\,\d\theta_1\d\theta_2+C.
    \end{equation*}
    Let $u=\theta_1-\theta_2,v=\theta_1+\theta_2$. Then
    \begin{equation*}
         -\frac{1}{4\delta^2}\int_{-\delta}^{\delta}\int_{-\delta}^{\delta}\log{|e^{\i(\theta-\theta'+\theta_1-\theta_2)}-1|}\,\d\theta_1\d\theta_2= -\frac{1}{4\delta^2}\int_{-2\delta}^{2\delta}(2\delta-|u|)\log{|e^{\i(\theta-\theta'+u)}-1|}\,\d u.
    \end{equation*}
We only need to deal with $\delta$ sufficiently small. Note that $|e^{\i(\theta-\theta'+u)}-1|^2=4\sin^2(\frac{\theta-\theta'+u}{2})$. Define $|t|_{2\pi}:=\dist(t,2\pi\mathbb Z)\in[0,\pi]$, then by $\frac{2t}{\pi}\leq\sin t\leq t$ on $[0,\frac\pi2]$, we have
\begin{equation*}
   \log\frac{1}{|\theta-\theta'+u|_{2\pi}} \leq -\log{|e^{\i(\theta-\theta'+u)}-1|}\leq \log\frac{\pi}{2|\theta-\theta'+u|_{2\pi}}.
\end{equation*}
Therefore, from now on we only need to compare $|\theta-\theta'+u|_{2\pi}$ with $|\theta-\theta'|_{2\pi}$. When $|\theta-\theta'|_{2\pi}\geq4\delta$, $|\theta-\theta'+u|_{2\pi}\geq|\theta-\theta'|_{2\pi}/2$ provides an upper bound. And when $|\theta-\theta'|_{2\pi}\leq4\delta$, we claim that
\begin{equation*}
     -\frac{1}{4\delta^2}\int_{-2\delta}^{2\delta}(2\delta-|u|)\log{\left|\sin\left(\frac{\theta-\theta'+u}{2}\right)\right|}\,\d u\leq-\log\delta+C\leq-\log|\theta-\theta'|_{2\pi}+C.
\end{equation*}
We only prove the first inequality as the second one is straightforward. First, by Fubini's theorem,
\begin{equation*}
   - \int_{-2\delta}^{2\delta}\frac{2\delta-|u|}{4\delta^2}\log{\left|\sin\left(\frac{\theta-\theta'+u}{2}\right)\right|}\,\d u=\int_0^{\infty}\int_{-2\delta}^{2\delta}\textbf{1}_{\{-\log{\left|\sin\left(\frac{\theta-\theta'+u}{2}\right)\right|}>\tau\}}\frac{2\delta-|u|}{4\delta^2}\,\d u\,\d\tau.
\end{equation*}
Assume $|\theta-\theta'|=|\theta-\theta'|_{2\pi}$ (or choose a representative modulo $2\pi$). Denote by $A(\tau):=\{u:-\log{\left|\sin\left(\frac{\theta-\theta'+u}{2}\right)\right|}>\tau,|\theta-\theta'+u|
\leq\pi\}$, then (since $\delta<\frac\pi6$)
\begin{equation*}    \int_{-2\delta}^{2\delta}\textbf{1}_{A(\tau)}\frac{2\delta-|u|}{4\delta^2}\,\d u\leq\frac1{2\delta}\min(\int_{|\theta-\theta'+u|\leq\delta}\,\d u,\int_{A(\tau)}\,\d u)=\frac{1}{2\delta}\int_{A(\tau)\cap\{|\theta-\theta'+u|\leq\delta\}}\,\d u.
\end{equation*}
Hence we conclude that
\begin{equation*}
\begin{split}
    -\int_{-2\delta}^{2\delta}\frac{2\delta-|u|}{4\delta^2}\log{\left|\sin\left(\frac{\theta-\theta'+u}{2}\right)\right|}\,\d u\leq{}&-\frac{1}{2\delta}\int_{\{|\theta-\theta'+u|\leq\delta\}}\log{\left|\sin\left(\frac{\theta-\theta'+u}{2}\right)\right|}\,\d u\\
    \leq{}&\frac{1}{\delta}\int_0^{\delta}\log\frac{\pi}{u}\,\d u\leq-\log\delta+C.\qedhere
\end{split}
\end{equation*}
\end{proof}
Using \cref{lemma_lowerB_DGreen}, \cref{lemma_upper_Green_DOZZ}, the invariance of $G_{D,s+t}$ under $z\to1/z$, and that $G_{D,s+t}$ is bounded by the covariance of $X_{g_0}$ $G_{g_0}(x,y)+\log2=\log\frac{(|x|\vee1)(|y|\vee1)}{|x-y|} +\log2$, we find \eqref{eq_condi_robin} is bounded by
\begin{equation*}
C\left|F(\widetilde{\varphi}^{k,1}_{s+t})\overline{G(\widetilde{\varphi}^{k,2}_{s+t})}\right|(1-e^{-t})^{-\beta^2}|e^{\i\theta}-e^{\i\theta'}|^{-\beta^2}.
\end{equation*}
This implies the operator norm of $(T_{s,t,\delta}(\theta'))^* T_{s,t,\delta}(\theta)$ is controlled by $Ce^{(Q^2-R^2k^2)(s+t)}\sup_{\theta\in[0,2\pi]}|\Re(e^{-\i\theta}v(e^{\i\theta}))|^2(1-e^{-t})^{-\beta^2}|e^{\i\theta}-e^{\i\theta'}|^{-\beta^2}$ by using Cauchy--Schwarz inequality and that $\widetilde{\varphi}^{k,1}_{s+t}$ has the same distribution as $\widetilde{\varphi}^{k}$. Therefore,
\begin{equation*}
\begin{split}
    M_1:={}&\sup_{\theta'\in[0,2\pi]}\int_{0}^{2\pi}\|(T_{s,t,\delta}(\theta'))^* T_{s,t,\delta}(\theta)\|^{\frac12}_{L^2\to L^2}\,\d\theta\\
    \leq{}& Ce^{\frac{Q^2-R^2k^2}{2}(s+t)}\sup_{\theta\in[0,2\pi]}|\Re(e^{-\i\theta}v(e^{\i\theta}))|(1-e^{-t})^{-\frac{\beta^2}2}\int_{0}^{2\pi}|e^{\i\theta}-1|^{-\frac{\beta^2}{2}}\d\theta<\infty.
\end{split}
\end{equation*}
Similarly, we have
\begin{equation*}
\begin{split}
    M_2:={}&\sup_{\theta'\in[0,2\pi]}\int_{0}^{2\pi}\|T_{s,t,\delta}(\theta') (T_{s,t,\delta}(\theta))^*\|^{\frac12}_{L^2\to L^2}\,\d\theta\\
    \leq {}&Ce^{\frac{Q^2-R^2k^2}{2}(s+t)}\sup_{\theta\in[0,2\pi]}|\Re(e^{-\i\theta}v(e^{\i\theta}))|(1-e^{-s})^{-\frac{\beta^2}2}\int_{0}^{2\pi}|e^{\i\theta}-1|^{-\frac{\beta^2}{2}}\d\theta<\infty
\end{split}
\end{equation*}
By using the continuous Cotlar-Stein lemma (see \cite{CALDERON1971}), we have
\begin{equation*}
\begin{split}
    &\left\| \int_0^{2\pi}T_{s,t,\delta}(\theta')\,\d\theta'\,\right\|_{L^2\to L^2}\leq\sqrt{M_1M_2}\\
    \leq{}&Ce^{\frac{Q^2-R^2k^2}{2}(s+t)}\sup_{\theta\in[0,2\pi]}|\Re(e^{-\i\theta}v(e^{\i\theta}))|(1-e^{-t})^{-\frac{\beta^2}4}(1-e^{-s})^{-\frac{\beta^2}4}\int_{0}^{2\pi}|e^{\i\theta}-1|^{-\frac{\beta^2}{2}}\d\theta.
\end{split}
\end{equation*}
This brings us back to \eqref{eq_estimate_PVP} and thus we have finished the proof.
\end{proof}
\begin{prop}\label{lem_App_2}
Let $N\geq0$. For $F,G\in\Ec(\DTh)$, we have
\begin{equation*}
|\langle V_\v^\mubf  F,G\rangle_{L^2}|\leq C\, (N+Q^2/4)^{\frac{\beta^2}{2}-1} \sup_{\theta\in[0,2\pi]}|\Re(e^{-\i\theta}v(e^{\i\theta}))|\,\sqrt{(N\|F\|_{L^2}^2+\widetilde{\mathcal Q}^{+,0}(F,F))(N\|G\|_{L^2}^2+\widetilde{\mathcal Q}^{+,0}(G,G))},
\end{equation*}
where $C>0$ only depends on $\beta$ and $\mubf$. In particular, $V_\v^\mubf$ extends to a bounded operator $\Dc(\wt{\Qc}^{+,0})\to\Dc'(\wt{\Qc}^{+,0})$.
\end{prop}
\begin{proof}
We simply follow \cref{V_v_estimate}, and prove
\begin{equation}\label{eq_estimate_PVP_bdy}
    \|e^{-s\mathbf H^{+,0}}V_\v^{\cap,(\delta)} e^{-t\mathbf H^{+,0}}\|_{L^2\to L^2}\leq e^{\frac{Q^2}{4}(s+t)}\,C_\beta \sup_{\theta\in[0,2\pi]}|\Re(e^{-\i\theta}v(e^{\i\theta}))|(1-e^{-s})^{-\frac{\beta^2}{4}}(1-e^{-t})^{-\frac{\beta^2}{4}}.
\end{equation}
Define
\begin{equation*}
    T_{s,t,\delta}^+(\theta):=-e^{-s\mathbf H^{+,0}}e^{\i\beta\widetilde{\varphi}^{\text{h,arc}}_\delta(\theta)+\frac{\beta^2}{2}\mathbb E[(\varphi^{\text{h,arc}}_\delta)^2(\theta)]} e^{-t\mathbf H^{+,0}}\Re(e^{-\i\theta}v(e^{\i\theta}))(2|\!\sin\theta|)^{\frac{\beta^2}{2}}.
\end{equation*}
The study of $\langle  T_{s,t,\delta}^+(\theta')F,T_{s,t,\delta}^+(\theta)G \rangle_{L^2}$ again boils down to the estimate of Green's function on a thin annulus. Specifically, denote by $(2\pi)^{-1}G_{D,s+t}$ the Dirichlet Green's function on the annulus $e^{-s-t}\leq|z|\leq e^{s+t}$. Then we have the mixed Green's function $(2\pi)^{-1}G_{D,s+t}^+(z,w):=(2\pi)^{-1}(G_{D,s+t}(z,w)+G_{D,s+t}(z,\bar{w}))$ on the half annulus $e^{-s-t}\leq|z|\leq e^{s+t},\Im z\geq0$. Following the proof of \cref{V_v_estimate}, we need to provide a $\delta$-uniform bound in terms of $s,t,\theta,\theta'$
\begin{equation*}
\begin{split}
    &I_{\delta,s,t}(\theta,\theta'):=-\frac{\beta^2}{8\delta^2}\Bigg(\int_{\theta-\delta}^{\theta+\delta}\int_{\theta-\delta}^{\theta+\delta}G_{D,s+t}^+(e^{-s+\i\theta_1},e^{-s+\i\theta_2})\,\d\theta_1\d\theta_2+\int_{\theta'-\delta}^{\theta'+\delta}\int_{\theta'-\delta}^{\theta'+\delta}G_{D,s+t}^+(e^{s+\i\theta_1},e^{s+\i\theta_2})\,\d\theta_1\d\theta_2\\
    &-2\int_{\theta'-\delta}^{\theta'+\delta}\int_{\theta-\delta}^{\theta+\delta}G_{D,s+t}^+(e^{-s+\i\theta_1},e^{s+\i\theta_2})\,\d\theta_1\d\theta_2\Bigg)+\frac{\beta^2}{2}(\mathbb E[(\varphi^{\text{h,arc}}_\delta)^2(\theta)]+\mathbb E[(\varphi^{\text{h,arc}}_\delta)^2(\theta')]+\log|\sin\theta|+\log|\sin\theta'|)
\end{split}
\end{equation*}
Since $\theta,\theta'\in[0,\pi]$, we have
\begin{equation*}
\begin{split}
     \int_{-\delta}^{\delta}\int_{-\delta}^{\delta}G_{D,s+t}(e^{-s+\i(\theta+\theta_1)},e^{-s-\i(\theta+\theta_2)})&+G_{D,s+t}(e^{s+\i(\theta'+\theta_1)},e^{s-\i(\theta'+\theta_2)})\\
     &-2G_{D,s+t}(e^{-s+\i(\theta+\theta_1)},e^{s-\i(\theta'+\theta_2)})\,\d\theta_1\d\theta_2\geq0,
\end{split}
\end{equation*}
which is a consequence of the following lemma.
\begin{lemma}\label{lem_AI3}
    Suppose $(2\pi)^{-1}G$ is the Dirichlet Green's function on the annulus $\mathbb A_{q}:=q<|z|<1$. Then when $z_1,\cdots,z_N\in\mathbb A_q\cap\mathbb H$, for $c_1,\cdots,c_N\in\mathbb R$
    \begin{equation*}
       \sum_{j,l}c_jc_lG(z_j,\bar z_l)\geq 0,
    \end{equation*}
    i.e. $G(z,\bar w)$ is positive definite.
\end{lemma}
\begin{proof}
For $z,w\in\mathbb A_q\cap\mathbb H$, define
\begin{equation*}
    H:=\sum_jc_jG(\cdot,\bar z_j),\ \ \ \check H:=\sum_jc_jG(\cdot,  z_j).
\end{equation*}
Using Green's first identity on $\mathbb A_q\cap\mathbb H$, we have
\begin{equation*}
    \int_{\mathbb A_q\cap\mathbb H}|\nabla H|^2\,\d x=\int_{\mathbb A_q\cap\mathbb R}H\partial_n H\,\d x,
\end{equation*}
with $\partial_n H$ the normal derivative.

Now let $B(z_j,\epsilon)$ be the closed ball centered at $z_j$ with radius $\epsilon$. Set $A_\epsilon^+:=(\mathbb A_q\cap\mathbb H)\backslash (\cup_j B(z_j,\epsilon))$, where $\epsilon$ is small enough so that $B(z_j,\epsilon)\cap B(z_l,\epsilon)=\emptyset$ whenever $j\not=l$. By Green's second identity on $A_\epsilon^+$, we get
\begin{equation*}
    0=\int_{A_\epsilon^+}H\Delta \check H-\check H\Delta  H\,\d x=\int_{\partial A_\epsilon^+}H\partial_n \check H-\check H\partial_n  H\,\d x.
\end{equation*}
Since $H,\check H$ vanish on $\partial \mathbb A_q$, we study the integration on $\mathbb A_q\cap\mathbb R$ and $\partial B(z_j,\epsilon)$.

On $\mathbb A_q\cap\mathbb R$, $H=\check H$ and $\partial_n H=-\partial_n \check H$. On $\partial B(z_j,\epsilon)$, by the properties of Green's function, we have
\begin{equation*}
    \partial_n \check H=\frac{c_j}{\epsilon}+O(1),\ \ \ \check H=c_j\log\frac{1}{\epsilon}+O(1).
\end{equation*}
Hence
\begin{equation*}
   0= \int_{\partial A_\epsilon^+}H\partial_n \check H-\check H\partial_n  H\,\d x=-2\int_{\mathbb A_q\cap\mathbb R}H\partial_n H\,\d x+2\pi \sum_j c_j H(z_j)+O(\epsilon\log\frac{1}{\epsilon}).
\end{equation*}
Let $\epsilon\to0$, and we conclude that
\begin{equation*}
    \frac{1}{\pi}\int_{\mathbb A_q\cap\mathbb H}|\nabla H|^2\,\d x=\sum_j c_j H(z_j)=\sum_{j,l}c_jc_lG(z_j,\bar z_l)\geq0.\qedhere
\end{equation*}
\end{proof}
Therefore we can throw away the boundary-sensitive terms $G_{D,s+t}(e^{-s+\i\theta},e^{-s-\i\theta}),G_{D,s+t}(e^{s+\i\theta},e^{s-\i\theta})$ and 

$G_{D,s+t}(e^{-s+\i\theta},e^{s-\i\theta})$ in $I_{\delta,s,t}(\theta,\theta')$, and use \cref{lemma_lowerB_DGreen} again to get
\begin{equation*}
\begin{split}
    &I_{\delta,s,t}(\theta,\theta')\leq -\beta^2\log|1-e^{-t}|-\beta^2\log|e^{\i\theta}-e^{\i\theta'}|+C,
\end{split}
\end{equation*}
with $C$ independent of $\delta,s,t,\theta,\theta'$. Note that the $\theta$ dependence in $\mathbb E[(\varphi^{\text{h,arc}}_\delta)^2(\theta)]$ is canceled by $\log|\sin\theta|$. Now we can use Cotlar-Stein lemma \cite[pp. 374-375]{CALDERON1971}  for $\int_{\delta}^{\pi-\delta}T_{s,t,\delta}^+(\theta)\d\theta$ and repeat what is done in \cref{V_v_estimate}.

Next, we deal with $V^{+,{\rm R},(\delta)}$; $V^{+,{\rm L},(\delta)}$ is handled in a similar way. We claim that
\begin{equation}\label{eq_claim_VR}
    |\langle e^{\i\frac\beta2\widetilde{\varphi}^{\text{h,arc}}_\delta(0)+\frac{\beta^2}{8}\mathbb E[\varphi^{\h,\text{arc}}_\delta(0)^2]} e^{-t\mathbf H^{+,0}}F,e^{-s\mathbf H^{+,0}} G\rangle_{L^2}|\leq C e^{\frac{Q^2}{4}(s+t)}(1-e^{-s})^{-\frac{\beta^2}{4}} (1-e^{-t})^{-\frac{\beta^2}{4}}\|F\|_{L^2}\|G\|_{L^2}.
\end{equation}
Denote by $\widetilde{\varphi}^\h_t(\theta):=(P\widetilde{\varphi}^\h+X_{\mathbb D^+})(e^{-t+\i\theta})$ the restriction of the GFF to a circle with radius $e^{-t}$, where $X_{\mathbb D^+}$ is the mixed GFF on $\mathbb D^+:=\mathbb D\cap\mathbb H$. Now we take two i.i.d. Dirichlet GFF $X_{\mathbb D^+}^1$ and $X_{\mathbb D^+}^2$ and we have $\widetilde{\varphi}^{\h,j}_t(\theta):=(P\widetilde{\varphi}^\h+X^j_{\mathbb D^+})(e^{-t+\i\theta}),\ j=1,2$.
\begin{equation*}
\begin{split}
  \langle e^{\i\frac\beta2\widetilde{\varphi}^{\text{h,arc}}_\delta(0)+\frac{\beta^2}{8}\mathbb E[\varphi^{\h,\text{arc}}_\delta(0)^2]} e^{-t\mathbf H^{+,0}}F,e^{-s\mathbf H^{+,0}} G\rangle_{L^2}=\int_0^{2\pi R}e^{\frac{Q^2}{4}(s+t)}\mathbb E[e^{\i\frac\beta2\widetilde{\varphi}^{\text{h,arc}}_\delta(0)+\frac{\beta^2}{8}\mathbb E[\varphi^{\h,\text{arc}}_\delta(0)^2]}F(\widetilde{\varphi}^{\h,1}_{t})\overline{G(\widetilde{\varphi}^{\h,2}_{s})}]\,\d c.
\end{split}
\end{equation*}
By studying the expectation $e^{\i\frac\beta2\widetilde{\varphi}^{\text{h,arc}}_\delta(0)}$ conditioned on $\widetilde{\varphi}^{\h,1}_{t},\widetilde{\varphi}^{\h,2}_{s}$, we get a bound 
\begin{equation*}
    |\langle e^{\i\frac\beta2\widetilde{\varphi}^{\text{h,arc}}_\delta(0)+\frac{\beta^2}{8}\mathbb E[\varphi^{\h,\text{arc}}_\delta(0)^2]} e^{-t\mathbf H^{+,0}}F,e^{-s\mathbf H^{+,0}} G\rangle_{L^2}|\leq e^{\frac{Q^2}{4}(s+t)}e^{\frac{\beta^2}{8}(\mathbb E[\varphi^{\h,\text{arc}}_\delta(0)^2]-\frac{1}{\delta^2}\int_{0}^\delta\int_{0}^\delta G_{D,s+t}^+(e^{\i\theta_1},e^{\i\theta_2})\,\d\theta_1\d\theta_2)}\|F\|_{L^2}\|G\|_{L^2},
\end{equation*}
where we can bound the Robin mass using \cref{lemma_lowerB_DGreen}
\begin{equation*}
    \mathbb E[\varphi^{\h,\text{arc}}_\delta(0)^2]-\frac{1}{\delta^2}\int_{0}^\delta\int_{0}^\delta G_{D,s\wedge t}^+(e^{\i\theta_1},e^{\i\theta_2})\,\d\theta_1\d\theta_2\leq-2\min(\log|1-e^{-s}|,\log|1-e^{-t}|)+C.
\end{equation*}
Finally the claim \eqref{eq_claim_VR} is true and we have obtained the estimate for $V_\v^\mubf$ that we want.
\end{proof}


\bibliographystyle{plain} 
\bibliography{virasoro}
\end{document}